\documentclass[reqno]{amsart}
\usepackage[top=25mm,bottom=25mm,left=25mm,right=25mm]{geometry}
\usepackage{amssymb, amsmath}
\usepackage{natbib}
\usepackage{lscape,comment}
\usepackage{xcolor}
\usepackage{dsfont}
\usepackage{xfrac}
\usepackage{mathrsfs, mathtools}
\usepackage{bm}

\usepackage[pdftex,plainpages=false,colorlinks,hyperindex,bookmarksopen,linkcolor=red,citecolor=blue,urlcolor=blue]{hyperref}

\DeclareMathAlphabet{\mathpzc}{OT1}{pzc}{m}{it}

\bibpunct{[}{]}{;}{n}{,}{,}

\newtheorem{thm}{Theorem}[section]

\newtheorem{rmk}[thm]{Remark}
\newtheorem{prop}[thm]{Proposition}
\newtheorem{lem}[thm]{Lemma}

\newtheorem{coro}[thm]{Corollary}

\numberwithin{equation}{section}

\def \l { \left( }
\def \r {\right) }
\def \ll { \left\lbrace }
\def \rr { \right\rbrace }

\DeclareMathOperator{\E}{\mathds{E}}
\DeclareMathOperator{\dP}{\mathds{P}}

\DeclareMathOperator{\N}{\mathbb{N}}
\DeclareMathOperator{\R}{\mathbb{R}}
\DeclareMathOperator{\cF}{\mathcal{F}}
\DeclareMathOperator{\cG}{\mathcal{G}}
\DeclareMathOperator{\cT}{\mathcal{T}}

\DeclareMathOperator{\cM}{\mathcal{M}}

\DeclareMathOperator{\cR}{\mathcal{R}}

\DeclareMathOperator{\cB}{\mathcal{B}}

\DeclareMathOperator*{\argmax}{argmax}

\newcommand{\pdsup}[3]{\frac{\partial^{#3} #1}{\partial #2^{#3}}}

\newcommand{\der}[2]{\frac{d #1}{d #2}}
\newcommand{\dersup}[3]{\frac{d^{#3} #1}{d #2^{#3}}}

\newcommand{\Norm}[2]{\left\Vert #1 \right\Vert_{#2}}

\newcommand{\btrev}[1]{{\textcolor{black}{#1}}}

\allowdisplaybreaks

\begin{document}

	\title[]{Non-local heat equations with moving boundary}
	\author[]{Giacome Ascione}
	\author[]{Pierre Patie}
	\author[]{Bruno Toaldo}
	\address[]{Scuola Superiore Meridionale - Universit\`{a} degli Studi di Napoli ``Federico II'', Largo S. Marcellino 10 - 80138, Napoli (Italy)}
	\email[]{g.ascione@ssmeridionale.it}
	\address[]{Cornell University, 219 Rhodes Hall - 14852, Ithaca, NY (USA)}
\email[]{ppatie@cornell.edu}
	\address[]{Dipartimento di Matematica ``Giuseppe Peano'' - Universit\`{a} degli Studi di Torino, Via Carlo Alberto 10 - 10123, Torino (Italy)}
	\email[]{bruno.toaldo@unito.it}
	\date{\today}

		\begin{abstract}
In this paper we consider non-local (in time) heat equations on time-increasing parabolic sets whose boundary is determined by a suitable curve. We provide a notion of solution for these equations and we study well-posedness under Dirichlet conditions outside the domain. A maximum principle is proved and used to derive uniqueness and continuity with respect to the initial datum of the solutions of the Dirichlet problem. Existence is proved by showing a stochastic representation based on the delayed Brownian motion killed on the boundary. Several related distributional properties of the delayed Brownian motion and its crossing probabilities are also obtained. The asymptotic behaviour of the mean square displacement of the process is determined, showing that the diffusive behaviour is anomalous.
\end{abstract}

\maketitle

\tableofcontents

\section{Introduction}
Fractional kinetic equations (FKEs) are typical of several physical systems. Indeed they naturally arise, for instance, in the context of Hamiltonian chaos as a balance in a mesoscopic scale between macroscopic deterministic dynamics leading to chaos and microscopic purely random behaviour. In \cite{zaslavsky1} the author shows that, under some assumptions on the chaotic dynamics (called moderate non-linearity), the classical Fokker-Plank-Kolmogorov equations constitute a good candidate to be the kinetic equations for the system. However, if this assumption is violated, the system could exhibit an anomalous diffusive behaviour, which can be described, for instance, by the theory of L\'evy walks and flights (see \cite{zaslavsky2}). In \cite{zaslavsky3} the author finally established the link between FKEs and models of non-moderate Hamiltonian chaos (see also \cite{zaslavsky4, zaslavsky5} for further developments). It is interesting to observe that in this setting FKEs appear from a scaling limit procedure involving the L\'evy walk model, see for instance \cite{Metzler}. Among the FKEs, we are interested in particular into the following one
\begin{equation}\label{eq:tfhe}
	\pdsup{}{t}{\alpha}q(t,x)=\frac{1}{2}\pdsup{}{x}{2}q(t,x),
\end{equation}
that will be referred as (time-)fractional heat equation. Here, $\alpha \in (0,1)$ and $\dersup{}{t}{\alpha}$ is the well-known Caputo fractional derivative, defined as
\begin{equation*}
	\dersup{}{t}{\alpha}f(t)=\frac{1}{\Gamma(1-\alpha)}\der{}{t}\int_0^t (t-\tau)^{-\alpha}(f(\tau)-f(0))\, d\tau
\end{equation*}
for $\alpha \in (0,1)$ and a suitable function $f$. To use a shorter notation, we will denote $\pdsup{}{t}{\alpha}$ and $\dersup{}{t}{\alpha}$ as $\partial_t^\alpha$ and we will use an analogous notation for classical derivatives. In \cite{saichev} the authors provided the fundamental solution of a general FKE by means of a \textit{mixing formulation}, i.e.
\begin{equation}\label{eq:fracheatsol}
	q(t,x)=\frac{1}{t^\alpha}\int_0^{+\infty}f_\alpha\left(\frac{s}{t^\alpha}\right)p(s,x)\, ds,
\end{equation}
where $p(s,x)$ is the usual Gaussian heat kernel and $f_\alpha$ is the density of a suitable random variable $X$ with characteristic function $\varphi_X(z)=E_\alpha(iz)=\sum_{k=0}^{\infty}\frac{(iz)^k}{\Gamma(\alpha k+1)}$. It is clear that \eqref{eq:fracheatsol} can be rewritten in terms of a \textit{superposition formulation} as
\begin{equation*}
	q(t,x)=\E[p(t^\alpha X,x)].
\end{equation*}
By a simple conditional probability argument, one can show that $q(t,x)$ is indeed the density of the random variable \textcolor{black}{$B(t^\alpha X)\overset{d}{=}B(L_\alpha(t))$}, where $B$ is a Brownian motion and $X$ is independent of $B$, while $L_\alpha(t)$ is an $\alpha$-self-similar process independent of $B$ such that $L_\alpha(1)=X$. It turns out that the process $L_\alpha(t)$ can be described as the \textit{first time} in which an $\alpha$-stable subordinator (i.e., a positively skew $\alpha$-stable, hence increasing, process) crosses the threshold $t \ge 0$, also called inverse $\alpha$-stable subordinator, see \cite{meerstra2}. This relation has been deeper explored in \cite{fracCauchy}, where the authors extended the superposition formula to general Feller semigroups. However, it has been shown that the process $X_\alpha(t):=B(L_\alpha(t))$ is the weak limit of a Continuous-Time Random Walk (CTRW) in the sense of one-dimensional distributions (see \cite{Metzler} for details), and also in the $J_1$ Skorokhod topology (see \cite{mersch} and also \cite{meer annals probab} for generalizations). It is worth noticing that, due to the time-nonlocal nature of the equation, the FKE \eqref{eq:tfhe} cannot describe properly the transition law of a Markov process and furthermore does not characterize the whole process $X_\alpha$, since it is not Markovian. As a consequence, the identification of $X_\alpha$ as the limit of the CTRW model which leads to the FKE gives a deeper knowledge on the mechanisms behind non-moderate Hamiltonian chaos. For such a reason, the process $X_\alpha$ is usually called a Fractional Kinetic Process (FKP).

Once we have this identification, it is clear that one can study kinetic equations of chaotic systems through several techniques involving scaling limits. This is the case, for instance, of the Hamiltonian systems describing cellular flows. In \cite{hairer2} the authors provided a stochastic representation result for the limit of an averaging-homogenization problem, which is expressed in terms of the process $X_{1/2}$. The link between such a problem and the FKE \eqref{eq:tfhe} has been better underlined in \cite{hairer}. 

While being one of the most representative, Hamiltonian chaos is not the only physical occurrence of FKEs as weak scaling limits of discrete models. In \cite{monthus}, trapping models have been used to study Debye-type ageing/relaxation properties of some glass formers.  However, in \cite{rinn} the authors observed that in a randomized energy environment, the same trap models exhibit a form of subageing. This has been better captured by \cite{BAOP}, where the authors proved that these trap models with random environment converge, under some assumptions, to a FKP $X_\alpha$ (see also \cite{dintrap} for further details). Similarly, in \cite{flor} the FKP emerges as limit of an exclusion process.

The FKE \eqref{eq:tfhe} also appears in the description of heat transfer and mass infiltration in heterogeneous media. It is the case, for instance, of \cite{voller1}, in which \eqref{eq:tfhe} appears in a Green-Ampt infiltration model for moisture in the soil. This is also the case of heat transfer in materials with memory. Indeed, several nonlocal generalizations of the Fourier conduction law have been considered. In particular, a widely used generalization of such a conduction law is given by
\begin{equation}\label{eq:hff}
	\mathbf{q}(t,x)\propto -\int_0^t (t-\tau)^{\alpha-1}\nabla_x T(\tau,x)d\tau
\end{equation}
where $\mathbf{q}$ is the heat flux and $T$ is the temperature. In this case, \eqref{eq:tfhe} plays the role of the heat equation. This leads to a fractional theory of heat conduction, as described, for instance, in \cite{voller2,povstenko}. It makes then sense to consider Cauchy-Dirichlet problems associated with \eqref{eq:tfhe}, as they model the evolution of temperature in a non-Fickian material satisfying \eqref{eq:hff} subject to an external heat source: such kind of problems have been addressed, for instance, in \cite{meerbounded}. If, furthermore, the source is sufficiently hot or cold to cause a change of state, then also latent heat has to be taken in consideration, see for instance \cite{voller3}. These ideas lead to fractional Stefan problems, as discussed in \cite{garra, voller,vollergarra} and references therein, see also \cite{prsoca}. Let us consider here the \textit{sharp interface case}:
\begin{equation}\label{eq:Stefan}
	\begin{cases}
		\partial_t^\alpha u(t,x)=\frac{1}{2}\, \partial_x^2 u(t,x), & x<\varphi(t), \ t>0,\\
		u(t,x)=0, & x \ge \varphi(t), \ t \ge 0,\\
		u(0,x)=f(x) ,& x<b,\\
		\partial_t \varphi(t)=-\partial_xu(t,\varphi(t)) ,& t>0,\\
		\varphi(0)=b,
	\end{cases}
\end{equation}
where $f:(-\infty,b) \to \R$ and $b \in \R$ are given initial data, with $f(x)>0$ for all $x<b$, and $u$ and $\varphi$ are the unknowns. To solve such a problem, one could use a fixed point argument, by first solving the Cauchy-Dirichlet problem given by the first three equations for a fixed interface $\varphi$ and then finding a new interface $\varphi^\prime$ by solving the related Cauchy problem given by the fourth and fifth equalities: it is clear that a fixed point of such a procedure, together with the related solution of the Cauchy-Dirichlet fractional heat equation with moving boundary, is the desired solution of \eqref{eq:Stefan}. To do this, we first need to find the solution of a generic Cauchy-Dirichlet fractional heat equation with moving boundary:
\begin{equation}\label{eq:moving}
	\begin{cases}
		\partial_t^\alpha u(t,x)=\frac{1}{2}\, \partial_x^2 u(t,x), & x<\varphi(t), \ t>0,\\
		u(t,x)=0, & x \ge \varphi(t), \ t \ge 0,\\
		u(0,x)=f(x), & x<\varphi(0),\\
		\lim_{x \to -\infty}u(t,x)=0, & \mbox{ locally uniformly with respect to }t \ge 0,
	\end{cases}
\end{equation}
where the latter condition is required to guarantee uniqueness of the solution, as will be explained later, and $\varphi:[0,+\infty) \to \R$ is a fixed interface. 

To better understand this problem, we need to go back to the FKP $X_\alpha$, or even further back, directly to the Brownian motion.  Indeed, for $\alpha=1$, \eqref{eq:moving} is the usual Cauchy-Dirichlet problem with moving boundary for the heat equation. To shorten the notation, we denote $X_1 \equiv B$. We can kill the process $X_\alpha$ upon crossing the moving threshold $\varphi$ as follows: we define $T_\alpha$ as the first time in which $X_\alpha$ touches $\varphi$ and then we set $X^\dagger_\alpha(t)=X_\alpha(t)$ for $t<T_\alpha$ and $X^\dagger_\alpha(t)=\infty$ for $t \ge T_\alpha$, where $\infty \not \in \R$ is the \textit{cemetery state}. Since $T_\alpha$ is a stopping time with respect to the natural filtration of the process $X_\alpha$ and for $\alpha=1$ we have that $X_1$ is a Feller process, $X_1$ satisfies the Markov property in $T_1$ and thus the Dynkin-Hunt formula holds (see \cite[Theorem 2.4]{chung} for the fixed threshold case and \cite[Formula (1.1)]{lerche} for the general case). Once this has been established, in \cite[Chapter I]{lerche} it has been proved that, for $\alpha=1$, the sub-probability density of $X_1^\dagger$ is the fundamental solution of \eqref{eq:moving}. This eventually relates the heat equation \eqref{eq:moving} with the behaviour of the trajectories of $X_1$ (in particular, to its first crossing time with $\varphi$): this is indeed expected since $X_1$ is a Feller process.

On the other hand, $X_\alpha$ for $\alpha<1$ is not even Markov, and, at the same time, the time-fractional heat equation does not characterize the transition distribution of the process. Hence, at a first glance, the connection between \eqref{eq:moving} and $X_\alpha$ could be unclear. It is however worth noticing that despite $X_\alpha$ is not Markov, it still exhibits the Markov property on a suitable selection of stopping times. Indeed, $X_\alpha$ inherits the semi-Markov property from its CTRW prelimit (see \cite{cinlarsemi,kaspi} and references therein for a discussion on the topic). In general, $T_\alpha$ is not a Markov time for $X_\alpha$, unless we ask for some conditions on $\varphi$: in such a case we are able to prove a suitable Dynkin-Hunt formula and then to relate $T_\alpha$ with \eqref{eq:moving}, as stated in the following theorem, that we will prove throughout the paper.
\begin{thm}\label{thm:mainfrac}
	Assume $\varphi$ is continuous and either constant or strictly increasing with $\varphi'(0+)>0$ and $f \in C_c(-\infty,\varphi(0))$. Then there exists a function $q_\alpha:(0,+\infty) \times \R \times \R$ such that for all Borel subsets $A \in \cB(\R)$ it holds
	\begin{equation*}
		\mathbb{P}_y(X_\alpha^\dagger(t) \in A)=\int_{A} q_\alpha(s,x;y)\, dx.
	\end{equation*}
	In particular, the unique solution of \eqref{eq:moving} is given by
	\begin{equation*}
		u(t,x)=\int_{-\infty}^{\varphi(0)}f(y)q_\alpha(s,x;y)\, dy.
	\end{equation*}
\end{thm}
Let us observe here some crucial features of this result. First, notice that in any case \eqref{eq:tfhe} does not characterize the distribution of the trajectories of $X_\alpha$; nevertheless \eqref{eq:moving} still provides a quite interesting piece of information concerning the behaviour of the trajectories of $X_\alpha$, which is an unexpected result. This feature is shared with Markov processes even if the process $X_\alpha$, and its generalizations considered later, are not Markovian: the semi-Markov property of $X_\alpha$ is crucial to preserve the role of the analytical approach. The explicit use of semi-Markov property to obtain analytical results on non-local PDEs is, indeed, one of the main technical novelty concerning the proofs of our results.
 Furthermore, the assumption that $\varphi$ is continuous and strictly increasing is not really restrictive: heuristic arguments based on the second law of thermodynamics, which are then confirmed by a combination of a weak maximum principle and Hopf's lemma (the latter has been shown in \cite{roscanihopf}, see also \cite{roscani}), tell us that $\partial_t \varphi>0$ in \eqref{eq:Stefan}.

We also underline that the FKEs and the FKPs constitute a particular case of a much wider class of processes and equations, arising from limits of CTRWs. Indeed, in \cite{meertri} more general scaling limits of CTRWs have been considered, leading to processes of the form $X(t):=B(L(t))$, where $L(t)$ is the inverse of a (non necessarily stable) subordinator. Usually, the process $X$ is called \textit{delayed Brownian motion}, as in \cite{magda}. The relation between the process $X$ and a suitable time-nonlocal heat equation has been explored in \cite{zqc} (see also \cite[Appendix]{alp} for further regularity results and \cite{baemstra, costaluisa, kochu, kochukondra, KoloCTRW, kololast, meerpoisson, pierre, savtoa} for different approaches or generalizations), while some first properties of its first exit times have been explored in \cite{annals2020, pierremladen}. It is important to recall that such processes are also addressed in the context of semi-Markov processes, as observed in \cite{meerstra}.

In this paper, we will prove a general result (by substituting the power kernel $t^{-\alpha}$ with a more general one) that includes Theorem \ref{thm:mainfrac} as a particular case. This is done as follows. As already described for the FKP case, we use the semi-Markov property of $X$ to show that the first crossing time with $\varphi$ is a Markov time for $X$. Then, we use this to provide a Dynkin-Hunt-type formula for the killed process $X^\dagger$. Finally, we use the latter formula together with the fact that the density of $X$ solves a time-nonlocal heat equation, on the entire real line, to obtain the desired result. This latter step is actually quite technical and needs some precise estimates on the derivatives of the density of $X$. Concerning uniqueness of the solution, this is obtained by means of a weak maximum principle: our results in this context constitute a crucial generalizations of the ones (for nonlocal operators) in \cite{luchko} to more general parabolic domains. The reader can consult \cite{jozsef} for a maximum principle when there is also a non-local spatial component.

We plan, in future, to use our results to prove existence and uniqueness of the Stefan problem with a general memory kernel. Notice that while \eqref{eq:hff} is the most used non-local generalization of Fourier law, it is not the unique one. Indeed, general kernels were originally considered in \cite{gurtin}, see also \cite{nunziato}. Hence, it makes sense to consider Stefan problems with sharp interface with a more general convolution kernel. Moreover, a further investigation on the relation between Hamiltonian chaos theory and this special processes $X$ emerging from triangular array limits is needed, to capture some more particular anomalous diffusive behaviour (and not necessarily the power-type one), such as the ultraslow diffusive behaviour \cite{meerultra,toaldodo} or some more particular relaxation phenomena \cite{meertoa}.

{The structure of the paper is as follows.} In the forthcoming section \ref{sec:prelim} we introduce main facts on subordinators, their inverse processes and the related non-local equations. Furthermore, in Section \ref{sec:maxprin} a maximum principle for a non-local parabolic problem on a time-dependent domain will be stated: this will be used later to prove uniqueness of our heat equation. In Section \ref{sec:delbm} we introduce the delayed Brownian motion, i.e., a Brownian motion time-changed with an independent inverse subordinator. Firstly we review some main facts on it and then, in particular, we obtain several regularity properties for the density and we formalize the semi-Markov property that will be crucial for our main results. In Section \ref{sec:killedbm} we discuss the effect of killing a delayed Brownian motion, obtaining a (Dynkin-Hunt) representation for its density that will be crucial for regularity, needed in our main results. We divide the exposition between killing on a constant or time-dependent boundary. Section \ref{sec:mainres} is devoted to state and prove the main results. Finally, in Section \ref{secmsd} we obtain the anomalous diffusive behavior of the process, i.e., we study the asymptotycs of the mean square \textcolor{black}{displacement} for our killed process. In order to improve the readability, all throughout the paper several technical proofs are postponed to the Appendix \ref{appendix}.

\section{Preliminaries on subordinators and non-local operators}
\label{sec:prelim}
Throughout the paper we consider the family of filtered probability spaces \linebreak $(\Omega, \cF, \{\cF_t\}_{t \ge 0}, (\mathds{P}_x)_{x\in \R})$, and the canonical process $(B, \sigma)$ so that $B$ is a Brownian motion and $\sigma$ is an independent conservative subordinator such that $\mathds{P}_x(B(0)=x, \sigma(0)=0)=1$, i.e.~$\mathds{P}_x=\mathds{P}_{x,0}$ for simplicity. We recall, in particular, that a subordinator is a non-decreasing (hence positive) L\'evy process. We denote, for $\lambda \ge 0$,
\begin{equation}\label{eq:Bern}
	\Phi(\lambda)=-\log\left(\E\left[e^{-\lambda \sigma(1)}\right]\right)=b\lambda+\int_{\R^+}(1-e^{-\lambda \tau})\nu(d\tau),
\end{equation}
for a constant $b \ge 0$ and a measure $\nu$ on $\R^+:=(0,+\infty)$ such that
	$\int_{\R^+}(\tau \wedge 1)\nu(d\tau)<\infty$.
A function of the form \eqref{eq:Bern} is usually called a \textit{Bernstein function}, see \cite{niels1, librobern} for details. In the following, we assume the  standing assumption:
\begin{align} \tag{\textbf{A1}} \label{ass2}
	\text{$b=0$ and } \overline{\nu}(0)=\nu(0,\infty)=+\infty.
\end{align}
i.e.~$\overline{\nu}(\tau)=\nu(\tau,\infty)$ for all $\tau\geq 0$ is the tail of the L\'evy measure.
It is important to note that every Bernstein function corresponds to a unique subordinator. For this reason, we may refer to a subordinator by  its associated Bernstein function.
Under \eqref{ass2}, it is well-known that $\sigma$ is a.s. strictly increasing and of pure jump. This entails that the  process $(L(t))_{t\geq 0}$ defined, for any $t\geq0$, by
\begin{equation*}
	L(t):=\inf\{s>0: \ \sigma(s)>t\},
\end{equation*}
i.e., the generalized inverse of $\sigma$, has a.s.~continuous paths. Furthermore, for all $t>0$, we know, 	by \cite[Theorem 3.1]{meertri}, that $L(t)$ is an absolutely continuous random variable whose density $f_L(\cdot,t)$ is given by the formula
\begin{align}
	f_L(s, t) \, = \, \int_0^t \overline{\nu}(t-\tau) g_\sigma(d\tau, s),
	\label{rapprdensfL}
\end{align}
whered $g_\sigma(A, t)=\mathds{P}(\sigma(t) \in A)$ for all $A \in \cB(\R)$ and $t \ge 0$. We point out  that $\overline{\nu} \in L^1_{\rm loc}(\R^+_0)$, where $\R^+_0:=[0,+\infty)$ and $\lim_{\tau \to 0} \tau\overline{\nu}(\tau)=0$, see \cite[Eq. (3.7)]{librobern}. We can then define the mapping $I_\Phi: \R^+_0 \to \R^+_0$ as
\begin{equation}
	I_\Phi(t):=\int_0^t \overline{\nu}(\tau)d\tau
	\label{Iphi}
\end{equation}
which \btrev{will} play an important role in the sequel.

We shall need several estimates of the density of $\sigma(t)$, $L(t)$ and their derivatives which we rely on the so-called Orey's condition, see \cite{orey1968continuity}:
\begin{gather}
	 \text{ There exist  $\gamma \in (1,2)$, $C_\gamma>0$  and $ t_\gamma>0$  such that } 
	\int_{0}^t\tau^2\nu(d\tau)>4 C_\gamma t^\gamma, \text{ for any } t < t_\gamma. 
	\label{orcond} \tag{\textbf{A2}}
\end{gather}
\begin{rmk}\label{rmk:gamma12}
	In the original paper \cite{orey1968continuity}, the condition requires $\gamma \in (0,2)$. However, since $\nu$ is the L\'evy measure of a subordinator, it can be easily seen that it must be $\gamma \in (1,2)$.	
\end{rmk}
\textcolor{black}{Let us recall, by \cite[Proposition 3.7]{librobern}, that $\Phi$ admits a unique continuous extension $\Phi:\{z \in \mathbb{C}: \ \Re(z)\ge 0\} \to \{z \in \mathbb{C}: \ \Re(z)\ge 0\}$. With this in mind, we proceed with the following useful  bounds.
\begin{lem} \label{lemmaintzero}
		Under \eqref{ass2}, there exist $C_0>0$ such that for any $a \ge 0$ and $0 \le |b|<1$ it holds 
		\begin{align}
			\Re \l \Phi (a+ib) \r -\Phi(a)\geq C_0 e^{-a}|b|^2.
		\end{align}
		Furthermore, if also \eqref{orcond} holds, we have for any $a \ge 0$ and $|b|>M_\gamma:=t_\gamma^{-1}$
		\begin{align}
			\Re \l \Phi (a+ib) \r-\Phi(a) \geq C_\gamma e^{-t_\gamma a}|b|^{2-\gamma}.
		\end{align}
	\end{lem}
	\begin{proof}
		 For fixed $a \ge 0$, notice that
		 \begin{equation*}
		 	\Phi(a+ib)=\int_0^{+\infty}(1-e^{-(a+ib)\tau})\nu(d\tau)
		 \end{equation*}
		 hence
		 \begin{equation*}
		 	\Phi(a+ib)-\Phi(a)=\int_0^{+\infty}(e^{-a\tau}-e^{-(a+ib)\tau})\nu(d\tau)=\int_0^{+\infty}e^{-a\tau}(1-e^{-ib\tau})\nu(d\tau).
		 \end{equation*}
		 Taking the real part we get
		 \begin{equation*}
		 	\Re(\Phi(a+ib))-\Phi(a)=\int_0^{+\infty}e^{-a\tau}(1-\cos(b\tau))\nu(d\tau).
		 \end{equation*}
		{Recalling that for $|x|<1$ it holds $1-\cos x \geq |x|^2/4$, we get}
		\begin{equation*}
			\Re(\Phi(a+ib))-\Phi(a) \ge \int_0^{\frac{1}{|b|}}\frac{|b|^2\tau^2}{4}e^{-a\tau}\nu(d\tau) \ge \frac{|b|^2 e^{-a}}{4}\int_0^1\tau^2 \nu(d\tau)=:C_0e^{-a}|b|^2, \ |b| \in (0,1),
		\end{equation*}
		while the inequality is trivial for $b=0$.
		If \eqref{orcond} holds, then
		\begin{align}\label{427}
			\Re \l \Phi (a+ib) \r-\Phi(a) \geq \frac{|b|^2}{4}\int_0^{\frac{1}{|b|}} \tau^2 e^{-a\tau} \nu(d\tau)\geq C_\gamma e^{-\frac{a}{M_\gamma}}|b|^{2-\gamma}, \qquad |b| > M_\gamma.
		\end{align}
	\end{proof}
	From now on, we will set $\Psi(\xi):=\Phi(-i\xi)$ for $\xi \in \R$.
}
As a consequence, as shown in \cite{orey1968continuity}, \eqref{orcond} implies
\begin{equation}
	|\E[e^{i\xi\sigma(t)}]|\le e^{- C_\gamma t |\xi|^{2-\gamma}}, \quad  |\xi|>M_\gamma,
	\label{oreychar}
\end{equation} 
and then for $t>0$ the r.v. $\sigma(t)$ admits an infinitely differentiable density on $(0,\infty)$, that we denote $g_\sigma(\cdot,t)$. With this in mind, we can provide the following estimates.
\begin{prop}
	\label{derfL}
	Assume \eqref{ass2} and \eqref{orcond}. Then $f_L$ admits first-order partial derivatives and, for $s,t>0$,
	\begin{align}
		\partial_t f_L(s, t) &=  \int_0^t \overline{\nu}(\tau) \partial_t g_\sigma(t-\tau, s) d\tau. \label{dert}\\
		\partial_s f_L(s, t) &=  \int_0^t \overline{\nu}(\tau) \partial_s g_\sigma(t-\tau, s) d\tau \label{ders}
	\end{align}
	In particular,
	\begin{align}\label{upbound}
		\left|\partial_t f_L(s,t)\right|&\le \frac{I_\Phi(t)}{\pi}\int_0^{+\infty}\xi e^{-s\Re(\Psi(\xi))}d\xi \\
		\left|\partial_s f_L(s,t) \right|&\le \frac{I_\Phi(t)}{\pi}\left(\sqrt{2}\int_{0}^{M_\gamma}|\Psi(\xi)|e^{-s\Re(\Psi(\xi))}\, d\xi\right. \nonumber\\
		&\left.+\int_{M_\gamma}^{+\infty}\left(3\overline{\nu}(1)+\xi \int_0^1 \tau \nu(d\tau)+\frac{\xi^2}{2}\int_0^1 \tau^2\nu(d\tau)\right)e^{-s\btrev{C_\gamma}\xi^{2-\gamma}}\, d\xi\right), \label{upbound2}
	\end{align}
	where $I_\phi$ is defined in \eqref{Iphi}. Furthermore we have that, locally uniformly for $t \in (0,+\infty)$,
	\begin{equation*}
		\lim_{s \to \infty}sf_L(s,t)=\lim_{s \to \infty}s^2|\partial_s f_L(s,t)|=0.
	\end{equation*}
	and also $\partial_t f_L(s,\cdot), \partial_s f_L(s,\cdot) \in C(\R^+)$,
	\begin{equation}\label{eq:unifconvparfL}
		\lim_{t \downarrow 0}\partial_t f_L(s,t)=\lim_{t \downarrow 0}\partial_s f_L(s,t)=0 \mbox{ locally uniformly with respect to }s>0.
	\end{equation}
\end{prop}
The proof is given in Appendix \ref{App1True}.

We can now define a family of linear operators associated to a subordinator, or equivalently to a Bernstein \btrev{function}. For a fixed Bernstein function $\Phi$ statisfying \eqref{ass2}, we define the convolution integral operator acting on $L^1_{\rm loc}(\R^+_0)$
\begin{equation*}
	\mathcal{I}^\Phi f(t):=\int_0^t \overline{\nu}(t-\tau)f(\tau)\, d\tau
\end{equation*}
and the \textit{generalized fractional derivative}
\begin{equation}\label{eq:gfd}
	\partial_t^\Phi f(t):=\der{}{t}\int_0^t \overline{\nu}(t-\tau)(f(\tau)-f(0))\, d\tau=\der{}{t}\mathcal{I}^\Phi (f(\cdot)-f(0))(t),
\end{equation}
provided the quantity is well-defined. It is, in any case, important to observe that if $f \in {\rm AC}[0,T]$ for some fixed $T>0$, then $\partial_t^\Phi f \in L^1[0,T]$.
\begin{lem}\label{lem:derdentro}
	Fix $T>0$. For $f\in {\rm AC}[0,T]$ it holds $\partial_t^\Phi f \in L^1[0,T]$ with
	\begin{equation}
		\partial_t^\Phi f (t)=\int_0^t \overline{\nu}(t-s)\partial_t f(s)ds, \ \mbox{ for a.e. }t \in [0,T].
		\label{derdentro}
	\end{equation}
	If furthermore $f$ is Lipschitz, then $\partial_t^\Phi f \in C[0,T]$ and \eqref{derdentro} holds for all $t \in [0,T]$.
\end{lem}
\begin{proof}
	Since $f \in {\rm AC}[0,T]$, then $\partial_t f \in L^1[0,T]$ and $\mathcal{I}^\Phi (\partial_t f) \in L^1[0,T]$. First, notice that
	
	\begin{align}\label{eq:FubiniinG}
		\begin{split}
			\int_0^t \int_0^s \overline{\nu}(\tau)|\partial_t f(s-\tau)|d\tau ds&=\int_0^t \overline{\nu}(\tau)\int_\tau^t |\partial_t f(s-\tau)|ds d\tau\\
			&=\int_0^t \overline{\nu}(\tau)\int_0^{t-\tau} |\partial_t f(z)|dz d\tau\\
			&
			\le \left(\int_0^t \overline{\nu}(\tau)d\tau\right)\left(\int_0^{t} |\partial_t f(z)|dz\right)<+\infty.
		\end{split}
	\end{align}
	Hence, by Fubini's theorem and \cite[Theorem $7.29$]{wheeden}
	\begin{align*}
		\int_0^t \mathcal{I}^\Phi(\partial_t f)(s)ds&=\int_0^t\int_0^s \overline{\nu}(\tau)\partial_t f(s-\tau)d\tau ds\\
		&=\int_0^t\overline{\nu}(\tau)\int_\tau^t \partial_t f(s-\tau)ds d\tau=\int_0^t\overline{\nu}(\tau)(f(t-\tau)-f(0))d\tau.
	\end{align*}
	Hence we can differentiate both sides to conclude the first part of the proof. The second part of the statement is a direct consequence of \cite[Proposition $1.3.2$]{abhn}. 
\end{proof}
\begin{rmk}
	We point out that there exist functions $f \not \in {\rm AC}[0,T]$ such that $\partial_t^\Phi f$ is well-defined. For instance, if $\Phi(\lambda)=\lambda^\alpha$, i.e.~$\overline{\nu}(\tau)=\frac{\tau^{-\alpha}}{\Gamma(1-\alpha)}$ for $\tau>0$, \btrev{then $\partial_t^{\Phi} f$ is well-defined for $f \in C^\beta[0,T]$} for \textcolor{black}{$\beta \ge \alpha$}, see \cite[Theorem 1.18]{amp2}.  
\end{rmk}
The relation between the operator $\partial_t^\Phi$ and the density $f_L$ is underlined by the next result.
\begin{prop}\label{prop:fL}
	Under Assumptions \eqref{ass2} and \eqref{orcond} it holds
	\begin{equation}\label{eqdensfrac}
		\partial_t^\Phi f_L(s;t)=-\partial_s f_L(s;t) \qquad s,t>0,
	\end{equation}
	where we set $f_L(s;0)=f_L(0;t)=0$ for all $s,t>0$.
\end{prop}
\begin{proof}
First, note that if one proves the equality
\begin{equation}\label{eq:preLap}
	\left(\mathcal{I}^\Phi f_L(s;\cdot)\right)(t)=-\int_0^t\partial_s f_L(s;w)\, dw, \ t,s>0
\end{equation}	
then \eqref{eqdensfrac} follows by differentiating both sides of the previous relation. To show \eqref{eq:preLap}, let us first observe, as a direct consequence of \eqref{rapprdensfL} that, for $\lambda>0$,
\begin{equation*}
	\int_0^{+\infty}e^{-\lambda t}f_L(s;t)\, dt=\frac{\Phi(\lambda)}{\lambda}e^{-s\Phi(\lambda)}
\end{equation*}
hence
\begin{equation*}
	\int_0^{+\infty}e^{-\lambda t}\left(\mathcal{I}^\Phi f_L(s;\cdot)\right)(t)\, dt=\frac{\Phi^2(\lambda)}{\lambda^2}e^{-s\Phi(\lambda)}.
\end{equation*}
On the other hand, from \eqref{upbound2}, we know that, for $\lambda>0$,
\begin{equation*}
	\int_0^{+\infty}e^{-\lambda t}|\partial_s f_L(s;t)|\, dt<\infty
\end{equation*}
hence
\begin{equation*}
	\int_0^{+\infty}e^{-\lambda t}\int_0^t\partial_s f_L(s;w)\, dw \, dt=\frac{1}{\lambda}\int_{0}^{+\infty}e^{-\lambda \btrev{w}}\partial_s f_L(s;w)\, dw.
\end{equation*}
Next, we observe that \eqref{upbound2} also holds on incremental ratios, see the proof in Appendix \ref{App1True}, so that, by a simple dominated convergence argument, we have
\begin{equation*}
	\int_0^{+\infty}e^{-\lambda t}\int_0^t\partial_s f_L(s;w)\, dw \, dt=\frac{1}{\lambda}\partial_s\left(\int_{0}^{+\infty}e^{-\lambda \btrev{w}} f_L(\cdot;w)\, dw\right)(s)=-\frac{\Phi^2(\lambda)}{\lambda^2}e^{-s\Phi(\lambda)}.
\end{equation*}
Hence, for fixed $s>0$, we know that \eqref{eq:preLap} holds for $t \in \R^+ \setminus \mathcal{N}$, where $|\mathcal{N}|=0$, by \cite[Theorem 1.7.3]{abhn}. However, since $g_\sigma(\cdot;s)$ is continuous on $\R_0^+$ for $s>0$, setting $g_\sigma(0;s)=0$, and $\overline{\nu} \in L^1_{\rm loc}(\R_0^+)$, then $f_L(s;\cdot)$ is continuous on $\R_0^+$ by \eqref{rapprdensfL} and, as consequence, $\mathcal{I}^\Phi f_L(s;\cdot)$ is continuous on $\R_0^+$. Hence \eqref{eq:preLap} holds for all $t,s>0$. 
\end{proof}
We will also make use of the following function
\begin{equation}
	U_{p}(t)=\E_x[L(t)^{p}], \ t>0,
	\label{Up}
\end{equation}
that is finite for all $t>0$ and $p>-1$ by \cite[Lemma 2.3]{amp} and \cite[Lemma 4.1]{ascione}. We need, however, the following integrability property.
\begin{prop}\label{prop:U12L1}
	Under Assumption \eqref{ass2} we have $U_{p} \in L^1_{\rm loc}(\R_0^+)$ for all $p>-1$.
\end{prop}
\begin{proof}
	Consider for $\lambda>0$
	\begin{equation*}
		\int_0^{+\infty}e^{-\lambda t}U_{p}(t)\, dt,
	\end{equation*}
	that is well-defined since $U_{p}(t)>0$ for all $t>0$. By Tonelli's theorem, we have
	\begin{align}\label{eq:Laptrans}
		\int_0^{+\infty}e^{-\lambda t}U_{p}(t)\, dt=&\int_0^{+\infty}s^{p}\int_0^{+\infty}e^{-\lambda t}f_L(s;t)\, dt \, ds\\
		=&\frac{\Phi(\lambda)}{\lambda}\int_0^{+\infty}s^{p}e^{-s\Phi(\lambda)}ds=\frac{\Gamma(p+1)}{\lambda(\Phi(\lambda))^{p}}<\infty.
	\end{align}
	In particular, for any $T>0$ it holds
	\begin{equation*}
		\int_0^T U_{p}(t)\, dt \le e^{T}\int_0^T e^{-t} U_{p}(t)\, dt \le \Gamma(p+1)(\Phi(1))^{-p}e^{T}<\infty.
	\end{equation*}
\end{proof}
\textcolor{black}{Concerning the density of subordinators, we will need a slightly improved version of \cite[Equation (5.32)]{tlms2024} under an additional condition on the Bernstein function.
\begin{prop}\label{prop:keyhole}
	Suppose Assumption \eqref{ass2} holds. and assume further that:
	\begin{itemize}
		\item There exists $\theta \in \left(\frac{\pi}{2},\pi\right)$ such that $\Phi$ admits a holomorphic extension in the complex sector
		\begin{equation*}
			\mathbb{C}(\theta):=\{z \in \mathbb{C}: \ |{\rm Arg}(z)| < \theta\},
		\end{equation*}
		where ${\rm Arg}(z) \in (-\pi,\pi]$ is the principal argument of $z \in \mathbb{C}$;
		\item $\Phi$ is continuous in $\overline{\mathbb{C}(\theta)}$;
		\item For all $z \in \overline{\mathbb{C}(\theta)} \setminus \{0\}$ it holds $\Re(\Phi(z))>0$;
		\item The limit
		\begin{equation*}
			\lim_{z \to \infty}\frac{\Phi(z)}{z}=0
		\end{equation*}
		holds uniformly in $\overline{\mathbb{C}(\theta)}$.
		\item For any fixed $a>0$ it holds
		\begin{equation*}
			\lim_{b \to +\infty}\frac{\Re(\Phi(a+ib))}{\log(b)}=+\infty.
		\end{equation*}
	\end{itemize}
	Then
	\begin{equation}\label{eq:gstbound}
		g(s,t)=\frac{1}{\pi}\int_0^{+\infty}\Im\left(e^{i\theta+s\xi e^{i\theta}}\left(e^{-t\Phi(\xi e^{i\theta})}-1\right)\right)\, d\xi.
	\end{equation} 
	If furthermore there exists $\alpha \in (0,1]$ such that 
	\begin{equation*}
		\lim_{r \to +\infty}\frac{|\Phi(re^{i \theta})|}{r^\alpha}=\ell 
	\end{equation*}
	for some $\ell \ge 0$, then there exists a costant $C>0$ such that 
	\begin{equation}\label{eq:controlgst}
		g(s,t) \le Ct\left(1+\frac{1}{s^{\alpha+1}}\right).
	\end{equation}
\end{prop}
The proof is quite technical and is left in Appendix \ref{app:keyhole}
}
\subsection{A weak maximum principle}
\label{sec:maxprin}
In the following we will be interested in a time-nonlocal heat equation, on a specific parabolic domain, involving the operator $\partial_t^\Phi$. To prove uniqueness of the solution, we will make use of a weak maximum principle, which we now state for a more general advection-diffusion operator. Before doing this, we need to introduce some notation concerning parabolic domains. For $T \in [0,+\infty]$ and $N \in \N$, let $E \subset [0,T] \times \R^N$ (if $T=+\infty$ set $[0,T]:=\R_0^+$). We denote by $\overline{E}$ the usual topological closure of $E$ while $\mathring{E}$ will be the usual topological interior of $E$. Given $(t_0,x_0) \in [0,T] \times \R^N$ and $r>0$, $B_r(t_0,x_0)$ denotes the Euclidean ball of radius $r>0$ and center $(t_0,x_0)$ in $\R^{N+1}$. The parabolic interior $E^\ast$ of $E$ is defined as follows 
\begin{equation*}
	(t_0,x_0) \in E^\ast \Leftrightarrow \exists r>0: \ B_r(t_0,x_0)\cap \{(t,x) \in [0,T]\times \R^N: \ t \le t_0\} \subset E.
\end{equation*}
It is clear that $\mathring{E} \subset E^\ast$. However, the two sets do not necessarily coincide. For instance, if $G \subset \R^N$ is an open set and $E=[0,T] \times G$, then $\mathring{E}=(0,T) \times G$ and $E^\ast=(0,T] \times G$. We say that $E$ is a parabolic open set if $E^\ast=E$. The parabolic boundary of $E$ is given by $\partial_p E:=\overline{E} \setminus E^\ast$. It is clear that $\partial_p E \subset \partial E$, where $\partial E:=\overline{E}\setminus \mathring{E}$ is the usual boundary of the set $E$. We also consider the slices
\begin{align*}
	E_1(x)&=\{t \in [0,T]: \ (t,x) \in E\} \qquad x \in \R^N\\
	E_2(t)&=\{x \in \R^N: \ (t,x) \in E\} \qquad t \in [0,T]
\end{align*}
and the projections $E_1=\bigcup_{x \in \R^N}E_1(x)$ and $E_2=\bigcup_{t \in [0,T]}E_2(t)$. We say that $E$ is non-decreasing with respect to the variable $t$ if $E_2(t_1)\subseteq E_2(t_2)$ whenever $0 \le t_1 \le t_2 \le T$. If $T<\infty$ and $E \subset [0,T] \times \R^N$, given two functions $f:(t,x) \in E \to \R$ and $u:E_1 \to \R$, we say that $\lim_{x \to \infty}f(t,x)=u(t)$ uniformly with respect to $t \in [0,T]$ if \textcolor{black}{defining $\widetilde{E}=E_1 \times \overline{E_2}$ and
\begin{equation*}
	\widetilde{f}(t,x)=\begin{cases} f(t,x) & (t,x) \in E \\
		u(t) & (t,x) \in \widetilde{E} \setminus E
	\end{cases}
\end{equation*}
we have the following two properties:
\begin{itemize}
	\item[(i)] $\widetilde{f}$ is a continuous function on $\widetilde{E}$;
	\item[(ii)] For all $\varepsilon>0$ there exists a compact set $K \subsetneq E_2$ such that
	\begin{equation*}
		\sup_{t \in E_1}\sup_{x \in E_2}|\widetilde{f}(t,x)-u(t)| <\varepsilon.
	\end{equation*}
\end{itemize}}
If instead $E \subset \R_0^+ \times \R^N$, we say that $\lim_{x \to \infty}f(t,x)=u(t)$ locally uniformly with respect to $t \ge 0$ if $\lim_{x \to \infty}f(t,x)=u(t)$ uniformly with respect to $t \in [0,T]$ for all $T>0$. \textcolor{black}{For $N=1$, if $E_2$ is not lower (resp. upper) bounded, then we say that $\lim_{x \to -\infty}f(t,x)=u(t)$ (resp. $\lim_{x \to +\infty}f(t,x)=u(t)$) uniformly with respect to $t \in [0,T]$ if extending $f$ to $\widetilde{f}$ as before we have
	\begin{equation*}
		\lim_{x \to -\infty}\sup_{t \in E_1}|\widetilde{f}(t,x)-u(t)|=0 \qquad \left(\text{ resp.} \lim_{x \to +\infty}\sup_{t \in E_1}|\widetilde{f}(t,x)-u(t)|=0\right).
	\end{equation*}
}

We can now state the following result, which is a weak maximum principle (the operators $\nabla$ and $\Delta$ are referred to the $x$ variable in $\R^N$).
\begin{thm}\label{thm:weakmax}
	\textcolor{black}{Assume \eqref{ass2} holds}. Let $T \in [0,+\infty)$ and $E \subset [0,T] \times \R^N$ be connected and non-decreasing with respect to the variable $t$ and define, for $x \in \overline{E}_2$, $t(x)=\min\{t \ge 0: \ x \in \overline{E}\}$. Let $u:\R_+^0 \times \R^N \to \R$ be such that $u \in C(\overline{E})$, $u(t,\cdot) \in C^2(E_2(t))$ for all $t \in E_1$ \textcolor{black}{with $t>0$}, $u(t,x)=u(t(x),x)$ for all $x \in E_2$ and $t \in [0,t(x)]$ and $\partial_t^\Phi u(t,x)$ is well-defined for $(t,x) \in E$ with $\partial_t^\Phi u(\cdot,x) \in C(E_1(x))\cap L^1_{\rm loc}(\overline{E_1(x)})$ for all $x \in E_2$. Furthermore, if $E$ is not bounded, assume that there exists a function $u_\infty \in C[0,T]$ such that $\lim_{x \to \infty}u(t,x)=u_\infty(t)$ uniformly with respect to $t \in [0,T]$. 
	Finally, suppose that
	\begin{equation*}
		\partial_t^\Phi u(t,x)-\langle p_1(x), \nabla u(t,x)\rangle-p_2(x)\Delta u(t,x) \le 0 \ \forall (t,x) \in E,
	\end{equation*}
	where $p_1:E_2 \to \R^N$ and $p_2:E_2 \to \R_0^+$. Then, if $E$ is bounded,
	\begin{equation*}
		\max_{(t,x) \in \overline{E}}u(t,x)=\max_{(t,x) \in \partial_p E}u(t,x)
	\end{equation*}
	while if $E$ is unbounded
	\begin{equation*}
		\sup_{(t,x) \in \overline{E}}u(t,x)=\max\left\{\max_{(t,x) \in \partial_p E}u(t,x),\max_{t \in [0,T]}u_\infty(t)\right\}.
	\end{equation*}
\end{thm}
We leave the proof of this technical result, which is of independent interest, in Appendix \ref{App2new}.
\section{The Brownian motion delayed by an inverse subordinator}
\label{sec:delbm}
\subsection{Definition and regularity of the density}
Let $B$ be the Brownian motion introduced in the previous section. Then we define the Brownian motion \textit{delayed} by an inverse subordinator (see \cite{magda}) as
\begin{equation*}
	X_\Phi(t):=B(L(t)), \ t \ge 0.
\end{equation*}
Notice that, under \eqref{ass2} the process $X_\Phi$ is $\mathds{P}_y$-a.s. continuous for any $y \in \R$ since it is the composition of two continuous functions. By independence of $B$ and $L$, it is immediate to notice that for any $A \in \cB(\R)$ and $y \in \R$ it holds
\begin{equation*}
	\mathds{P}_y(X_\Phi(t) \in A)=\int_{A} p_\Phi(t,x;y)\, dx
\end{equation*}
where
\begin{equation}\label{eq:pPhidef}
p_\Phi(t,x;y)=\int_0^{+\infty}p(s,x;y)f_L(s;t)\, ds
\end{equation}
and
\begin{equation*}
	p(s,x;y)=\frac{1}{\sqrt{2\pi s}}e^{-\frac{(x-y)^2}{2s}}, \ (s,x,y) \in (0,\infty) \times \R^2
\end{equation*}
is the density of the Brownian motion $B$ under $\mathds{P}_y$. It is clear that $p_\Phi(t,x;y)$ depends actually on $x-y$. Now we give some results concerning the regularity of the function $p_\Phi$ with respect to both \btrev{$x$ and $t$ variables}.

\begin{prop}\label{prop:regx}
	Under \eqref{ass2}, $p_\Phi$, $\partial_x p_\Phi$ and $\partial^2_x p_\Phi$ exist in $C(\R^+ \times (\R^2 \setminus {\sf diag}(\R^2)))$, where ${\sf diag}(\R^2)=\{(x,x), \ x \in \R\}$ and
	\begin{align}
		\begin{split}
		\partial_x p_\Phi(t,x;y)&=\int_0^{+\infty}\partial_x p(s,x;y)f_L(s;t)\, ds\\
		&=-\btrev{\frac{1}{\sqrt{2\pi}}}\int_0^{+\infty}\frac{(x-y)e^{-\frac{(x-y)^2}{2s}}}{s^\frac{3}{2}}f_L(s;t)\, ds,	
		\end{split}
		 \label{derx}\\
		\begin{split}
		\partial^2_xp_\Phi(t,x;y)&=\int_0^{+\infty}\partial^2_xp(s,x;y)f_L(s;t)\, ds\\
		&=\btrev{\frac{1}{\sqrt{2\pi}}}\int_0^{+\infty}\frac{1}{ s^{\frac{3}{2}}}\left(\frac{(x-y)^2}{s}-1\right)e^{-\frac{(x-y)^2}{2s}}f_L(s;t)\, ds.	
		\end{split}
		 \label{derx2}
	\end{align}
	In particular, $p_\Phi \in C(\R^+ \times \R^2)$,
	\begin{equation*}
	\lim_{t \to 0^+}p_\Phi(t,x;y)=0 \mbox{ locally uniformly in } \{(x,y) \in \R^2: \ x \not = y\}	
	\end{equation*}
	and, in weak convergence,
	\begin{equation*}
	\lim_{t \to 0^+}p_\Phi(t,x;y)dy=\delta_x(dy)	
	\end{equation*}
	If also \eqref{orcond} holds, then, for $x \not = y$, $p_\Phi (\cdot, x;y)$ is differentiable on $(0, \infty)$, $\partial_t p_\Phi \in C(\R^+ \times (\mathbb{R}^2 \setminus {\sf diag}(\R^2)) )$,
	\begin{align}
		\partial_t p_{\Phi}(t,x;y) \, = \, \int_0^{+\infty}p(s, x;y) \,  \partial_t f_L(s,t) ds
		\label{diffpPhi}
	\end{align}
	and
	\begin{equation}\label{eq:locunifpartpphi}
	\lim_{t \to 0^+}\partial_t p_\Phi(t,x;y) \to 0 \ \mbox{locally uniformly with respect to } (x,y) \in \mathbb{R}^2 \setminus {\sf diag}(\R^2).
	\end{equation}
	Furthermore, for $x \not = y$, $\partial_t^\Phi p_\Phi(\cdot,x;y)$ is well-defined on $(0,\infty)$, $\partial_t^\Phi p_\Phi \in C(\R^+ \times (\R^2 \setminus {\sf diag}(\R^2)))$,
	\begin{align}
		\partial_t^\Phi p_{\Phi}(t,x;y) =  \int_0^{+\infty}p(s, x;y) \,  \partial_t^\Phi f_L(s,t) ds=-\int_0^{+\infty}p(s, x;y) \,  \partial_s f_L(s,t) ds.
		\label{diffPhipPhi}
	\end{align}
	and
	\begin{equation}\label{eq:locunifpartpphi2}
		\lim_{t \to 0^+}\partial^\Phi_t p_\Phi(t,x;y) \to 0 \ \mbox{locally uniformly with respect to } (x,y) \in \mathbb{R}^2 \setminus {\sf diag}(\R^2).
	\end{equation}
\end{prop}
The technical proof is given in Appendix \ref{Appreg}.

In next theorem, we prove that $p_\Phi$ satisfies pointwise a time-nonlocal heat equation for $t>0$ and $x \not = y$. To do this, however, we preliminarily set $p_\Phi(0,x;y)=0$ whenever $x \not = y$. In such a way, $p_\Phi \in C(\R_0^+ \times (\R^2 \setminus {\sf diag}(\R^2)))$.
\begin{thm}
	\label{lemmaeqfraz}
	Under Assumptions \eqref{ass2} and \eqref{orcond}, 
	\begin{equation}
		\partial_t^\Phi p_\Phi(t,x;y)=\frac{1}{2}\partial_x^2 p_\Phi(t,x;y), \ t>0,\ (x,y) \in \R^2 \setminus {\sf diag}(\R^2).
		\label{eqcalfraz}
	\end{equation}
\end{thm}
\begin{proof}
	Arguing as in Proposition \ref{prop:fL}, it is sufficient to show that for $(x,y) \in \R^2 \setminus {\sf diag}(\R^2)$ and $t>0$,
	\begin{align}
		\mathcal{I}^\Phi p_\Phi(t,x;y)=\frac{1}{2} \int_0^t \partial_x^2 p_\Phi(s,x;y) \, ds 
		\label{intfrac}
	\end{align}
	To do this, first observe that for $\lambda>0$
	
	\begin{align}
		\int_0^{+\infty} e^{-\lambda t}  p_\Phi(t, x;y)\, dt &= \int_0^{+\infty} e^{-\lambda t} \int_0^{+\infty} p(s, x;y) f_L(s, t) \, ds \, dt \notag \\
		 &= \,  \frac{\Phi(\lambda)}{\lambda} \int_0^{+\infty}e^{-s\Phi(\lambda)}p(s, x;y)\, ds
		\label{539}
	\end{align}
	and then
	\begin{align}
		\int_0^{+\infty} e^{-\lambda t}  \mathcal{I}^\Phi p_\Phi(t, x;y)\, dt &=\frac{\Phi^2(\lambda)}{\lambda^2} \int_0^{+\infty}e^{-s\Phi(\lambda)}p(s, x;y)\, ds.
		\label{5392}
	\end{align}
	Let us also recall that, clearly,
	\begin{align}
		&\int_0^{+\infty} e^{-\lambda t} \left(\left| \partial_x p(t, x;y) \right|+\left| \partial_x^2 p(t, x;y) \right|\right) dt < \infty
		\label{l1}
	\end{align}
	and then, by \eqref{derx2},
	\begin{equation}\label{eq:preLap2}
		\int_0^{+\infty}e^{-\lambda t}\left(\int_0^{t}\partial_x^2 p_\Phi(s,x;y)\, \, ds\right) \, dt=\frac{\Phi(\lambda)}{\lambda^2}\int_0^{+\infty}e^{-s\Phi(\lambda)}\partial_x^2p(s,x;y)\, ds.
	\end{equation}
	However, we recall that $p(t,x;y)$ satisfies
	\begin{equation*}
		\partial_t p(t,x;y)=\frac{1}{2}\partial_x^2 p(t,x;y) \qquad t>0, \ (x,y) \in \R^2
	\end{equation*}
	and \eqref{l1} guarantees that $\partial_t p(\cdot,x;y)$ admits Laplace transform with non-positive abscissa of convergence. Hence, by \eqref{eq:preLap2} and the fact that $\Phi(\lambda)>0$, we get
	\begin{align}
		\int_0^{+\infty}e^{-\lambda t}\left(\int_0^{t}\partial_x^2 p_\Phi(s,x;y)\, \, ds\right) \, dt&=\frac{2\Phi(\lambda)}{\lambda^2}\int_0^{+\infty}e^{-s\Phi(\lambda)}\partial_sp(s,x;y)\, ds \nonumber \\
		&=\frac{2\Phi^2(\lambda)}{\lambda^2}\int_0^{+\infty}e^{-s\Phi(\lambda)}p(s,x;y)\, ds. \label{eq:preLap3}
	\end{align}
	Comparing \eqref{5392} and \eqref{eq:preLap3}, by the injectivity of the Laplace transform we get that for all $(x,y) \in \R^2 \setminus {\sf diag}(\R^2)$ \eqref{intfrac} holds for $t \in \R^+ \setminus \mathcal{N}$, where $|\mathcal{N}|=0$. However, by continuity of both sides of \eqref{intfrac}, that easily follows by \eqref{eq:pPhidef} and the fact that $\overline{\nu} \in L^1_{\rm loc}$, we have that $\mathcal{N}=\emptyset$ and \eqref{intfrac} follows.
\end{proof}
\begin{rmk}\label{rmkcont}
	The previous theorem also guarantees that $\partial_t^\Phi p_\Phi \in C(\R^+ \times (\R^2 \setminus {\sf diag}(\R^2)))$.
\end{rmk}

Notice that $\partial_x p_\Phi$, $\partial_x^2 p_\Phi$ and $\partial_t^\Phi p_\Phi$ can be infinite for $t>0$ as $x-y \to 0$. \textcolor{black}{To guarantee a uniform limit in this regime, we need a further condition on $\Phi$:
\begin{gather}
	\text{There exists $\theta \in \left(\frac{\pi}{2},\pi\right)$ such that $\Phi$ admits a holomorphic extension} \notag\\
	\text{on the complex sector $\mathbb{C}(\theta)$ which is continuous on $\overline{\mathbb{C}(\theta)}$ and satisfies} \tag{\textbf{A3}}  \label{ass4}\\ \text{$\lim_{z \to \infty}\frac{\Phi(z)}{z}=0$ uniformly in $\overline{\mathbb{C}(\theta)}$ and ${\rm Arg}(\Phi(z))<\pi$}\notag 
\end{gather}
Let us first stress that, as evidenced in \cite[Section 3.2.1]{tlms2024}, Assumption \eqref{ass4} is verified whenever $\Phi$ is a complete Bernstein function satisfying Assumption \eqref{ass2} or has the form $\Phi(\lambda)=\Phi_1(\lambda^\alpha)$ for some $\alpha \in (0,1)$ and any other Bernstein function $\Phi_1$ satisfying \eqref{ass2}. Furthermore, a consequence of such a condion is given by the following result, that is a direct consequence of \cite[Theorem 3.18]{tlms2024} for $n=0$.
\begin{prop}\label{prop:boundderfL}
	Suppose that $\Phi$ satisfies Assumption \eqref{ass4}. Then
	\begin{equation}\label{eq:convtonu}
		\lim_{s \downarrow 0}f_L(s,t)=\overline{\nu}(t), \qquad \mbox{locally uniformly with respect to }t>0 
	\end{equation}
	and 
	\begin{equation}\label{eq:linflocpsfl}
	\partial_s f_L \in L^\infty_{\rm loc}(\R^+_0 \times \R^+)	
	\end{equation}
\end{prop}
}

With the previous assumption we can prove the following limit behaviour.
\begin{prop}\label{prop:limitseconder}
	Suppose Assumptions \eqref{ass2} and \eqref{ass4} hold. Then
	\begin{equation}\label{eq:limtder1}
		\lim_{x-y \to 0^\pm}\partial_x p_\Phi(t,x;y)=\mp \overline{\nu}(t)\, \quad \mbox{locally uniformly w.r.t. }t>0.
	\end{equation}	
	Furthermore, if also Assumption \eqref{orcond} holds, then
	\begin{equation}\label{eq:limtder2}
		\displaystyle \frac{\lim\limits_{x-y \to 0^\pm}\partial_x^2 p_\Phi(t,x;y)}{2}=\lim_{x-y \to 0^\pm}\partial_t^\Phi p_\Phi(t,x;y)=-\frac{1}{\sqrt{2\pi}}\int_0^{+\infty}s^{-\frac{1}{2}}\partial_sf_L(s;t)\, ds.
	\end{equation}
	locally uniformly with respect to $t>0$.
\end{prop}
\textcolor{black}{\begin{rmk}\label{rmk:lessrestrict}
	The previous proposition still holds if, in place of Assumption \eqref{ass4}, we ask for \eqref{eq:convtonu} and that for all compact $K \subset (0,+\infty)$ there exists $\delta>0$ and a function $h:(0,\delta) \to \R$ such that 
	\begin{equation*}
		\left|\partial_s f_L(s,t)\right| \le h(s) \ \forall (s,t) \in (0,\delta) \times K \qquad \mbox{ and }\qquad \int_0^{\delta}\frac{h(s)}{\sqrt{s}}\, ds<\infty.
	\end{equation*}
	The latter is implied, according to Proposition \ref{prop:boundderfL}, by Assumption \eqref{ass4}.
\end{rmk}}
The technical proof of Proposition~\ref{prop:limitseconder} is in Appendix \ref{AppLimits}.

\textcolor{black}{Once we have \eqref{eq:limtder1} and \eqref{eq:limtder2}, one can also evaluate explicitly the integral of $\partial_x^2p_\Phi(t,x;y)$, as in the following result.
\begin{prop}\label{prop:secder}
	Let $f \in C_{\sf b}(\R)$ and consider the function
	\begin{equation*}
		u(t,x)=\int_{\R} f(y)p_\Phi(t,x;y)\, dy.
	\end{equation*}
	Then $u \in C(\R_0^+ \times \R)$ with $u(0,x)=f(x)$. Furthermore, for all $t>0$ $u(t,\cdot) \in C^2(\R)$, and $\partial_x u,\partial^2_xu \in C(\R^+ \times \R)$. Moreover, we have
	\begin{equation}\label{eq:der1int}
		\partial_x u(t,x)=\int_{\R} f(y)\partial_x p_\Phi(t,x;y)\, dy
	\end{equation}	
	and
	\begin{equation}\label{eq:der2int}
		 \partial_x^2 u(t,x)=\int_{\R}f(y)\partial_x^2p_\Phi(t,x;y)\, dy-2\nu(t)f(x).
	\end{equation}
	In particular,
	\begin{equation}\label{eq:intofpd2}
		\int_{\R}\partial_x^2p_\Phi(t,x;y)\, dy=2\nu(t)
	\end{equation}
	and \eqref{eq:der2int} can be rewritten as
	\begin{equation}\label{eq:der2int2}
		\partial_x^2 u(t,x)=\int_{\R}(f(y)-f(x))\partial_x^2p_\Phi(t,x;y)\, dy.
	\end{equation}
\end{prop}
Again, the proof of this proposition is technical and thus left in Appendix \ref{Appsecder}.}

\textcolor{black}{Let us introduce the following notation: for a function $f \in L^1[0,T]$, we set
\begin{equation*}
	D^\Phi_t f(t)=\der{}{t}\int_0^t \overline{\nu}(t-\tau)f(\tau)\, d\tau=\der{}{t}\mathcal{I}^\Phi f(t),
\end{equation*}
provided the quantity is well-defined. We observe that, in general,
\begin{equation*}
	\partial^\Phi_t f(t)=D^\Phi_t(f(\cdot)-f(0))(t)=D^\Phi_t f(t)-\overline{\nu}(t)f(0).
\end{equation*}
This operator plays a role of a Riemann-Liouville-type formulation of the generalized fractional derivative in \eqref{eq:gfd}. The advantage of using the Riemann-Liouville-type operator consists in the fact that we do not need the specification of the initial value of the function. Whenever such an initial value is $0$, the two generalized fractional derivatives coincide. This is the case, for instance, of $f_L(s;\cdot)$ when $s>0$ and $p_\Phi(\cdot,x;y)$ when $x \not = y$, leading to
\begin{align*}
	D^\Phi_t f_L(s;t)&=\partial_t^\Phi f_L(s;t), \ s,t>0\\
	D^\Phi_t p_\Phi(t,x;y)&=\partial_t^\Phi p_\Phi(t,x;y), \ t>0, \ x\not = y. 
\end{align*}
It is clear that $\partial_t^\Phi p_\Phi(t,x;x)$ is not well-defined since $\lim_{t \downarrow 0}p_\Phi(t,x;x)=+\infty$. Nevertheless, since the evaluation of $D^\Phi_t p_\Phi(t,x;x)$ would not require the value of $p_\Phi(0,x;x)$, we can prove the following statement, whose proof is left in Appendix~\ref{app:RLderpPhi}
\begin{prop}\label{prop:RLderpPhi}
	Under Assumptions \eqref{ass2}, \eqref{orcond} and \eqref{ass4}, 
	\begin{equation}
		D_t^\Phi p_\Phi(t,x;x)=-\frac{1}{\sqrt{2\pi}}\int_{0}^{+\infty}s^{-\frac{1}{2}}\partial_s f_L(s;t)\, ds.
		\label{eq:RLderpPhi}
	\end{equation}
	Furthermore, for any compact set $K \subset \R^+$ it holds
	\begin{equation}\label{eq:boundDphi}
		\sup_{\substack{t \in K \\ x \in \R}}\left|D_t^\Phi p_\Phi(t,x;x)\right|<\infty.
	\end{equation}
\end{prop}
Notice that we have actually shown that for $t>0$
\begin{equation*}
	2D^\Phi_t p_\Phi(t,x;y)=\begin{cases}
		\displaystyle \partial_x^2 p_\Phi(t,x;y) & x\not = y \\[7pt]
		\displaystyle \lim_{|x-y| \to 0}\partial_x^2 p_\Phi(t,x;y) & x=y,
	\end{cases}
\end{equation*}
where the case $x=y$ is not covered by the actual second derivative due to the fact that $\partial_x p_\Phi(t,\cdot;y)$ admits a jump discontinuity, as evidenced in \eqref{prop:secder}. Differently from $\partial_t^\Phi p_\Phi$, this also shows that $D_t^\Phi p_\Phi \in C(\R^+ \times \R^2)$. Actually, we have the following statement.
\begin{prop}\label{prop:exchangeRL}
	Let $f \in C_{\sf b}(\R)$ and consider the function
	\begin{equation*}
		u(t,x)=\int_{\R} f(y)p_\Phi(t,x;y)\, dy.
	\end{equation*}
	Then
	\begin{equation*}
		D^\Phi_t u(t,x)=\int_{\R} f(y)D^\Phi_t p_\Phi(t,x;y)\, dy
	\end{equation*}
	and, in particular, $D^\Phi_t u \in C(\R^+ \times \R)$.
\end{prop}
The proof is given in Appendix~\ref{app:exchangeRL}
}

\textcolor{black}{In the next theorem, we want to show that actually $p_\Phi$ is the fundamental solution of the time-nonlocal heat equation on the entire real line. Notice that such a result follows as a consequence of \cite[Theorem 2.1]{zqc} when $f \in C^2_0(\R)$. Here, we remove the $C^2$ regularity assumption, substituting it with the less restrictive Dini-continuity assumption to retrieve uniqueness of the solution. Precisely, we say that $f \in C(\R)$ is Dini-continuous if there exists a non-decreasing function $\varpi:[0,1] \to \R$ with the following properties:
\begin{itemize}
	\item $\varpi(0)=0$;
	\item $|f(x)-f(y)| \le \varpi(|x-y|)$ for any $(x,y) \in \R^2$ such that $|x-y| \le 1$;
	\item It holds
	\begin{equation*}\int_{0}^1 \frac{\varpi(y)}{y}\, dy<\infty.
	\end{equation*}
\end{itemize}
We refer to $\varpi$ as the modulus of continuity of $f$. We are now ready to state the aforementioned result.
\begin{thm}\label{thm:Dini1}
		Suppose that Assumptions \eqref{ass2}, \eqref{orcond} and \eqref{ass4} hold. Let $f \in C_{{\sf b}}(\R)$. Then the function $u:\R_0^+ \times \R \to \R$
		\begin{equation}\label{eq:utxsol}
			u(t,x)=\int_{\R} p_\Phi(t,x;y)f(y)\, dy
		\end{equation}
		is a solution of the time-nonlocal Cauchy-Dirichlet problem
		\begin{equation}\label{eq:nonlocfull}
			\begin{cases}
				\displaystyle \partial_t^\Phi u(t,x)=\frac{1}{2}\partial_x^2 u(t,x) & (t,x) \in \R^+ \times \R \\[7pt]
				u(0,x)=f(x) & x \in \R
			\end{cases}
		\end{equation}
		in the sense that:
		\begin{itemize}
			\item[(i)] $u \in C(\R_0^+ \times \R)$;
			\item[(ii)] For all $t>0$, it holds $u(t,\cdot) \in C^2(\R)$;
			\item[(iii)] For all $x \in \R$, it holds $\partial_t^\Phi u(\cdot,x) \in C(\R^+)$;
			\item[(iv)] $u$ satisfies \eqref{eq:nonlocfull}.
		\end{itemize}
		Furthermore if $f$ is Dini-continuous, then
		\begin{itemize}
			\item[(v)] for all $x \in \R$, it holds $\partial_t^\Phi u(\cdot,x), \partial_x^2 u(\cdot,x) \in L^1_{\sf loc}(\R^+_0)$.
		\end{itemize}
		Moreover, if $f \in C_0(\R)$, then
		\begin{itemize}
			\item[(vi)] it holds:
			\begin{equation*}
				\lim_{x \to \infty}u(t,x)=0 \qquad \mbox{ locally uniformly with respect to }t \ge 0.
			\end{equation*}
		\end{itemize}
		Finally, if $f \in C_0(\R)$ and is Dini-continuous, then $u$ is the unique solution of \eqref{eq:nonlocfull} satisfying all the conditions (i) to (vi).
	\end{thm}
\begin{proof}
	Let us consider $u$ as in the statement. We already know that $u \in C(\R_0^+ \times \R)$, with $u(0,x)=f(x)$, and $u(t,\cdot) \in C^2(\R)$ for $t>0$ by Proposition \ref{prop:secder}. To prove that \eqref{eq:nonlocfull} holds, let us first observe that
	\begin{equation}\label{eq:CaptoRL}
		\partial_t^\Phi u(t,x)=D_t^\Phi \left(u(t,x)-f(x)\right)=D_t^\Phi u(t,x)-\overline{\nu}(t)f(x).
	\end{equation}
	This already shows that $\partial_t^\Phi u \in C(\R^+ \times \R)$. Now, by Proposition~\ref{prop:exchangeRL} we get that
	\begin{align*}
		D_t^\Phi u(t,x)=\int_{\R} f(y)D_t^\Phi p_\Phi(t,x;y)\, dy=\frac{1}{2}\int_{\R} f(y) \partial_x^2 p_\Phi(t,x;y)\, dy,
	\end{align*}
	where we used the fact that $D_t^\Phi p_\Phi(t,x;\cdot)=\partial_x^2 p_\Phi(t,x;\cdot)$ almost everywhere for fixed $t>0$ and $x \in \R$. Next, by \eqref{eq:der2int}, we have
	\begin{align*}
		D_t^\Phi u(t,x)=\frac{1}{2}\int_{\R} f(y) \partial_x^2 p_\Phi(t,x;y)\, dy=\frac{1}{2}\partial_x^2 u(t,x)+\overline{\nu}(t)f(x).
	\end{align*}
	Plugging the latter into \eqref{eq:CaptoRL}, we get \eqref{eq:nonlocfull}.\\
	Now assume further that $f$ is Dini-continuous and let us show that $\partial_x^2 u(\cdot,x) \in L^1_{\sf loc}(\R^+_0)$, that in turn clearly implies that $\partial_t^\Phi u(\cdot,x) \in L^1_{\sf loc}(\R^+_0)$. To do this, we use \eqref{eq:der2int2}
	\begin{equation*}
		\partial^2_x u(t,x)=\int_{\R}(f(y)-f(x))\partial^2_x p_\Phi(t,x;y)\, dy.
	\end{equation*}
	Precisely, we prove that the function
	\begin{equation*}
		g(t,x):=\int_{\R}|f(y)-f(x)|\left|\partial^2_x p_\Phi(t,x;y)\right|\, dy
	\end{equation*}
	is Laplace transformable, as this clearly implies the desired claim. To do this, notice that for $\lambda>0$ we have
	\begin{align*}
		\int_0^{+\infty}\int_0^{+\infty}e^{-\lambda t}s^{-\frac{5}{2}}(x-y)^2e^{-\frac{(x-y)^2}{2s}}f_L(s;t)\, dt \, ds&=\frac{\Phi(\lambda)(x-y)^2}{\lambda}\int_{0}^{+\infty}s^{-\frac{5}{2}}e^{-\frac{(x-y)^2}{2s}-s\Phi(\lambda)}\, ds\\
		&=2^{\frac{7}{4}}\sqrt{|x-y|}\frac{\Phi^{\frac{7}{4}}(\lambda)}{\lambda}K_{-\frac{3}{2}}\left(|x-y|\sqrt{2\Phi(\lambda)}\right)
	\end{align*}
	and
	\begin{align*}
		\int_0^{+\infty}\int_0^{+\infty}e^{-\lambda t}s^{-\frac{3}{2}}e^{-\frac{(x-y)^2}{2s}}f_L(s;t)\, dt\, ds&=\frac{\Phi(\lambda)}{\lambda}\int_{0}^{+\infty}s^{-\frac{3}{2}}e^{-\frac{(x-y)^2}{2s}-s\Phi(\lambda)}\, ds\\
		&=2^{\frac{5}{4}}\frac{\Phi^{\frac{5}{4}}(\lambda)}{\lambda\sqrt{|x-y|}}K_{-\frac{1}{2}}\left(|x-y|\sqrt{2\Phi(\lambda)}\right)\\
		&=\sqrt{2\pi}\frac{\Phi(\lambda)}{\lambda |x-y|}e^{-|x-y|\sqrt{2\Phi(\lambda)}}.
	\end{align*}
	Hence, for fixed $x \in \R$, we have
	\begin{align*}
		\int_0^{+\infty}e^{-\lambda t}g(t,x)\, dt
		&\le \frac{2^{\frac{7}{4}}}{\sqrt{2\pi}}\frac{\Phi^{\frac{7}{4}}(\lambda)}{\lambda}\int_{\R}|f(y)-f(x)|\sqrt{|x-y|}K_{\frac{3}{2}}\left(|x-y|\sqrt{2\Phi(\lambda)}\right)\, dy\\
		&\qquad +\frac{\Phi(\lambda)}{\lambda}\int_{\R}\frac{|f(y)-f(x)|}{|x-y|}e^{-|x-y|\sqrt{2\Phi(\lambda)}}\, dy
	\end{align*}
	where we also used the fact that $K_{-\alpha}=K_\alpha$. Now recall (see \cite[Chap. 9, Eqs. 9.6.6 and 9.7.2]{abr68}) that there exists two constant $C_1$ and $C_2$ such that
	\begin{align}
		K_{\frac{3}{2}}(y) &\le C_1y^{-\frac{3}{2}}, &\mbox{ for }y \in (0,1] \label{eq:Bessbound1}\\
		K_{\frac{3}{2}}(y) &\le C_2y^{-\frac{1}{2}}e^{-y}, &\mbox{ for } y \in [1,+\infty) \label{eq:Bessbound2}
	\end{align}
	Let $r_\lambda:=\frac{1}{\sqrt{2\Phi(\lambda)}}$ and  $B_{r_\lambda}(x)=\{y \in \R: \ |x-y|< r_\lambda\}$ and split the integrals as follows:
	\begin{align*}
		\int_0^{+\infty}&e^{-\lambda t}g(t,x)\, dt
		\\
		&\le \frac{2^{\frac{7}{4}}}{\sqrt{2\pi}}\frac{\Phi^{\frac{7}{4}}(\lambda)}{\lambda}\int_{B_{r_\lambda}(x)}|f(y)-f(x)|\sqrt{|x-y|}K_{\frac{3}{2}}\left(|x-y|\sqrt{2\Phi(\lambda)}\right)\, dy\\
		&\qquad +\frac{\Phi(\lambda)}{\lambda}\int_{B_{r_\lambda}(x)}\frac{|f(y)-f(x)|}{|x-y|}e^{-|x-y|\sqrt{2\Phi(\lambda)}}\, dy\\
		&\qquad + \frac{2^{\frac{7}{4}}}{\sqrt{2\pi}}\frac{\Phi^{\frac{7}{4}}(\lambda)}{\lambda}\int_{B_{r_\lambda}^c(x)}|f(y)-f(x)|\sqrt{|x-y|}K_{\frac{3}{2}}\left(|x-y|\sqrt{2\Phi(\lambda)}\right)\, dy\\
		&\qquad +\frac{\Phi(\lambda)}{\lambda}\int_{B_{r_\lambda}^c(x)}\frac{|f(y)-f(x)|}{|x-y|}e^{-|x-y|\sqrt{2\Phi(\lambda)}}\, dy\\
		&=:I_1+I_2+I_3+I_4.
	\end{align*}
	From now on, $C$ will denote any generic constant, possibly still depending on $\lambda$, whose value is irrelevant. To handle $I_1$, notice that by \eqref{eq:Bessbound1} we have
	\begin{equation*}
		I_1 \le C\int_{B_{r_\lambda}(x)}\frac{|f(y)-f(x)|}{|x-y|}\, dy \le C\int_{0}^{1}\frac{\varpi(y)}{y}\, dy<\infty,
	\end{equation*}
	where we used the Dini-continuity assumption. The same holds for $I_2$. Concering $I_3$ instead we have by \eqref{eq:Bessbound2}
	\begin{equation*}
		I_3 \le \Norm{f}{L^\infty(\R)}C \int_{r_\lambda}^{+\infty}e^{-y\sqrt{2\Phi(\lambda)}}\, dy <\infty
	\end{equation*}
	and the same holds for $I_4$.\\
	Next assume that $f \in C_0(\R)$. We want to show that (vi) holds. To do this, let us fix $\varepsilon>0$ and consider $M>0$ such that
	\begin{equation*}
		\sup_{|x|>M}|f(x)|<\frac{\varepsilon}{2}.
	\end{equation*}
	Now notice that
	\begin{equation*}
		\left|u(t,x)\right|=\int_{\R}\left|f(y)\right|p_\Phi(t,x;y)\, dy \le \int_{-M}^{M}|f(y)|p_\Phi(t,x;y)\, dy+\frac{\varepsilon}{2}.
	\end{equation*}
	Let $M^\prime>2M$ and assume that $|x|>M^\prime$. Notice that for $y \in [-M,M]$ we have
	\begin{equation*}
		|x-y| \ge |x|-M.
	\end{equation*}
	and then
	\begin{equation*}
		p_\Phi(t,x;y) \le \frac{1}{\sqrt{2\pi}}\int_{0}^{+\infty}s^{-\frac{1}{2}}e^{-\frac{(|x|-M)^2}{2s}}\, f_L(s;t)\, ds.
	\end{equation*}
	Now consider, for $r>M^\prime$, the function
	\begin{equation*}
		g_r(s)=s^{-\frac{1}{2}}e^{-\frac{(r-M)^2}{2s}}
	\end{equation*}
	and notice that its derivative is given by
	\begin{equation*}
		g'_r(s)=\frac{s^{-\frac{5}{2}}}{2}\left((r-M)^2-s\right)e^{-\frac{(r-M)^2}{2s}},
	\end{equation*}
	hence $g_r$ admits a maximum in $s=(r-M)^2$, leading to
	\begin{equation*}
		g_r(s) \le \frac{e^{-\frac{1}{2}}}{r-M}.
	\end{equation*}
	As a consequence we have
	\begin{equation*}
		p_\Phi(t,x;y) \le \frac{1}{(|x|-M)\sqrt{2\pi e}} < \frac{1}{(M^\prime-M)\sqrt{2\pi e}}
	\end{equation*}
	and then
	\begin{equation*}
		\left|u(t,x)\right|< \frac{2M\Norm{f}{L^\infty(\R)}}{(M^\prime-M)\sqrt{2\pi e}}+\frac{\varepsilon}{2}.
	\end{equation*}
	Hence, if we select
	\begin{equation*}
		M^\prime>\max\left\{\frac{M}{\varepsilon\sqrt{2\pi e}}\left(4\Norm{f}{L^\infty(\R)}+\varepsilon\sqrt{2\pi e}\right),2M\right\}
	\end{equation*}
	we get
	\begin{equation*}
		\sup_{\substack{t \ge 0 \\ |x|>M^\prime}}\left|u(t,x)\right|<\varepsilon.
	\end{equation*}
	Since $\varepsilon>0$ is arbitrary, this shows that $\lim_{x \to \pm \infty}u(t,x)=0$ uniformly with respect to $t \ge 0$.\\
	Finally, assume that $f$ is both in $C_0(\R)$ and Dini-continuous. To prove that $u$ is the unique solution satisfying (i) to (vi), notice that if $\widetilde{u}$ is another solution with the same properties, then $w=u-\widetilde{u}$ satisfies
	\begin{equation*}
		\begin{cases}
		\partial_t^\Phi w(t,x)=\frac{1}{2}\partial_x^2 w(t,x), & (t,x) \in \R^+ \times \R \\
		w(0,x)=0 & x \in \R \\
		\lim_{x \to +\infty}w(t,x)=0 & \mbox{ locally uniformly with respect to $t \ge 0$.}
		\end{cases}
	\end{equation*}
	Hence, by Theorem \ref{thm:weakmax}, we have that $w \equiv 0$ and then $u \equiv \widetilde{u}$.
	\end{proof}
	\begin{rmk}
		\label{adessobasta}
		A quite interesting consequence of the previous theorem consists in the fact that we can underline the difference between $D^\Phi_t p_\Phi$ and $\partial_x^2 p_\Phi$ in terms of their action on bounded continuous functions. Indeed, notice that if $f \in C^\infty_c(\R)$, then setting
		\begin{equation*}
			\widetilde{p}_\Phi(t,x;y)=\partial_x p_\Phi(t,x;y)+2\overline{\nu}(t)1_{\R_0^+}(x-y)
		\end{equation*}
		we have
		\begin{align*}
			\int_{\R} \partial_x f(x)\partial_x p_\Phi(t,x;y)\, dx&=\int_{\R} \partial_x f(x)\widetilde{p}_\Phi(t,x;y)\, dx-2\overline{\nu}(t)\int_{y}^{+\infty}\partial_x f(x)\, dx\\
			&=-\int_{\R} f(x)\partial_x \widetilde{p}_\Phi(t,x;y)\, dx+2\overline{\nu}f(y).
		\end{align*}
		In practice, once we recall that $2D^\Phi_t p_\Phi(t,x;y)$ actually coincides with the derivative of $\widetilde{p}_\Phi(t,x;y)$ on all $t>0$ and $x,y \in \R$, we have the distributional relation
		\begin{equation*}
			D^\Phi_t p_\Phi(t,x;y)=\frac{1}{2}\partial_x^2p_\Phi(t,x;y)-\overline{\nu}(t)\delta_{\{x\}}(dy).
		\end{equation*}
	\end{rmk}
}
\subsection{The semi-Markov property}
\label{secsemi}
Note that $X_\Phi$ is not a Markov process. However, it can be embedded in a strong Markov process as follows. Let
\begin{align}
	H(t) := \sigma (L(t)),
\end{align}
which is called the \textit{overshooting} of $\sigma$. Then the process 
$\ll \l X_\Phi(t), H(t)-t \r, t \geq 0 \rr$ is a time homogeneous strong Markov processes with respect to filtration 
\textcolor{black}{$(\cG_t)_{t \ge 0}$, where $\cG_t=\cF_{L(t)}$ for all $t \ge 0$,}
(see \cite[Section 4]{meerstra}) and furthermore it is a Hunt process (see \cite[Theorem 3.1]{meerstra}). Actually, the process $X_\Phi$ still exhibits the Markov property at the random time $H(t)$. Indeed, since $(B, \sigma)$ is Markov additive in the sense of \cite[Definition 1.4]{cinlar2}, 
$\l \Omega, \cF, \l \cG_{t}\r_{t \geq 0}, \overline{\mathcal{M}}, X_\Phi, \dP_x  \r$ is regenerative in the sense of \cite[Definition 2.2 and Example 2.13]{kaspi}, where the regenerative set $\overline{\mathcal{M}}$ is given by the closure of
\begin{align}
	\mathcal{M}(\omega) \,= \, \ll t \in \R^+_0: t= \sigma(u, \omega) \text{ for some } u\geq 0 \rr, \ \omega \in \Omega.
\end{align}
In particular, $X_\Phi$ {exhibits} the strong regeneration property in the sense of \cite[eq. (2.14)]{kaspi}, so that the strong Markov property holds with respect to any 
$\l \mathcal{G}_{t}\r_{t \geq 0}${-Markov} time taking values in $\cM$.
The reader should also consult \cite{jacod} for instructive discussions of renerative systems and semi-Markov property; we also suggest the recent paper \cite{sampling} for a deep discussion about the trajectories of semi-Markov processes obtain via random time-changed with inverse subordinators.

Let us stress that we can also express $H(t)$ in terms of the random set $\cM$ as follows:
\begin{align}\label{eq:discpoint}
	H(t, \omega) = \inf \ll s>t : {s} \in \cM(\omega) \rr, \omega \in \Omega.
\end{align}
In the following we will be interested in the process obtained by killing $X_\Phi$ upon reaching a certain moving boundary $\varphi:\R_0^+ \to \R$. In particular, we need the crossing time to be a Markov time for $X_\Phi$. Hence, let us first introduce the notation
\begin{equation}\label{fpt}
	\mathcal{T}:=\inf\{t \in \R_0^+: \ X_\Phi(t) \ge \varphi(t)\},
\end{equation}
where we set $\inf \emptyset=+\infty$.
\begin{lem}\label{lem:cTM}
	Under Assumption \eqref{ass2}, if $\varphi$ is non-decreasing and continuous then, $\mathbb{P}_y$-almost surely for all $y \in \R$, either $\mathcal{T} \in \cM$ or $\mathcal{T}=\infty$.
\end{lem}
\begin{proof}
	Let us first observe that if $y=X_\Phi(0) \ge \varphi(0)$, then $\mathcal{T}=0 \in \cM$ $\mathbb{P}_y$-almost surely. Hence, let us assume that $y=X_\Phi(0)<\varphi(0)$. Then either $\cT=\infty$ or $X_\Phi(\cT) \ge \varphi(\cT)$. We only need to study the latter case. For the sake of the reader, we omit the dependence on $\omega \in \Omega$. Being both $X_\Phi$ and $\varphi$ continuous (a.s.), the set $\{t \in [0,+\infty): \ X_\Phi(t)-\varphi(t)=0\}$ is closed and 
	\begin{equation}\label{eq:cTmin}
		0<\cT=\min\{t \in [0,+\infty): \ X_\Phi(t)-\varphi(t)=0\}.	
	\end{equation}
	Now suppose that $\cT \not \in \overline{\cM}$. Then we know that $\cT \in (\sigma(w-), \sigma(w))$ for some $w>0$ and
	$$X_\Phi(\sigma(w-))=X_\Phi(\cT)=\varphi(\cT) \ge \varphi(\sigma(w-)).$$
	Combining this information with the fact that $X_\Phi(0)<\varphi(0)$, we get that there exist $t_0 \in (0,\sigma(w-)]$, and then in particular $t_0<\cT$, such that $X_\Phi(t_0)-\varphi(t_0)=0$, which is a contradiction. Hence $\cT \in \overline{\cM}$. 
	
	Now	 let $\btrev{\delta=\left( {\varphi(0)-X_\phi(0)}\right)/{2}}>0$ and consider the sequence of stopping time
	\begin{equation*}
		\mathcal{T}_n=\inf\left\{t \ge 0: \ X_\Phi(t) \ge \varphi(t)-\frac{\delta}{n}\right\}.
	\end{equation*}
	It is clear that $\cT_n \le \cT$. Furthermore, since both $\varphi$ and $X_\phi$ are (a.s.) continuous, we have
	\begin{equation}\label{eq:preTn}
		X_\Phi(\cT_n)=\varphi(\cT_n)-\frac{\delta}{n}<\varphi(\cT_n)	
	\end{equation}
	and then $\cT_n<\cT$ a.s.. Moreover, $\cT_n$ is an non-decreasing sequence of stopping times with $\cT_n<\cT$ and then $\cT_\star=\sup_{n \in \N} \cT_n \le \cT$. However, taking the limit as $n$ tends to infinity in \eqref{eq:preTn} we achieve
	\begin{equation}
		X_\Phi(\cT_\star)=\varphi(\cT_\star),
	\end{equation}
	which, by \eqref{eq:cTmin}, implies $\cT_\star=\cT$. Thus $\cT_n \uparrow \cT$ with $\cT_n < \cT$ and then, since $(X_\Phi(t),H(t)-t)$ is a Hunt process, $\cT$ must be a continuity point for such a process. This implies that $\cT$ is not isolated on the right, otherwise, if $\cT \in \overline{\cM} \setminus \cM$, by \eqref{eq:discpoint} it would be clear that it is a discontinuity point for $H$, which is a contradiction.
\end{proof}
\begin{rmk}
	In case $\mathcal{T}<\infty$, we have $H(\mathcal{T})=\mathcal{T}$.
\end{rmk}
\begin{rmk}
One can actually apply the previous proof to any fixed open set, even if $B$ is a $N$-dimensional Brownian motion, to get that any exit time from an open set is a Markov time for $X_\Phi$. The latter is called semi-Markov property of $X_\Phi$, see \cite[Definition $1$, Chapter $3$]{harlamov}.
\end{rmk}
\section{The killed delayed Brownian motion and a Dynkin-Hunt formula}
\label{sec:killedbm}
As we stated before, we are interested in the process obtained by killing $X_\Phi$ upon reaching $\varphi$, that is to say
\begin{equation*}
	X_\Phi^\dagger(t):=\begin{cases} X_\Phi(t), & t < \mathcal{T}, \\
		\infty, & t \ge \mathcal{T},
	\end{cases}
\end{equation*}
where $\infty \not \in \R$ is a cemetery point. Such a process defines the family of sub-probability measures on $(\R, \mathcal{B}(\R))$ as follows
\begin{equation*}
	\mathcal{B}(\R) \ni A \mapsto \mathds{P}_y(X_\Phi^\dagger(t) \in A), \ y \in \R, \ t \ge 0.
\end{equation*}
Now we want to prove that for each $y \in \R$ and $t>0$, these sub-probability measures admit a density and we want to relate such a density with the non-killed one $p_\Phi$ through a Dynkin-Hunt-type formula. To do this, we first need to ensure that $\mathcal{T}$ is a continuous random variable, at least for $y<\varphi(0)$.
\begin{lem}\label{lem:noatoms}
	Under Assumption \eqref{ass2}, for all $y<\varphi(0)$, $\mathcal{T}$ is a continuous random variable, i.e. $\mathds{P}_y(\mathcal{T}=w)=0$ for any $w \ge 0$.
\end{lem}
\begin{proof}
	Just notice that for any $w_0\geq 0$ and $y<\varphi(0)$,
	\begin{equation}
		\dP_y(\cT \in \{w_0\})\le \dP_y(X_\Phi(w_0)=\varphi(w_0))=0.
		\label{noatoms}
	\end{equation} 
\end{proof}
From now on, we denote by $\mu_\varphi(\cdot;y)$ the law of $\mathcal{T}$ under $\dP_y$, i.e. for all $A \in \cB(\R)$ we set $\mu_\varphi(A;y)=\dP_y(\mathcal{T} \in A)$. Now we are ready to prove the following Dynkin-Hunt formula.
\begin{thm}\label{lem:intrep}
	Suppose Assumption \eqref{ass2} holds and $\varphi$ is non-decreasing and continuous. Then, setting 
	\begin{align}
		q_\Phi(t,x;y)=p_\Phi(t,x;y)-\int_0^tp_\Phi(t-w,x;\varphi(w))\mu_\varphi(dw;y) \, t>0, \ (x,y) \in \R \times (-\infty,\varphi(0)),
		\label{reprq}
	\end{align}
	and $q_\Phi(t,x;y)=0$ for $t>0$ and $(x,y) \in \R \times [\varphi(0),+\infty)$, we have, for $t > 0$, $y \in \R$ and $A \in \mathcal{B}(\R)$,
	\begin{align}
		\dP_y \l X_\Phi^\dag(t) \in A \r \, = \, \int_A q_\Phi(t,x;y) dx.
	\end{align}
\end{thm}
\begin{proof}
	{The statement is clear if $y \ge \varphi(0)$ as in this case $\btrev{q_\Phi}(t,x;y)=0$ for any $t \ge 0$ and $x \in \R$. Thus let us focus on the case $y < \varphi(0)$.}
	Let us first observe that for any Borel set $A \in \cB(\R)$ {and any $t \ge 0$ it holds}
	\begin{align}
		\dP_y \l X_\Phi^\dag(t) \in A \r &=  \dP_y(\cT >t, X_\Phi(t) \in A) \notag \\
		&=\dP_y(X_\Phi(t) \in A)-\dP_y(\cT\le t, X_\Phi(t) \in A)\nonumber\\
		&=\dP_y(X_\Phi(t) \in A)-\int_0^t\dP_y(X_\Phi(t) \in A| \cT=w)\mu_\varphi(dw;y),\nonumber
	\end{align}
	where the function $w \in \R_0^+ \mapsto \dP_y(X_\Phi(t) \in A| \cT=w)\dP_y \l \mathcal{T} \in dw \r$ is the one that exists by the Doob-Dynkin Lemma applied to the conditional expectation $\E[1_A(X_\Phi(t))\mid \mathcal{T}]$, i.e. it is a regular conditional probability.
	Since $\varphi$ is non-decreasing {and continuous} we have by Lemma \ref{lem:cTM} that $ \cT  \in {\mathcal{M}}$ and thus $X_\Phi$ has the Markov property at $\cT$ (see Section \ref{secsemi}). Hence
	\begin{align}
		\dP_y \l X_\Phi^\dag(t) \in A \r =&\dP_y(X_\Phi(t) \in A)-\int_0^t\dP_y(X_\Phi(t) \in A| X_\Phi(H(w))=\varphi(w), \mathcal{T}=w)\mu_\varphi(dw;y) \notag \\
		=  &\dP_y(X_\Phi(t) \in A)-\int_0^t\dP_{\varphi(w)}(X_\Phi(t-w) \in A) \mu_\varphi(dw;y)  ,
	\end{align}
	where in the second equality we used the same argument as in \cite[Theorem 6.11]{Brownian}, and then, by Fubini's theorem
	\begin{align}
		\dP_y \l X_\Phi^\dag(t) \in A \r =\int_{A} p_\Phi(t,x;y)dx-\int_{A}\int_0^tp_\Phi(t-w,x;\varphi(w))\mu_\varphi(dw;y) dx.
	\end{align}
	Being $A$ any Borel set and $t \ge 0$ arbitrary, we {conclude the proof.}
\end{proof}
Before investigating the regularity properties of $q_\Phi$, let us discuss separately the easier case in which $\varphi(t) \equiv c$ for some $c \in \R$. Some properties related to the constant case will be useful when dealing with the general moving boundary case.
\subsection{The constant boundary case}
Let $c \in \R$ and set
\begin{equation}\label{eq:Tcdef}
	\mathcal{T}_c:=\inf\{t \in \R_0^+: \ X_\Phi(t) \ge c\},
\end{equation}
which is a particular case of \eqref{fpt} when $\varphi(t)=c$ for all $t \ge 0$. We are first interested in the distribution of $\mathcal{T}_c$. Notice that, for $y<c$,
\begin{equation*}
	\dP_y\left(\mathcal{T}_c > t\right)=\dP_y\left(\max_{0 \le s \le t}X_\Phi(s) \le c\right)=\dP_0\left(\max_{0 \le s \le t}X_\Phi(s) \le c-y\right),
\end{equation*}
hence, it could be useful to prove the analogous of the reflection principle for the Brownian motion on $X_\Phi$.
\begin{prop}
	\label{refprinc}
	We have that, for any $c>0$,
	\begin{align}
		\dP_0 \l \max_{0 \leq s \leq t} X_\Phi(s) >c\r = 2 \dP_0 \l X_\Phi(t) > c \r.
	\end{align}
\end{prop}
\begin{proof}
	First of all, notice that $\max\limits_{0 \le s \le t}X_\Phi(s) >c$ if an only if $\max\limits_{0 \le s \le L(t)}B(s) >c$, since the maximum actually coincide, as $X_\Phi$ is the composition of $B$, that is continuous, with $L$, that is continuous and non-decreasing. By conditioning on $L(t)$ and using the independence of $B$ and $L$ we have that
	\begin{align*}
		\dP_0 \l \max_{0 \leq s \leq t} X_\Phi (s) > c \r  = & \dP_0 \l \max_{0 \leq w \leq L(t)} B(w) > c \r \\ 
		=&\mathds{E}_0\left[\dP_0 \l \max_{0 \leq w \leq L(t)} B(w) > c \mid L(t) \r\right] \\
		=& \int_0^\infty \dP_0 \l \max_{0 \leq w \leq s} B(w) > c  \r \, f_L(s, t)  \,ds\\
		=& 2\int_0^\infty \dP_0 \l B(s) > c \r \, f_L(s, t) \, ds \\=& 2\dP_0 \l X_\Phi(t) > c \r.
	\end{align*}
\end{proof}
Hence, as a consequence, we have, for $y<c$,
\begin{equation}\label{eq:Tc}
	\dP_y(\mathcal{T}_c \le t)=2\dP_0(X_\Phi(t)>c-y).
\end{equation}
Let us now investigate the regularity of $\mathcal{T}_c$. To do this we recall that, as done in \cite{annals2020}, if we set
\begin{equation}\label{eq:TcB}
	T_c:=\inf\{t \in \R_0^+: \ B(t) \ge c\}
\end{equation}
we have that $\mathcal{T}_c=\sigma(T_c)$. Using this relation, we can prove the following result.
\begin{prop}\label{prop:44}
	Under Assumptions \eqref{ass2} and \eqref{orcond}, the r.v. $\mathcal{T}_c$ is absolutely continuous with density $p_{\mathcal{T}_c}(\cdot;y)$. Furthermore, for all $T>0$ and any compact sets $I,J \subset \R$ such that $z_1<z_2$ for all $z_1 \in I$ and $z_2 \in J$,
	\begin{align}\label{eq:Sc}
		S(I,J)&:=\sup_{c \in J}\sup_{y \in I}\sup_{0 \le t \le T}p_{\mathcal{T}_c}(t;y)<\infty
	\end{align}	
\end{prop}
\begin{proof}
	Let us observe that since Assumptions \eqref{ass2} and \eqref{orcond} hold, we know that $\sigma$ admits density and then the proof of \cite[Theorem 2.8]{annals2020} can be applied analogously. As a consequence $\mathcal{T}_c$ is absolutely continuous. Furthermore, since $\mathcal{T}_c=\sigma(T_c)$, then
	\begin{equation*}
		\E_y\left[e^{i\xi \mathcal{T}_c}\right]=\E_y\left[\E_y\left[e^{i\xi \sigma(T_c)} \mid T_c\right]\right]=\E_y\left[\textcolor{black}{e^{-T_c\Psi(\xi)}}\right]=\int_{0}^{+\infty}e^{-\Psi(\xi)s}p_{T_c}(s;y)\, ds,
	\end{equation*}
	where
	\begin{equation}\label{eq:pTc}
		p_{T_c}(s;y)=\frac{c-y}{\sqrt{2\pi s^3}}e^{-\frac{(c-y)^2}{2s}}.
	\end{equation}
	Taking the modulus we have
	\begin{equation*}
		\left|\E_y\left[e^{i\xi \mathcal{T}_c}\right]\right| \le \int_{0}^{+\infty}e^{-\Re(\Psi(\xi))s}p_{T_c}(s;y)\, ds
	\end{equation*}
	and then
	\begin{align*}
		\int_{\R}\left|\E_y\left[e^{i\xi \mathcal{T}_c}\right]\right|\, d\xi &\le 2\int_0^{+\infty}\int_{0}^{+\infty}e^{-\Re(\Psi(\xi))s}p_{T_c}(s;y)\, ds d\xi\\
		&\le 2\int_{0}^{+\infty}\int_0^{M_\gamma}e^{-\Re(\Psi(\xi))s}p_{T_c}(s;y)\, d\xi \, ds +2\int_{0}^{+\infty}\int_{M_\gamma}^{+\infty}e^{-C_\gamma \xi^{2-\gamma}s}p_{T_c}(s;y)\, d\xi\, ds \\
		&\le 2M_\gamma+\frac{2\Gamma\left(\frac{1}{2-\gamma}\right)}{(2-\gamma)C_\gamma^{\frac{1}{2-\gamma}}}\E_y\left[T_c^{-\frac{1}{2-\gamma}}\right]<\infty,
	\end{align*}
	where we used \eqref{eq:pTc} to guarantee that $\E_y\left[T_c^{-\frac{1}{2-\gamma}}\right]<\infty$. As a consequence, we can write
	\begin{equation*}
		p_{\mathcal{T}_c}(t;y)=\frac{1}{2\pi}\int_{\R}\int_{0}^{+\infty}e^{-\Psi(\xi)s-i\xi t}p_{T_c}(s;y)\, ds \, d\xi.
	\end{equation*}
	Thus
	\begin{align*}
		\left|p_{\mathcal{T}_c}(t;y)\right| &\le \frac{1}{\pi}\int_{0}^{+\infty}\int_{0}^{+\infty}e^{-\Re(\Psi(\xi))s}p_{T_c}(s;y)\, ds \, d\xi\le \frac{M_\gamma}{\pi}+\frac{\Gamma\left(\frac{1}{2-\gamma}\right)}{\pi(2-\gamma)C_\gamma^{\frac{1}{2-\gamma}}}\E_y\left[T_c^{-\frac{1}{2-\gamma}}\right].
	\end{align*}
	Now notice that if $c \in J$ and $y \in I$, then $c-y \ge \min_{z_1 \in I, \ z_2 \in J}|z_1-z_2|=\varepsilon(I,J)$. Hence
	\begin{align*}
		\E_y\left[T_c^{-\frac{1}{2-\gamma}}\right]&=\frac{1}{\sqrt{2\pi}}\int_{0}^{+\infty}(c-y)s^{-\frac{3}{2}-\frac{1}{2-\gamma}}e^{-\frac{(c-y)^2}{2s}}\, ds \\
		&\le 
		\sqrt{2}e^{-\frac{1}{2-\gamma}}\left(\frac{4}{2-\gamma}\right)^{\frac{1}{2-\gamma}}(c-y)^{-\frac{2}{2-\gamma}}\int_{0}^{+\infty}\frac{(c-y)}{\sqrt{4\pi s^3}}e^{-\frac{(c-y)^2}{4s}}\, ds\\
		& \le \sqrt{2}e^{-\frac{1}{2-\gamma}}\left(\frac{4}{2-\gamma}\right)^{\frac{1}{2-\gamma}}\left(\varepsilon(I,J)\right)^{-\frac{2}{2-\gamma}},
	\end{align*}
	that ends the proof.
\end{proof}

It is relevant that, for $\varphi(t) \equiv c$, we can explicitly determine $q_\Phi$. 
\begin{prop}\label{prop:refltc}
	Suppose Assumption \eqref{ass2} holds and $\varphi(t)=c$ for all $t \ge 0$. Then, for $t>0$ and $x,y<c$ it holds
	\begin{equation*}
		q_\Phi(t,x;y)=p_\Phi(t,x;y)-p_\Phi(t,2c-x;y)
	\end{equation*}
\end{prop}
\begin{proof}
	Let
	\begin{equation*}
		B^\dagger(t)=\begin{cases}
			B(t) & t < T_c \\
			\infty & t \ge T_c
		\end{cases}
	\end{equation*}
	where $T_c$ is defined in \eqref{eq:TcB}. Notice that
	\begin{equation*}
		B^\dagger(L(t))=\begin{cases}
			X_\Phi(t) & L(t) < T_c \\
			\infty & L(t) \ge T_c.
		\end{cases}
	\end{equation*}
	Notice that $L(t)<T_c$ if and only if $\sigma(T_c-) > t$. Furthermore, since $\sigma$ is stochastically continuous and $T_c$ is independent of it, $\sigma(T_c-)=\sigma(T_c)$ a.s. Hence, a.s.
	\begin{equation*}
		B^\dagger(L(t))=\begin{cases}
			X_\Phi(t) & t<\sigma(T_c) \\
			\infty & t \ge \sigma(T_c),
		\end{cases}
	\end{equation*}
	i.e., $B^\dagger(L(t))=X_\Phi^\dagger(t)$ a.s. Now let $q:\R^+ \times \R^2 \to \R$ be such that for all $A \in \cB(\R)$ it holds
	\begin{equation*}
		\dP_y(B^\dagger(t) \in A)=\int_{A} q(t,x;y)\, dx.
	\end{equation*}
	Recall also that for $x,y<c$ (see \cite[Chapter 2, Problem 8.6]{karatzas})
	\begin{align*}
		q(t,x;y)&=p(t,x;y)-p(t,2c-x;y).
	\end{align*}
	Hence, by a simple conditioning argument, recalling that $L(t)$ and $B^\dagger$ are independent, we have, if $A \subset (-\infty,c)$ and $y<c$,
	\begin{align*}
		\dP_y(X_\Phi^\dagger(t) \in A)=\int_{A} \int_0^{+\infty} q(s,x;y) f_L(s;t)\, ds \, dx
		=\int_{A} \left(p_\Phi(t,x;y)-p_\Phi(t,2c-x;y)\right) dx
	\end{align*}
	 that ends the proof.
\end{proof}
\begin{rmk}
	The same result can be obtained by employing the reflection principle in Proposition \ref{refprinc} together with the fact that $X_\Phi$ is Markov in $\mathcal{T}_c$, as a consequence of Lemma \ref{lem:cTM}. In this case, the proof proceeds exactly as in the classical case of the Brownian motion.
\end{rmk}
\textcolor{black}{Before proceeding with the non-constant boundary case, let us study the process $(\mathcal{T}_c)_{c \ge 0}$.
\begin{prop}\label{prop:subordinatorT}
	Under $\dP_0$, the process $(\mathcal{T}_t)_{t \ge 0}$ is a subordinator with Laplace exponent $\widetilde{\Phi}(\lambda)=\sqrt{2\Phi(\lambda)}$.
\end{prop}
\begin{proof}
	First notice that we have $\mathcal{T}_0=0$. Next, since $\mathcal{T}_t$ are regeneration times for $X_\Phi$ and thus $X_\Phi$ exhibits the strong Markov property in $\mathcal{T}_t$, it is clear that if $t_1<t_2 \le t_3<t_4$, then $\mathcal{T}_{t_2}-\mathcal{T}_{t_1}$ is independent of $\mathcal{T}_{t_4}-\mathcal{T}_{t_3}$. Next, consider $t_1<t_2$. Then, recalling that by definition $\mathcal{T}_t=\sigma(T_t-)$, where $\sigma$ is the subordinator with inverse $L$ such that $X_\Phi(t)=B(L(t))$, we have
	\begin{align*}
		\E_0[e^{-\lambda (\mathcal{T}_{t_2}-\mathcal{T}_{t_1})}]&=\E_0[e^{-\lambda \left(\sigma(T_{t_2})-\sigma(T_{t_1})\right)}]=\E_0\left[\E_0\left[e^{-\lambda \left(\sigma(T_{t_2}-T_{t_1})\right)} \mid B(t), \ t \ge 0\right]\right]\\
		&=\E_0\left[e^{-(T_{t_2}-T_{t_1})\Phi(\lambda)}\right]=e^{-(t_2-t_1)\sqrt{2\Phi(\lambda)}},
	\end{align*}
	ending the proof.
\end{proof}
}

\subsection{The non-decreasing and continuous boundary case}
In this section we discuss some regularity properties of $\mu_\varphi(\cdot;y)$ and $q_\Phi$ for a more general moving boundary $\varphi$. First of all, we prove the following regularity result concerning $\mu_\varphi(\cdot;y)$.
\begin{thm}\label{thm:density}
	Suppose that Assumptions \eqref{ass2} and \eqref{orcond} hold. Then the law $\mu_\varphi(\cdot;y)$ of $\mathcal{T}$ under $\dP_y$ is absolutely continuous with respect to the Lebesgue-Stieltjes measure $d(\varphi+\iota)$, where $\iota(t)=t$ for any $t \ge 0$. Furthermore, if $\varphi$ is absolutely continuous, then $\mathcal{T}$ is an absolutely continuous random variable and, if $\varphi$ is locally Lipschitz, the density $p_{\mathcal{T}}(\cdot;y)$ of $\mathcal{T}$ satisfies
	\begin{equation}\label{eq:STIT}
		S_{\mathcal{T}}(I,T)=\sup_{t \in [0,T], y \in I}p_{\mathcal{T}}(t,y)<\infty.
	\end{equation}
\end{thm}
\begin{proof}
	Throughout the proof, we denote $J=[\varphi(0),\varphi(T)]$. First of all, notice that, since $\mathcal{T}$ is a continuous random variable, then $\mu_\varphi([s,t];y)=\mu_\varphi((s,t];y)=\mu_\varphi([s,t);y)=\mu_\varphi((s,t);y)$ for all $0\le a<b$. Hence, let us focus on the measure $\mu_\varphi((s,t];y)$. Now let $\omega \in \Omega$ such that $\mathcal{T}(\omega) \in (s,t]$. Then $\mathcal{T}_{\varphi(s)}(\omega) \le \mathcal{T}(\omega) \le \mathcal{T}_{\varphi(t)}(\omega)$. Hence $\mathcal{T}_{\varphi(s)}(\omega) \le t$ and $\mathcal{T}_{\varphi(t)}(\omega)>s$. Hence
	\begin{align*}
			\mu_\varphi((s,t];y) \le \dP_y(\mathcal{T}_{\varphi(s)} \le t, \ \mathcal{T}_{\varphi(t)}>s)
		=&\dP_y(s \le \mathcal{T}_{\varphi(s)} \le t, \ \mathcal{T}_{\varphi(t)}>s)+\dP_y(\mathcal{T}_{\varphi(s)} < s, \ \mathcal{T}_{\varphi(t)}>s)\\
		=:&P_1(t,s)+P_2(t,s).
	\end{align*}
	Concerning $P_1(t,s)$ we have
	\begin{equation*}
		P_1(t,s) \le \dP_y(s \le \mathcal{T}_{\varphi(s)} \le t) \le S(I,J)(t-s).
	\end{equation*}
	Now we handle $P_2(t,s)$. We have
	\begin{equation*}
		P_2(t,s)=\E_y\left[\mathbf{1}_{[0,s)}(\mathcal{T}_{\varphi(s)})\E_y\left[\mathbf{1}_{[s-\mathcal{T}_{\varphi(s)},+\infty)}(\mathcal{T}_{\varphi(t)}-\mathcal{T}_{\varphi(s)}) \mid \mathcal{G}_{\mathcal{T}_{\varphi(s)}}\right]\right].
	\end{equation*}
	By Lemma \ref{lem:cTM}, we know that $X_\Phi$ satisfies the Markov property in $\mathcal{T}_{\varphi(s)}$ and $(X_\Phi(t),H(t)-t)$ is time-homogeneous, then, recalling that $H(\mathcal{T}_{\varphi(s)})=\mathcal{T}_{\varphi(s)}$, we get
	\begin{equation*}
		\E_y\left[\mathbf{1}_{[s-\mathcal{T}_{\varphi(s)},+\infty)}(\mathcal{T}_{\varphi(t)}-\mathcal{T}_{\varphi(s)}) \mid \mathcal{G}_{\mathcal{T}_{\varphi(s)}}\right]=G(\mathcal{T}_{\varphi(s)};s),
	\end{equation*}
	where, since $\mathcal{T}_{\varphi(t)}$ is absolutely continuous, for $\tau \in [0,s)$
	\begin{align*}
		G(\tau;s)&:=\E_{\varphi(s)}\left[\mathbf{1}_{[s-\tau,+\infty)}(\mathcal{T}_{\varphi(t)})\right]=\dP_{\varphi(s)}(\mathcal{T}_{\varphi(t)}\ge s-\tau)\\
		&=1-\dP_{\varphi(s)}(\mathcal{T}_{\varphi(t)}\le s-\tau)=1-2\dP_{0}(X_\Phi(s-\tau)>\varphi(t)-\varphi(s))
	\end{align*}
	where we used the reflection principle. However, notice also that
	\begin{multline*}
		2\dP_{0}(X_\Phi(s-\tau)>\varphi(t)-\varphi(s))\\=\dP_{0}(X_\Phi(s-\tau)>\varphi(t)-\varphi(s))+\dP_{0}(X_\Phi(s-\tau)<\varphi(s)-\varphi(t))\\
		=\dP_{0}(|X_\Phi(s-\tau)|>\varphi(t)-\varphi(s))
	\end{multline*}
	and then
	\begin{align*}
		G(\tau;s)&=\dP_{0}(|X_\Phi(s-\tau)| \le \varphi(t)-\varphi(s))=\int_{\varphi(s)-\varphi(t)}^{\varphi(t)-\varphi(s)}p_{\Phi}(s-\tau,z;0)\, dz.
	\end{align*}
	As a consequence, we finally achieve
	\begin{equation*}
		P_2(t,s)=\int_0^s\int_{\varphi(s)-\varphi(t)}^{\varphi(t)-\varphi(s)}p_{\Phi}(s-\tau,z;0) p_{\mathcal{T}_{\varphi(s)}}(\tau;y)\, dz \, d\tau.
	\end{equation*}
	By \eqref{eq:Sc}, we get
	\begin{align*}
		P_2(t,s) &\le \frac{\sqrt{2}S(I,J)}{\sqrt{\pi}}(\varphi(t)-\varphi(s))\int_0^sU_{-\frac{1}{2}}(s-\tau)\, d\tau \\
		&\le \frac{\sqrt{2}S(I,J)}{\sqrt{\pi}}(\varphi(t)-\varphi(s))\int_0^TU_{-\frac{1}{2}}(\tau)\, d\tau <\infty,
	\end{align*}
	where we used that
	\begin{equation*}
		p_\Phi(s-\tau,z;0)=\E_0\left[\frac{1}{\sqrt{2\pi L(s-\tau)}}e^{-\frac{z^2}{2L(s-\tau)}}\right] \le \frac{U_{-\frac{1}{2}}(s-\tau)}{\sqrt{2\pi}} 
	\end{equation*}
	and $U_{-\frac{1}{2}} \in L^1_{\rm loc}(\R_0^+)$ by Proposition \ref{prop:U12L1}.
	
	Resuming, we have shown that there for all $T>0$ and any compact $I \subset (-\infty,\varphi(0))$, it holds
	\begin{equation*}
		\mu_\varphi((s,t];y) \le C(I,T)(t-s+\varphi(t)-\varphi(s)), \ \forall 0 \le s<t \le T, \ y \in I.
	\end{equation*}
	This also shows that, setting $\iota(t)=t$ and $d(\varphi+\iota)$ the distributional derivative of $\varphi+\iota$, which is a Radon measure, it holds $\mu_\varphi(\cdot;y)$ is absolutely continuous with respect to $d(\varphi+\iota)$ and the Radon-Nikodym derivative of $\mu_\varphi(\cdot;y)$ with respect to $d(\varphi+\iota)$ is a locally bounded function. In particular, if $\varphi$ is an absolutely continuous function, then $\mathcal{T}$ is an absolutely continuous random variable and if $\varphi$ is locally Lipschitz then its density $p_{\mathcal{T}}(\cdot;y)$ satisfies for any $T>0$ and any compact set $I \subset (-\infty,\varphi(0))$
	\begin{equation*}
		\sup_{t \in [0,T], y \in I}p_{\mathcal{T}}(t;y) \le C(I,T)(1+{\sf Lip}(\varphi;T))<\infty,
	\end{equation*}
	where
	\begin{equation*}
		{\sf Lip}(\varphi;T)=\sup_{\substack{t,s \in [0,T] \\ t \not = s}}\frac{|\varphi(t)-\varphi(s)|}{|t-s|}.
	\end{equation*}
\end{proof}
With the previous result in mind, we can prove the continuity of $q_\Phi(\cdot,\cdot;y)$.
\begin{thm}\label{thm:cont}
	Suppose that Assumptions \eqref{ass2} and \eqref{orcond} hold and $\varphi$ is non-decreasing, continuous and locally Lipschitz. Then $q_\Phi(\cdot,\cdot;y)$ is continuous on $\R^+ \times \R$ and $q_\Phi(t,x;y)=0$ for all $x \ge \varphi(t)$ and $t>0$.
\end{thm}
\begin{proof}
	Let us denote
	\begin{equation*}
		r_\Phi(t,x;y)=\int_0^t p_\Phi(t-w,x;\varphi(w))p_{\mathcal{T}}(w;y)\, dw, \ (t,x) \in \R^+ \times \R,
	\end{equation*}
	where the density $p_{\mathcal{T}}(\cdot;y)$ exists by Theorem \ref{thm:density}.
	It is clear that if $r_\Phi(\cdot,\cdot;y)$ is continuous in $(t,x)$, then also $q_\Phi(\cdot,\cdot;y)$ is continuous in $(t,x)$ and vice versa. First, consider the case of a point $(t,x) \in \R^+ \times \R$ with $x \not = \varphi(t)$. Then, if we consider a sequence $(t_n,x_n) \to (t,x)$, it is not difficult to check that there exists a compact set $K$ such that $(0,0) \not \in K$ and $(t_n-w,x_n-\varphi(w)) \in \mathring{K}$ for $n \in \N$ \textcolor{black}{large enough} and $w \in [0,t_n]$. Hence we have
	\begin{equation*}
		p_\Phi(t_n-w,x_n;\varphi(w))\mathbf{1}_{(0,t_n)}(w) \le \sup_{(\tau,\xi) \in K}p_\Phi(\tau,\xi;0)
	\end{equation*}
	and then $\lim_{n \to \infty}r_\Phi(t_n,x_n;y)=r_\Phi(t,x;y)$ by the dominated convergence theorem. 
	
	Now we need to handle the case $(t,x)=(t,\varphi(t))$. This requires several steps. First of all, let $x^\prime \in \R$, \textcolor{black}{$T \ge t$} and consider
	\begin{align}\label{eq:firstboundcont}
		\begin{split}
		|r_\Phi(t,x^\prime;y)&-r_\Phi(t,\varphi(t);y)|\\
		&\le \int_0^t \int_0^{+\infty}|p(s,x^\prime;\varphi(w))-p(s,\varphi(t);\varphi(w))|f_L(s;t-w)\, p_{\mathcal{T}}(w;y)ds \, dw \\
		& \le S_{\mathcal{T}}(\{y\},T)\int_0^t \int_0^{+\infty}\left(\int_{m(w)}^{M(w)}\frac{|\xi|}{\sqrt{2\pi}s^{\frac{3}{2}}}e^{-\frac{\xi^2}{2s}}\, d\xi\right) f_L(s;t-w)\, ds \, dw,	
		\end{split}
	\end{align}
	where $m(w)=(x^\prime-\varphi(w))\wedge (\varphi(t)-\varphi(w))$, $M(w)=(x^\prime-\varphi(w))\vee (\varphi(t)-\varphi(w))$ and we used \eqref{eq:STIT}. Now fix $\varepsilon \in \left(0,\frac{1}{2}\right)$ and recall that, for any $\lambda>0$ and $\tau \ge 0$, it holds
	\begin{equation*}
		\tau^{1-\varepsilon}e^{-\lambda \tau} \le \left(\frac{1-\varepsilon}{\lambda}\right)^{1-\varepsilon}e^{-(1-\varepsilon)}.
	\end{equation*}
	Using $\tau=s^{-1}$ and $\lambda=\xi^2$, we get
	\begin{equation*}
		\frac{|\xi|}{\sqrt{2\pi}s^{\frac{3}{2}}}e^{-\frac{\xi^2}{2s}} \le \frac{2^{1-\varepsilon}(1-\varepsilon)^{1-\varepsilon}e^{-(1-\varepsilon)}}{\sqrt{2\pi} s^{\frac{1}{2}+\varepsilon}|\xi|^{1-2\varepsilon}}.
	\end{equation*}
	Notice that \textcolor{black}{$M(w) \ge 0$}
	 for $w \in (0,t)$ and let us distinguish among two cases. If $m(w) \ge 0$, that is to say $x^\prime-\varphi(w) \ge 0$ then
	\begin{align*}
		\int_{m(w)}^{M(w)}\frac{|\xi|}{\sqrt{2\pi}s^{\frac{3}{2}}}e^{-\frac{\xi^2}{2s}}\, d\xi &\le \frac{2^{1-\varepsilon}(1-\varepsilon)^{1-\varepsilon}e^{-(1-\varepsilon)}}{\sqrt{2\pi} s^{\frac{1}{2}+\varepsilon}}\int_{m(w)}^{M(w)}\xi^{-1+2\varepsilon}\, d\xi\\
		&=\frac{2^{1-\varepsilon}(1-\varepsilon)^{1-\varepsilon}e^{-(1-\varepsilon)}}{2\varepsilon \sqrt{2\pi} s^{\frac{1}{2}+\varepsilon}}(M(w)^{2\varepsilon}-m(w)^{2\varepsilon}) \\
		&\le \frac{2^{1-\varepsilon}(1-\varepsilon)^{1-\varepsilon}e^{-(1-\varepsilon)}}{2\varepsilon \sqrt{2\pi} s^{\frac{1}{2}+\varepsilon}}|x^\prime-\varphi(t)|^{2\varepsilon}.
	\end{align*}
	If instead $m(w)<0$, then $m(w)=x^\prime-\varphi(w)$, $M(w)=\varphi(t)-\varphi(w)$ and we have
	\begin{align*}
		\int_{m(w)}^{M(w)}\frac{|\xi|}{\sqrt{2\pi}s^{\frac{3}{2}}}e^{-\frac{\xi^2}{2s}}\, d\xi &\le \frac{2^{1-\varepsilon}(1-\varepsilon)^{1-\varepsilon}e^{-(1-\varepsilon)}}{\sqrt{2\pi} s^{\frac{1}{2}+\varepsilon}}\left(\int_{0}^{\varphi(t)-\varphi(w)}\xi^{-1+2\varepsilon}+\int_{0}^{\varphi(w)-x^\prime}\xi^{-1+2\varepsilon}\right)\\
		&=\frac{2^{1-\varepsilon}(1-\varepsilon)^{1-\varepsilon}e^{-(1-\varepsilon)}}{2\varepsilon \sqrt{2\pi} s^{\frac{1}{2}+\varepsilon}}((\varphi(t)-\varphi(w))^{2\varepsilon}+(\varphi(w)-x^\prime)^{2\varepsilon}) \\
		&\le \frac{2^{1-\varepsilon}(1-\varepsilon)^{1-\varepsilon}e^{-(1-\varepsilon)}}{2\varepsilon \sqrt{2\pi} s^{\frac{1}{2}+\varepsilon}}2^{1-2\varepsilon}|x^\prime-\varphi(t)|^{2\varepsilon}.
	\end{align*}
	Hence, in general
	\begin{align*}
		\int_{m(w)}^{M(w)}\frac{|\xi|}{\sqrt{2\pi}s^{\frac{3}{2}}}e^{-\frac{\xi^2}{2s}}\, d\xi \le \frac{2^{1-3\varepsilon}(1-\varepsilon)^{1-\varepsilon}e^{-(1-\varepsilon)}}{\varepsilon \sqrt{2\pi} s^{\frac{1}{2}+\varepsilon}}|x^\prime-\varphi(t)|^{2\varepsilon}.
	\end{align*}
	Going back to \eqref{eq:firstboundcont}, we get
	\begin{align*}
		|r_\Phi(t,x;y)&-r_\Phi(t,\varphi(t);y)|\\
		&\le 
		S_{\mathcal{T}}(\{y\},T)\frac{2^{1-3\varepsilon}(1-\varepsilon)^{1-\varepsilon}e^{-(1-\varepsilon)}}{\varepsilon \sqrt{2\pi}}|x^\prime-\varphi(t)|^{2\varepsilon}\int_0^t U_{-\frac{1}{2}-\varepsilon}(t-w)\, dw\\
		&\le S_{\mathcal{T}}(\{y\},T)\frac{2^{1-3\varepsilon}(1-\varepsilon)^{1-\varepsilon}e^{-(1-\varepsilon)}}{\varepsilon \sqrt{2\pi}}|x^\prime-\varphi(t)|^{2\varepsilon}\int_0^T U_{-\frac{1}{2}-\varepsilon}(w)\, dw<\infty.
	\end{align*}
	by Proposition \ref{prop:U12L1}. This proves that for all $\varepsilon \in \left(0,\frac{1}{2}\right)$, $y<\varphi(0)$ and $T>0$ there exists a constant $C_\varepsilon(y;T)$ such that
	\begin{equation}\label{eq:control}
		|r_\Phi(t,x^\prime;y)-r_\Phi(t,\varphi(t);y)| \le C_\varepsilon(y;T)|x^\prime-\varphi(t)|
	\end{equation}
	for $t \in (0,T)$ and $x^\prime \in \R$. Furthermore, this guarantees that for fixed $t>0$ the function $r_\Phi(t,\cdot;y)$ is continuous in $\varphi(t)$ and thus also $q_\Phi(t,\cdot;y)$.
	
	Now we prove that $q_\Phi(t,x;y)=0$ for all $x \ge \varphi(t)$. Indeed, fix $t>0$ and let $x>\varphi(t)$. Consider $\delta>0$ such that $x-\delta>\varphi(t)$ and let $A=[x-\delta,x+\delta]$. Then
	\begin{equation*}
		\frac{1}{2\delta}\int_{x-\delta}^{x+\delta}q_\Phi(t,z;y)\, dz=\frac{1}{2\delta}\dP_y(X_\Phi^\dagger(t) \in A)=0
	\end{equation*}
	taking the limit as $\delta \to 0$, by continuity of $q_\Phi(t,\cdot;y)$ in $x$, we have $q_\Phi(t,x;y)=0$ for all $x>\varphi(t)$. Then $q_\Phi(t,\varphi(t);y)=0$ follows again by continuity.
	
	Now we are ready to prove that $r_\Phi(\cdot,\cdot;y)$ is continuous in $(t,\varphi(t))$. For this, let $(t_n,x_n) \to (t,\varphi(t))$ and observe that
	\begin{equation*}
		r_\Phi(t_n,\varphi(t_n);y)=p_\Phi(t_n,\varphi(t_n);y)-q_\Phi(t_n,\varphi(t_n);y)=p_\Phi(t_n,\varphi(t_n);y).
	\end{equation*}
	Taking the limit, we have 
	\begin{equation*}
	\lim_{n \to \infty}r_\Phi(t_n,\varphi(t_n);y)=\lim_{n \to \infty}p_\Phi(t_n,\varphi(t_n);y)=p_\Phi(t,\varphi(t);y)=r_\Phi(t,\varphi(t);y).	
	\end{equation*}
	Hence, it is sufficient to prove that
	\begin{equation*}
		\lim_{n \to \infty}\left|r_\Phi(t_n,x_n;y)-r_\Phi(t_n,\varphi(t_n);y)\right|=0.
	\end{equation*}
	However, since $t_n \to t$, there exists $T>0$ such that $t_n,t<T$ for all $n \in \N$ and then, by \eqref{eq:control}, for some $\varepsilon <\frac{1}{2}$
	\begin{equation*}
		\left|r_\Phi(t_n,x_n;y)-r_\Phi(t_n,\varphi(t_n);y)\right| \le C_\varepsilon(y;T)|x_n-\varphi(t_n)| \to 0,
	\end{equation*}
	ending the proof.
\end{proof}

Concerning the continuity in the $y$ variable, we need some further assumption on the moving boundary and on $\Phi$.
\begin{thm}
		\label{thm314}
	Suppose	Assumptions \eqref{ass2} and \eqref{orcond} hold. Let $\varphi:\R_0^+ \to \R$ be non-decreasing, continuous and locally Lipschitz. Assume further that
	\begin{equation}\label{384}
		\limsup_{t \to \infty}(X_\Phi(t)-\varphi(t))=+\infty, \ \textcolor{black}{\mbox{a.s.}}
	\end{equation}
	and either $\varphi$ is concave or $\lim_{\lambda \to \infty}\frac{\Phi(2\lambda)}{\Phi(\lambda)}<2$. Then $q_\Phi:\R^+ \times \R \times (-\infty,\varphi(0)) \to \R^+_0$ is continuous.
\end{thm}
\begin{proof}
	Let us first show that for fixed $(t,x) \in \R^+ \times \R$ the function $y \in (-\infty,\varphi(0)) \mapsto q_\Phi(t,x;y)$ is continuous. If $x \ge \varphi(t)$, this is obvious, hence let us assume that $x<\varphi(t)$. Arguing as in Theorem \ref{thm:cont}, we only have to show that $r_\Phi(t,x;\cdot)$ is continuous. To prove that $r_\Phi(t,x;\cdot)$ is continuous in $y$, by Helly's Theorem (e.g., \cite[Theorem 4, page 370]{kolfom}), since $\mathcal{T}$ is a continuous random variable by Lemma \ref{lem:noatoms}, {it will be sufficient to prove} that $\mathds{P}_{y_n} \l \mathcal{T} \leq t \r \to \mathds{P}_{y} \l \mathcal{T} \leq t\r$ for every $t \ge 0$ and $\{y_n\}_{n \in \N}\subset (-\infty,\varphi(0))$ such that $y_n \to y \in (-\infty,\varphi(0))$. Note that
	\begin{align}
		\mathds{P}_y \l \mathcal{T} \leq t \r \, = \, \mathds{P}_0 \l \mathcal{T}^y \leq t \r
	\end{align}
	where
	\begin{align}
		\mathcal{T}^y : = \inf \ll t \geq 0 : y+ X_\Phi(t) -\varphi(t){>}0 \rr.
	\end{align}
	\textcolor{black}{From now on, we will denote $\mathds{P}:=\mathds{P}_0$.} Therefore we have to prove that
	\begin{align}
		\mathds{P} \l \mathcal{T}^{y_n} \leq t \r \, \to \, \mathds{P} \l \mathcal{T}^y \leq t \r
		\label{481}
	\end{align}
	whenever $y_n \to y$. To do this, we proceed as follows. Let $D(\R_0^+)$ be the space of real valued c\`{a}dl\`{a}g functions and let $D_\infty(\R_0^+) \subset D(\R_0^+)$ the set containing elements $d \in D(\R_0^+)$ such that $\lim_{t \to +\infty} \sup_{0 \leq s \leq t} d(s) = \infty$. Since $\limsup_{t \to \infty}X_\Phi(t)-\varphi(t)=+\infty$ a.s., then for any $n \in \N$ it holds $y_n+X_\Phi(\cdot)-\varphi(\cdot) \in D_\infty(\R^+_0)$ a.s. Furthermore, it is clear that $y+X_\Phi(\cdot)-\varphi(\cdot) \in C(\R_0^+)$ almost surely and that $y_n +X_\Phi(\cdot) -\varphi(\cdot) \to y+X_\Phi(\cdot)-\varphi(\cdot)$ almost surely in the uniform topology of $C(\R_0^+)$, thus implying that the convergence also holds in $D_\infty(\R_0^+)$ with Skorokhod's $M$ topology (see \cite[Theorem 12.3.2 and Eq. (3.2) in Section 3.3]{whittbook}).
	
	Now we show that the distribution of $\max_{0 \le s \le t}y+X_{\Phi}(s)-\varphi(s)$ cannot admit an atom at $0$. Indeed
	\begin{align}
		\dP &\l \max_{0 \leq w \leq t} (y+X_\Phi(w)-\varphi(w)) =0\r \notag \\ &\leq \, \dP \l y+X_\Phi(w) -\varphi(w) \leq 0 , \forall w \in [\mathcal{T}^y, t] , \cT^y \leq t \r \notag \\
		=\, & \int_0^t \dP \l y+X_\Phi(w) - \varphi(w) \leq 0 , \forall w \in [z,t] \mid \mathcal{T}^y=z \r \dP\l \mathcal{T}^y \in dz \r .
	\end{align}
	We recall that, by Lemma \ref{lem:cTM}, $X_\Phi$ satisfies the strong Markov property at $\mathcal{T}^y$, hence
	\begin{align}
		&\dP \l y+X_\Phi(w) - \varphi(w) \leq 0 , \forall w \in [z,t] \mid \mathcal{T}^y=z \r \notag \\
		= \, &\dP_{\varphi(z)-y} \l y+X_\Phi(w) - \varphi(w+z) \leq 0 , \forall w \in [0,t-z] \r \notag\\
		\leq \, & \dP_{\varphi(z)-y} \l  \exists \overline{\varepsilon} >0: \forall t \in [0,\overline{\varepsilon}], y+X_\Phi(t) - \varphi(z+t) \leq 0 \r.
		\label{this9}
	\end{align}
	Now, since \eqref{384} holds, if $\lim_{\lambda \to \infty}\frac{\Phi(2\lambda)}{\Phi(\lambda)}<2$ we use a restatement of \cite[Proposition III.10]{bertoinb} in terms of the inverse of the subordinator (as for instance done in \cite[Proposition 4.4]{bertoins} for \cite[Theorem III.9]{bertoinb}), otherwise if $\varphi$ is concave, then $t \in \R^+ \mapsto \frac{\varphi(z+t)-\varphi(z)}{t} \in \R$ is non-increasing and we can use \cite[Proposition 4.4]{bertoins}. In both cases, we get
	\begin{align}
		\lim_{t \to 0} \frac{L(t)}{\varphi(z+t)-\varphi(z)} = +\infty, \text{ a.s.,}
	\end{align}
	from which it follows that, for almost all $\omega \in \Omega$, there exists $\overline{\varepsilon}_0(\omega)$ such that $\varphi(t+z)-\varphi(z) \leq L(t, \omega)$, for all $t \in [0, \overline{\epsilon}_0(\omega)]$, {that implies that}, a.s., \textcolor{black}{$y+X_\phi(t, \omega)-\varphi(t+z) \geq y+X_\Phi(t, \omega)-\varphi(z)-L(t, \omega)$}, for all $t \in [0,\overline{\varepsilon}_0(\omega)]$.
	\btrev{Fix now $\varepsilon_0>0$ small enough and use \eqref{this9} for $\bar{\varepsilon} < \varepsilon_0$, to get}
		\begin{align}
		\dP & \l y+X_\Phi(w) - \varphi(w) \leq 0 , \forall w \in [z,t] \mid \mathcal{T}^y=z \r \notag \\ \leq \, &\dP_{\varphi(z)-y} \l \exists \overline{\varepsilon} < \overline{\varepsilon}_0 : \forall t \in [0, \overline{\varepsilon}], y+X_\Phi(t)-\varphi(z+t) \leq 0 \r \notag \\
		\leq \, & \dP_{\varphi(z)-y} \l \exists \overline{\varepsilon} < \overline{\varepsilon}_0 : \forall t \in [0, \overline{\varepsilon}], y+X_\Phi(t)-L(t)-\varphi(z) \leq 0 \r \notag \\
		= \, &\dP_{\varphi(z)} \l \exists \overline{\varepsilon} < \overline{\varepsilon}_0 : \forall t \in [0, \overline{\varepsilon}], B(L(t))-L(t) \leq \varphi(z) \r.
		\label{386}
	\end{align}
	Now notice that
	\begin{align}
		\ll B(L(t, \omega), \omega)-L(t, \omega): t \in [0,\overline{\varepsilon}_0(\omega)] \rr \, = \, \ll B(t, \omega)-t: t \in [0, L(\overline{\varepsilon}_0(\omega), \omega)]  \rr
	\end{align}
	hence
	\begin{align*}
		\dP_{\varphi(z)} & \l \exists \overline{\varepsilon} < \overline{\varepsilon}_0 : \forall t \in [0, \overline{\varepsilon}], B(L(t))-L(t) \leq \varphi(z) \r\\
		&=\dP_{\varphi(z)} \l \exists \overline{\varepsilon} < L(\overline{\varepsilon}_0) : \forall t \in [0, \overline{\varepsilon}], B(t)-t \leq \varphi(z) \r\\
		&=\E_{\btrev{\varphi(z)}}\left[\dP_{\varphi(z)} \l \exists \overline{\varepsilon} < L(\overline{\varepsilon}_0) : \forall t \in [0, \overline{\varepsilon}], B(t)-t \leq \varphi(z) \r \mid \{L(t), \ t \ge 0\}\right]=0,
	\end{align*}
	where we notice that $\overline{\varepsilon}_0$ is a functional of $\{L(t), t \ge 0\}$ and $L(\overline{\varepsilon}_0)>0$. This proves that \eqref{386} is zero.
	
	Now we show that $y+X_{\Phi}(t)-\varphi(t)$ does not assume the value $0$ in the interval $(\cT^y-\varepsilon,\cT^y)$ a.s. Indeed, if for $\omega \in \Omega$ there exists $t<\cT^y$ such that $y+X_{\Phi}(t,\omega)-\varphi(t)=0$, then $\max_{0 \le s \le z}y+X_{\Phi}(s,\omega)-\varphi(s)=0$ for any {$t\leq z < \cT^y$} and, in particular, we can suppose $z \in \mathbb{Q}$. Thus, we have 
	\begin{align*}
		\dP(\exists t < \cT^y: \ y+X_{\Phi}(t)-\varphi(t)=0) 
		&\le \dP(\exists z < \cT^y, \ z \in \mathbb{Q}: \ \max_{0 \le s \le z}y+X_{\Phi}(s)-\varphi(s)=0)\\
		& \le \dP(\exists z \in \mathbb{Q}: \ \max_{0 \le s \le z}y+X_{\Phi}(s)-\varphi(s)=0)\\
		&\le \sum_{z \in \mathbb{Q}}\dP(\max_{0 \le s \le z}y+X_{\Phi}(s)-\varphi(s)=0)=0,
	\end{align*}
	where the latter equality holds since the distribution of $\max_{0\leq s \leq t} y+X_\Phi(s) - \varphi(s)$ does not admit any atom at $0$ for any $t>0$.
	
	Finally, once this has been verified, we have that $\mathds{P}_{0}(\mathcal{T}^{y_n} \le t) \to \mathds{P}_{0}(\mathcal{T}^{y} \le t)$ for all $t>0$ by \cite[Theorem 13.6.4]{whittbook}.
	
	It is worth noticing that actually $\mathds{P}_{0}(\mathcal{T}^{y_n} \le t) \to \mathds{P}_{0}(\mathcal{T}^{y} \le t)$ uniformly, see Appendix \ref{AppUnifConv}.

	Now we are ready to show that $q_\Phi$ is continuous on $\R^+ \times \R \times (-\infty,\varphi(0))$. To do this, fix $(t,x,y) \in \R^+ \times \R \times (-\infty,\varphi(0))$ and let $\{(t_n,x_n,y_n)\}_{n \in \N} \subset \R^+ \times \R \times (-\infty,\varphi(0))$ be such that $(t_n,x_n,y_n) \to (t,x,y)$. Let us first observe that if $x>\varphi(t)$, then we can consider $\varepsilon>0$ such that $x-\varphi(t)>\varepsilon$. Without loss of generality, we can assume that $x_n-\varphi(t)>\frac{\varepsilon}{2}$ and, by continuity of $\varphi$, $\varphi(t_n)-\varphi(t)<\frac{\varepsilon}{4}$, so that
	
	\begin{equation*}
	x_n-\varphi(t_n)=x_n-\varphi(t)-(\varphi(t_n)-\varphi(t))>\frac{\varepsilon}{4}.
	\end{equation*}
	As a consequence, $q_\Phi(t_n,x_n;y_n)=0=q_\Phi(t,x;y)$ for all $n \in \N$. Next, assume that $x<\varphi(t)$. Then, it is sufficient to show that $r_\Phi$ is continuous in $(t,x,y)$. To do this, notice that
	\begin{multline*}
		|r_\Phi(t_n,x_n;y_n)-r_\Phi(t,x;y)|
		 \le |r_\Phi(t_n,x_n;y_n)-r_\Phi(t,x;y_n)|+|r_\Phi(t,x;y_n)-r_\Phi(t,x;y)|=S^1_n+S^2_n.
	\end{multline*}
	We have already shown that $\lim_{n \to \infty}S^2_n=0$. To prove that $\lim_{n \to \infty}S^1_n=0$, let us first consider the case in which $t_n \uparrow t$. Then
	\begin{align*}
		S^1_n &\le \int_0^{t_n}|p_\Phi(t_n-w,x_n;\varphi(w))-p_\Phi(t-w,x;\varphi(w))|\dP_{y_n}(\mathcal{T} \in dw)\\
		&+\int_{t_n}^tp_\Phi(t-w,x;\varphi(w))\dP_{y_n}(\mathcal{T} \in dw)=I^1_n+I^2_n.
	\end{align*}
	Since $x<\varphi(t)$, let $\varphi(t)-x>\varepsilon$ we can assume, without loss of generality, that $\varphi(t)-x_n>\frac{\varepsilon}{2}$ and, by continuity of $\varphi$, $\varphi(t)-\varphi(t_n)<\frac{\varepsilon}{4}$. Then
	\begin{equation*}
		\varphi(t_n)-x_n=\varphi(t)-x_n-(\varphi(t)-\varphi(t_n))>\frac{\varepsilon}{4}>0.
	\end{equation*}
	As a consequence, there exists a compact $K\subset \R^+_0 \times \R^2$ such that all the curves $w \in [0,t_n] \mapsto (t_n-w,x_n,\varphi(w)) \in \R^+_0 \times \R^2$ are contained in $\mathring{K}$ and $(0,0,0) \not \in K$. Since $p_\Phi$ is continuous in $K$, it is also uniformly continuous, hence there exists a non-decreasing modulus of continuity $\varpi:\R_0^+ \to \R_0^+$ such that $\varpi(0)=\lim_{r \downarrow 0}\varpi(r)=0$ and
	\begin{equation*}
		|p_\Phi(t_n-w,x_n;\varphi(w))-p_\Phi(t-w,x;\varphi(w))| \le \varpi(|t_n-t|+|x_n-x|).
	\end{equation*}
	Hence
	\begin{equation*}
		I^1_n \le \varpi(|t_n-t|+|x_n-x|)\dP_{y_n}(\mathcal{T} \le t_n) \le \varpi(|t_n-t|+|x_n-x|) \to 0.
	\end{equation*}
	For $I^2_n$, notice that the curve $w \in [0,t] \mapsto (t-w,x;\varphi(w)) \in \R^+_0 \times \R^2$ lies in the interior of a compact set $K \subset \R^+_0 \times \R^2$ such that \textcolor{black}{$(0,0,0) \not \in K$}. Hence
	\begin{equation*}
		I^2_n \le \max_{(s,\xi;\eta) \in K}p_\Phi(s,\xi;\eta)(\dP_{y_n}(\mathcal{T} \le t)-\dP_{y_n}(\mathcal{T} \le t_n)).
	\end{equation*}
	Since we have that $\dP_{y_n}(\mathcal{T} \le \cdot)$ converge uniformly to $\dP_{y}(\mathcal{T} \le \cdot)$, it is clear that $\lim_{n \to \infty}I^2_n=0$. The same argument holds if $t_n \downarrow t$. Notice that this is enough to prove that $\lim_{n \to \infty}S^1_n=0$. Indeed, let $\{(t_{n_k},x_{n_k};y_{n_k})\}_{k \in \N}$ be the subsequence such that $\lim_{k \to \infty}S^1_{n_k}=\limsup_{n \to \infty}S^1_n$. Then we can consider a further, non-relabelled, subsequence $\{(t_{n_k},x_{n_k};y_{n_k})\}_{k \in \N}$ such that $t_{n_k}$ converges monotonically towards $t$. Then, by the previous argument, $0=\lim_{k \to \infty}S^1_{n_k}=\limsup_{n \to \infty}S^1_n$.
	
	Now let $x=\varphi(t)$. If $x_n=\varphi(t_n)$, then we already know that $q_\Phi(t_n,\varphi(t_n);y_n)=0=q_\Phi(t,\varphi(t);y)$ and then $\lim_{n \to \infty}r_\Phi(t_n,\varphi(t_n);y_n)=r_\Phi(t,\varphi(t);y)$. In general, we split the absolute value as follows:
	\begin{multline*}
		\left|r_\Phi(t_n,x_n;y_n)-r_\Phi(t,x;y)\right|\le \left|r_\Phi(t_n,x_n;y_n)-r_\Phi(t_n,\varphi(t_n);y_n)\right|\\+\left|r_\Phi(t_n,\varphi(t_n);y_n)-r_\Phi(t,\varphi(t);y)\right|=S^3_n+S^4_n.
	\end{multline*}
	Clearly, $\lim_{n \to \infty}S^4_n=0$. To handle $S^3_n$, notice that, arguing as in Theorem \ref{thm:cont}
	\begin{align*}
		S^3_n
		&\le S_{\mathcal{T}}(I,T)\frac{2^{1-3\varepsilon}(1-\varepsilon)^{1-\varepsilon}e^{-(1-\varepsilon)}}{\varepsilon \sqrt{2\pi}}|x_n-\varphi(t_n)|^{2\varepsilon}\int_0^T U_{-\frac{1}{2}-\varepsilon}(s)\, ds
	\end{align*}
	where $\varepsilon \in \left(0,\frac{1}{2}\right)$, $T>\max\{t_n,t\}$  for all $n \in \N$ and $\{y_n\}_{n \in \N} \subset I \subset (-\infty,\varphi(0))$ for a compact set $I$ and $U_p$ is defined in \eqref{Up}. Taking the limits as $n \to \infty$, we finally get $\lim_{n \to \infty}S^3_n=0$.
\end{proof}
\begin{rmk}
	Notice that \eqref{384} is for instance implied by asking that $\varphi$ is bounded. Other conditions that imply the first limit in \eqref{384} can be found by means of the law of the iterated logarithm for $X_\Phi(t)$, as given in \cite[Theorem 4]{magda}. Let us also observe that the condition $\limsup_{\lambda \to \infty}\frac{\Phi(2\lambda)}{\Phi(\lambda)}<2$ is, for instance, implied by asking that $\Phi$ is regularly varying at infinity of order $\alpha \in [0,1)$. Furthermore, it is also equivalent to $\limsup_{\lambda \to \infty}\frac{\lambda \Phi'(\lambda)}{\Phi(\lambda)}<1$ (see \cite[Exercise \btrev{III.7}]{bertoinb}).
\end{rmk}
Concerning the behaviour of $q_\Phi$ as $t \downarrow 0$, we have the following result.
\begin{prop}\label{eq:extendto0}
	Suppose Assumptions \eqref{ass2} and \eqref{orcond}. Assume further that $\varphi$ is non-decreasing and locally Lipschitz. For any fixed $y \in (-\infty,\varphi(0))$ it holds $\lim_{t \downarrow 0}q_\Phi(t,x;y)=0$ locally uniformly with respect to $x \in \R\setminus \{y\}$.	
\end{prop}
\begin{proof}
	Recall that, for $t \le 1$,
	\begin{equation}\label{eq:boundrPhi}
		r_\Phi(t,x;y)=\int_0^t p_\Phi(t-w,x;\varphi(w))\btrev{\mu_\varphi (dw;y)} \le \frac{S_{\mathcal{T}}(\{y\},1)}{\sqrt{2\pi}}\int_0^tU_{-\frac{1}{2}}(w)\, dw
	\end{equation}
where $U_p$ is defined in \eqref{Up}.	Hence $\lim_{t \to 0}r_\Phi(t,x;y)=0$ uniformly in $x$. Recall, however, that we have $\lim_{t \to 0}p_\Phi(t,x;y)=0$ locally uniformly with respect to $x \in \R \setminus \{y\}$, as shown in Proposition \ref{prop:regx}.
\end{proof}
\begin{rmk}
	Actually, with the same proof, we have $\lim_{t \to 0}q_\Phi(t,x;y)=0$ locally uniformly with respect to $(x,y) \in (\R\times (-\infty,\varphi(0)))\setminus {\sf diag}(\R^2)$.
\end{rmk}
The latter allows us to define, for $x \not = y$, $q_\Phi(0,x;y)=0$. In next proposition we study the regularity in the $x$ variable of the sub-probability density $q_\Phi(\cdot,\cdot;y)$.
\begin{prop}\label{prop:derqphix}
	Suppose Assumptions \eqref{ass2} and \eqref{ass4} hold. Let $\varphi: \R_0^+ \to \R$ be continuous, non-decreasing and such that if there exists $0 \le t_1<t_2<\infty$ such that $\varphi(t_1)=\varphi(t_2)$ then $\varphi(t_1)=\varphi(t)$ for all $t \ge t_1$. Then $\partial_x q_\Phi(t,x;y)$ is well-defined and continuous in any $(t,x)$ such that $x \not = \varphi(t),y$ and $t>0$ and, for $x<\varphi(t)$ with $x \not = y$, and $y < \varphi(0)$,
	\begin{equation*}
		\partial_x q_\Phi(t,x;y)=\partial_x p_\Phi(t,x;y)-\int_0^t \partial_x p_\Phi(t-w,x;\varphi(w)) \btrev{\mu_\varphi (dw;y)}.
	\end{equation*}
	 If furthermore Assumption \eqref{orcond} holds, then $\partial_x^2 q_\Phi(t,x;y)$ is well-defined and continuous in any $(t,x)$ such that $x<\varphi(0)$ with $x \not = y$ and $t>0$, and, for $x<\varphi(0)$ with $x \not = y$,
	\begin{equation*}
		\partial^2_x q_\Phi(t,x;y)=\partial^2_x p_\Phi(t,x;y)-\int_0^t \partial^2_x p_\Phi(t-w,x;\varphi(w))\btrev{\mu_\varphi (dw;y)}.
	\end{equation*}
\end{prop}
\begin{proof}
	The case $x>\varphi(t)$ is trivial. We only need to show the statement for $r_\Phi$. To do this, notice that we can construct a compact set $K \subset \R^+_0 \times \R$ such that $w \in [0,t]\mapsto (t-w,x-\varphi(w)) \in \R^+_0 \times \R$ lies in $\mathring{K}$ and $(0,0) \not \in K$. According to Proposition \ref{prop:regx}, $\partial_x p_\Phi(\cdot,\cdot;0)$ is defined in $K \setminus \{(t,0), \ t \ge 0\}$, where $|K \cap \{(t,0), \ t \ge 0\}|=0$, and belongs to $L^\infty(K)$. Hence $\partial_x p_\Phi(t-w,x;\varphi(w))$ is well-defined for all $w \in [0,t]$ such that $x \not = \varphi(w)$. However, notice that if $x \in [\varphi(0),\varphi(t))$ then the preimage $\varphi^{-1}(x)=\{w_x\}$ for some $w_x \in [0,t)$, while if $x < \varphi(0)$, then $\varphi^{-1}(x)=\emptyset$. In any case, $\dP_y(\mathcal{T} \in \varphi^{-1}(x))=0$ since $\mathcal{T}$ is a continuous random variable and then we can state that $\partial_x p_\Phi(t-w,x;\varphi(w))$ is well-defined for $\btrev{\mu_\varphi (dw;y)}$-a.a. $w \in [0,t]$. Furthermore, for fixed $w \in [0,t]$ such that $x \not = \varphi(w)$, the function $\partial_x p_\Phi(t-w,\cdot;\varphi(w))$ is continuous in a neighbourhood of $x$ and, according to the previous argument, it is bounded by $\sup_{(t,x) \in K}|\partial_x p_\Phi(t,x;0)|$, which is independent of $w$. Hence, by a simple application of the dominated convergence theorem, we have
	\begin{equation}\label{eq:derxr}
		\partial_x r_\Phi(t,x;y)=\int_0^t \partial_x p_\Phi(t-w,x;\varphi(w))\btrev{\mu_\varphi (dw;y)}.
	\end{equation}
	Continuity follows by exactly the same argument. Under Assumptions \eqref{orcond} and \eqref{ass4} the same argument holds for the second derivative when $x<\varphi(0)$, \textcolor{black}{getting, in particular,
	\begin{equation}\label{eq:derx2r}
		\partial_x^2 r_\Phi(t,x;y)=\int_0^t \partial_x^2 p_\Phi(t-w,x;\varphi(w))\btrev{\mu_\varphi (dw;y)}.
	\end{equation}}
\end{proof}
\textcolor{black}{As a further step, we would like to guarantee that $\mathcal{T}$ admits a density $p_{\mathcal{T}}(\cdot;y)$ that is continuous on $(0,+\infty)$. To do this, we need some further regularity assumptions on $\varphi$. In the following, we say that $\varphi \in C^1[0,T)$ if there exist $\varepsilon>0$ and a function $\widetilde{\varphi} \in C^1(-\varepsilon,T)$ such that $\widetilde{\varphi}(t)=\varphi(t)$ for all $t \in [0,T)$. Now we are ready to prove the following theorem.
	\begin{thm}
		Suppose that Assumptions \eqref{ass2}, \eqref{orcond} and \eqref{ass4} hold. Let $\varphi:\R^+_0 \to \R$ be non-decreasing, locally Lipschitz and assume further that if there exists $t_1<t_2$ such that $\varphi(t_1)=\varphi(t_2)$, then $\varphi(t)=\varphi(t_1)$ for all $t \ge t_1$. Let
		\begin{equation*}
			t_{\sf const}:=\inf\{t>0: \ \varphi(t)=\varphi(t+1)\}, \quad \mbox{ where } \inf\emptyset=+\infty,
		\end{equation*}
		and assume that $\varphi \in C^1[0,t_{\sf const})$ with $\varphi'(t)>0$ for all $t \in [0,t_{\sf const})$. Then for $y<\varphi(0)$ the law $\mu_\varphi(\cdot;y)$ of $\mathcal{T}$ under $\dP_y$ admits a density $p_{\mathcal{T}}(\cdot;y)$ that is continuous in $(0,t_{\sf const})$. 
\end{thm}}
\begin{proof}
	\textcolor{black}{Let $p_{\mathcal{T}}$ be the density of $\mathcal{T}$ as in Theorem \ref{thm:density}. Fix $t \in [0,t_{\sf const})$ and $y<\varphi(0)$. By the reflection principle, we have
		\begin{equation*}
			A_y(t):=\dP_y(\mathcal{T}_{\varphi(t)} \le t)=2\int_{\varphi(t)-y}^{+\infty}p_{\Phi}(t,z;0)\, dz.
		\end{equation*}
		We now prove that $A_y \in C^1(0,+\infty)$ and we evaluate the derivative. To do this, let $h \in \left(-\frac{t}{2},\frac{t}{2}\right)$ and observe that
		\begin{align}
			\begin{split}
				\frac{A_y(t+h)-A_y(t)}{h}&=\frac{2}{h}\left(\int_{\varphi(t+h)-y}^{+\infty}p_\Phi(t+h,z;0)\, dz-\int_{\varphi(t)-y}^{+\infty}p_\Phi(t,z;0)\, dz\right)\\
				&=\frac{2}{h}\int_{\varphi(t+h)-y}^{\varphi(t)-y}(p_\Phi(t+h,z;0)-p_\Phi(t,z;0))\, dz\\
				&\quad +2\int_{\varphi(t)-y}^{+\infty}\frac{p_\Phi(t+h,z;0)-p_\Phi(t,z;0)}{h}\, dz\\
				&\quad + \frac{2}{h}\int_{\varphi(t+h)-y}^{\varphi(t)-y}p_\Phi(t,z;0)\, dz\\
				&=:2(I_1(h)+I_2(h)+I_3(h))
			\end{split}
		\end{align}
		To handle $I_1$, notice that $z \in \left[\varphi\left(\frac{t}{2}\right)-y,\varphi\left(\frac{3t}{2}\right)-y\right]=:K_1$, where $0 \not \in K_1$. Hence
		\begin{equation*}
			\left|I_1(h)\right| \le \left|\frac{\varphi(t+h)-\varphi(t)}{h}\right|\max_{\substack{z \in K_1 \\ \tau \in [t-|h|,t+|h|]}}|p_\Phi(\tau,z;0)-p_\Phi(t,z;0)|.
		\end{equation*}
		Since $p_\Phi(\cdot,\cdot;0)$ is continuous on $\R^+ \times \R$, and thus uniformly continuous in compact subsets, and $\varphi$ is differentiable in $t$, it holds
		\begin{equation*}
			\lim_{h \to 0}I_1(h)=0.
		\end{equation*}
		Now we need to move to $I_2(h)$. Let us provide a stricter control on $\partial_t p_\Phi$. We have, for $\tau \in [t-|h|,t+|h|]$ and $z \ge \varphi(t)-y>0$.
		\begin{equation}\label{eq:partialtpphi2}
			|\partial_t p_\Phi(\tau,z;0)| \le \int_0^{+\infty} p(w,z;0)|\partial_t f_L(w,\tau)|dw.
		\end{equation}
		Now let $a \in (0,M_\gamma)$. We observe that, by \cite[Proposition 3.1, Item (4)]{bernsteingamma},
		\begin{equation*}
			\lim_{b \to \pm \infty}\frac{|\Phi(a+ib)|}{\sqrt{a^2+b^2}}=0,
		\end{equation*}
		while $b \in \R \mapsto \Phi(a+ib) \in \mathbb{C}$ is a continuous function. Hence, we can take
		\begin{equation*}
			C_1=\sup_{|b| \ge M_\gamma}\frac{|\Phi(a+ib)|}{\sqrt{a^2+b^2}}<\infty, \qquad C_2=\max_{|b| \le M_\gamma}|\Phi(a+ib)|.
		\end{equation*}
		Furthermore, by Assumption \eqref{orcond} and Lemma \ref{lemmaintzero}, we know that for $|b|>M_\gamma$ it holds
		\begin{equation*}
			\Re(\Phi(a+ib)) \ge C_\gamma e^{-t_\gamma a}|b|^{2-\gamma}.
		\end{equation*}
		We have, since $a<M_\gamma$ and $\Re(\Phi(a+ib)) \ge 0$, 
		\begin{align*}
			\int_{-\infty}^{+\infty}|\Phi(a+ib)|e^{-w\Re(\Phi(a+ib))}\, db &\le 2\int_{0}^{M_\gamma}|\Phi(a+ib)|e^{-w\Re(\Phi(a+ib))}\, db+2\sqrt{2}C_1\int_{M_\gamma}^{+\infty}be^{-we^{-t_\gamma a}b^{2-\gamma}}\, db\\
			&\le 2C_2M_\gamma+\frac{2\sqrt{2}C_1}{(2-\gamma)(we^{-t\gamma_a})^{\frac{2}{2-\gamma}}}\Gamma\left(\frac{2}{2-\gamma}\right)<\infty.
		\end{align*}
		Hence, by \cite[Proposition 4.1]{tlms2024}, we get that
		\begin{equation*}
			\left|\partial_t f_L(w,\tau)\right| \le \frac{e^{\frac{3}{2}at}}{2\pi}\int_{-\infty}^{+\infty}|\Phi(a+ib)|e^{-w\Re(\Phi(a+ib))}\, db.
		\end{equation*}
		Plugging the latter into \eqref{eq:partialtpphi2} we have
		\begin{align}\label{eq:upperpartialphi}
			\begin{split}
				|\partial_t p_\Phi(\tau,z;0)| &\le \frac{e^{\frac{3}{2}at}}{2\pi}\int_0^{+\infty}\int_{-\infty}^{+\infty}p(w,z;0)|\Phi(a+ib)|e^{-w\Re(\Phi(a+ib))}\, db \, dw\\
				&=\frac{e^{\frac{3}{2}at}}{2\pi}\int_{-\infty}^{+\infty}|\Phi(a+ib)|\left(\int_0^{+\infty}p(w,z;0)e^{-w\Re(\Phi(a+ib))}\, dw \right)\, db\\
				&=\frac{e^{\frac{3}{2}at}}{2\sqrt{2}\pi}\int_{-\infty}^{+\infty}\frac{|\Phi(a+ib)|}{\sqrt{\Re(\Phi(a+ib))}}e^{-z\sqrt{2\Re(\Phi(a+ib))}}\, db,
			\end{split} 
		\end{align}
		where the last equality comes from the explicit evaluation of the Bessel function $K_{\frac{1}{2}}$, whose details are analogous to the ones of the proof of Proposition \ref{prop:regx} in Appendix \ref{Appreg}. Since the prefactor is constant in $z$, let
		\begin{equation*}
			F(z):=\int_{-\infty}^{+\infty}\frac{|\Phi(a+ib)|}{\sqrt{\Re(\Phi(a+ib))}}e^{-z\sqrt{2\Re(\Phi(a+ib))}}\, db.
		\end{equation*} 
		We now observe that, by Fubini's theorem and direct integration,
		\begin{align}\label{eq:controlFz}
			\begin{split}			
				\int_{\varphi(t)-y}^{+\infty}F(z)\, dz&=\int_{-\infty}^{+\infty}\frac{|\Phi(a+ib)|}{\Re(\Phi(a+ib))}e^{-(\varphi(t)-y)\sqrt{2\Re(\Phi(a+ib))}}\, db\\
				&=2\left(\int_{0}^{M_\gamma}\frac{|\Phi(a+ib)|}{\Re(\Phi(a+ib))}e^{-(\varphi(t)-y)\sqrt{2\Re(\Phi(a+ib))}}\, db+\int_{M_\gamma}^{+\infty}\frac{|\Phi(a+ib)|}{\Re(\Phi(a+ib))}e^{-(\varphi(t)-y)\sqrt{2\Re(\Phi(a+ib))}}\, db\right).
			\end{split}
		\end{align}
		For the first integral, let $\widetilde{\theta}=\arctan\left(\frac{M_\gamma}{a}\right)<\frac{\pi}{2}$ and notice that, by \cite[Proposition 3.7]{librobern}, $\Phi$ maps $\overline{\mathbb{C}(\widetilde{\theta})}$ into itself, hence for $z \in \overline{\mathbb{C}(\widetilde{\theta})}$ it holds $|{\sf Arg}(\Phi(z))| \le \widetilde{\theta}$. Since the arctangent is an odd function, we have for $z \in \overline{\mathbb{C}(\widetilde{\theta})}$
		\begin{equation*}
			\arctan\left(\frac{|\Im(\Phi(z))|}{\Re(\Phi(z))}\right) \le \widetilde{\theta} \qquad \mbox{ that implies }
			|\Im(\Phi(z))| \le \tan(\widetilde{\theta})\Re(\Phi(z)).
		\end{equation*}
		Hence, since for $b \in [0,M_\gamma]$ we have $a+ib \in \overline{\mathbb{C}(\widetilde{\theta})}$, we get	\begin{equation*}
			\frac{|\Phi(a+ib)|}{\Re(\Phi(a+ib))} \le \sqrt{\tan^2(\widetilde{\theta})+1}
		\end{equation*}
		and then
		\begin{equation}\label{eq:controlfirstint}
			\int_{0}^{M_\gamma}\frac{|\Phi(a+ib)|}{\Re(\Phi(a+ib))}e^{-(\varphi(t)-y)\sqrt{2\Re(\Phi(a+ib))}}\, db \le M_\gamma\sqrt{\tan^2(\widetilde{\theta})+1}. 
		\end{equation}
		To handle the second integral, we use instead the fact that $|\Phi(a+ib)| \le C_1\sqrt{a^2+b^2} \le \sqrt{2}C_1b$ and $\Re(\Phi(a+ib)) \ge C_\gamma e^{-t_\gamma a}b^{2-\gamma}$ for $b \ge M_\gamma$ to get
		\begin{align}\label{eq:controlsecond}
			\begin{split}
				\int_{M_\gamma}^{+\infty}\frac{|\Phi(a+ib)|}{\Re(\Phi(a+ib))}&e^{-(\varphi(t)-y)\sqrt{2\Re(\Phi(a+ib))}}\, db \le \frac{\sqrt{2}C_1}{C_\gamma}e^{t_\gamma a} \int_{M_\gamma}^{+\infty}b^{\gamma-1}e^{-\sqrt{C_\gamma}(\varphi(t)-y)e^{-\frac{t_\gamma}{2}a}b^{1-\frac{\gamma}{2}}}\, db\\
				&\le \frac{C_1e^{\frac{2t_\gamma}{2-\gamma}a}}{\sqrt{2}C_\gamma(2-\gamma)(C_\gamma(\varphi(t)-y))^{\frac{2\gamma}{2-\gamma}}}\Gamma\left(\frac{2\gamma}{2-\gamma}\right)<\infty.
			\end{split}
		\end{align}
		Plugging \eqref{eq:controlfirstint} and \eqref{eq:controlsecond} into \eqref{eq:controlFz} we have
		\begin{equation*}
			\int_{\varphi(t)-y}^{+\infty}F(z)\, dz<\infty.
		\end{equation*}
		Now since $p_\Phi(\cdot,z;0) \in C_1(0,+\infty)$ when $z>0$, we can use the mean value theorem to get
		\begin{equation*}
			\frac{p_\Phi(t+h,z;0)-p_\Phi(t,z;0)}{h}=\partial_t p_\Phi(\zeta,z;0),
		\end{equation*}
		where $\zeta$ depends on $t,h,z$ and belongs to the interval $\left[\frac{t}{2},\frac{3}{2}t\right]$. Then we have, by \eqref{eq:upperpartialphi}, 
		\begin{equation*}
			\left|\frac{p_\Phi(t+h,z;0)-p_\Phi(t,z;0)}{h}\right|=|\partial_t p_\Phi(\zeta,z;0)| \le \frac{e^{\frac{3}{2}at}}{2\sqrt{2\pi}}F(z), 
		\end{equation*}
		where the right-hand side is independent of $h$ and integrable in $[\varphi(t)-y,+\infty)$. Hence we can use the dominated convergence theorem to achieve
		\begin{equation*}
			\lim_{h \to 0}I_2(h)=\int_{\varphi(t)-y}^{+\infty}\partial_t p_\Phi(t,z;0)\, dz.
		\end{equation*}
		For $I_3$, it is clear that
		\begin{equation*}
			\lim_{h \to 0}I_3(h)=\varphi'(t)p_\Phi(t,\varphi(t);0).
		\end{equation*}
		This finally proves that
		\begin{equation*}
			\partial_t A_y(t)=-2\varphi'(t)p_\Phi(t,\varphi(t)-y;0)+2\int_{\varphi(t)-y}^{+\infty}\partial_t p_\Phi(t,z;0)\, dz.
		\end{equation*}
		Moreover, notice that the first summand is clearly continuous in $t$. For the second summand, if we fix $t>0$ and we let $h \in \left(-\frac{t}{2},\frac{t}{2}\right)$, then we can split
		\begin{multline*}
			\int_{\varphi(t+h)-y}^{+\infty}\partial_t p_\Phi(t+h,z;0)\, dz-\int_{\varphi(t)-y}^{+\infty}\partial_t p_\Phi(t,z;0)\, dz\\
			=\int_{\varphi(t+h)-y}^{\varphi(t)-y}\partial_t p_\Phi(t+h,z;0)\, dz+\int_{\varphi(t)-y}^{+\infty}(\partial_t p_\Phi(t+h,z;0)-\partial_t p_\Phi(t,z;0))\, dz.
		\end{multline*}
		For the first integral, it is sufficient to notice that $\partial_t p_\Phi(\cdot,\cdot;0)$ is continuous in $\left[\frac{t}{2},\frac{3t}{2}\right]$ and $z \in K_1$ to get
		\begin{equation*}
			\left|\int_{\varphi(t+h)-y}^{\varphi(t)-y}\partial_t p_\Phi(t+h,z;0)\, dz\right| \le \sup_{\substack{z \in K_1 \\ \tau \in \left[\frac{t}{2},\frac{3t}{2}\right]}}\left|\partial_t p_\Phi(\tau,z;0)\right||\varphi(t+h)-\varphi(t)| \to 0
		\end{equation*} 
		as $h \to 0$. On the other hand, for the second integral, we notice that
		\begin{equation*}
			|\partial_t p_\Phi(t+h,z;0)-\partial_t p_\Phi(t,z;0)| \le \frac{e^{\frac{3}{2}at}}{\sqrt{2\pi}}F(z)
		\end{equation*}
		where the right-hand side is integrable in $[\varphi(t)-y,+\infty)$. Hence, we can use the dominated convergence theorem to state that
		\begin{equation*}
			\lim_{h \to 0}\int_{\varphi(t)-y}^{+\infty}(\partial_t p_\Phi(t+h,z;0)-\partial_tp_\Phi(t,z;0))\, dz=0.
		\end{equation*}
		This shows that $\partial_t A_y \in C(0,t_{\sf const})$. }
	
	\textcolor{black}{Now let us notice that if $\mathcal{T}_{\varphi(t)} \le t$, then $\mathcal{T} \le \mathcal{T}_{\varphi(t)}$, i.e. the process must have already hit the boundary. Then it must hit the fixed level $\varphi(t)$ in the remaining time $t-\mathcal{T}$. Recalling that we have shown in Lemma \ref{lem:cTM} that $X_\Phi$ satisfies the Markov property in $\mathcal{T}$, the latter can be expressed as
		\begin{equation}\label{eq:prevolterra}
			A_y(t)=\E_y[\dP_{X_{\mathcal{T}}}(\mathcal{T}_{\varphi(t)}\le t-s)1_{(0,t)}(\mathcal{T})]=\int_0^t \dP_{\varphi(s)}(\mathcal{T}_{\varphi(t)} \le t-s)p_\mathcal{T}(s;y)\, ds.
		\end{equation}
		Let us denote by
		\begin{equation*}
			K(t,s)=\dP_{\varphi(s)}(\mathcal{T}_{\varphi(t)} \le t-s)=2\int_{\varphi(t)-\varphi(s)}^{+\infty}p_\Phi(t-s,z;0)\,dz
		\end{equation*}
		where the second equality follows again by the reflection principle. Now we would like to take the derivative on both sides of \eqref{eq:prevolterra}. To do this, we first observed that $\partial_t K(t,s)$ is well-defined and continuous for $t \in [0,t_{\sf const})$ and $s \in [0,t)$. To do this, actually, it is sufficient to apply exactly the same argument we used for $A_y(t)$, with $h \in \left(\frac{s-t}{2},\frac{t-s}{2}\right)$, so that $t+h \in \left[\frac{t-s}{2}, \frac{3t-s}{2}\right]$ and with $ K_2:=\left[\varphi\left(\frac{t-s}{2}\right)-\varphi(s),\varphi\left(\frac{3t-s}{2}\right)-\varphi(s)\right]$ in place of $K_1$. In particular, one gets
			\begin{equation}\label{eq:partialtK}
			\partial_t K(t,s)=-2\varphi'(t)p_{\Phi}(t-s,\varphi(t)-\varphi(s);0)+2\int_{\varphi(t)-\varphi(s)}^{+\infty}\partial_t p_\Phi(t-s,z;0)\, dz.
	\end{equation}
}
	\textcolor{black}{From now on, let us only consider the case in which $t_{\sf const}<+\infty$; when $t_{\sf const}=+\infty$, we can use any $T>0$ in place of $t_{\sf const}$ and proceed analogously, getting the same statement since the choice of $T>0$ is arbitrary. To evaluate the derivative on the right-hand side of \eqref{eq:prevolterra}, let $h \in \left(0,\frac{t_{\sf const}-t}{2}\right)$. We consider the incremental ratio
		\begin{align*}
			\frac{1}{h}&\left(\int_0^{t+h}K(t+h,s)p_\mathcal{T}(s;y)\, ds-\int_0^{t}K(t,s)p_\mathcal{T}(s;y)\, ds\right)\\
			&=\int_0^{t}\frac{K(t+h,s)-K(t,s)}{h}p_\mathcal{T}(s;y)\, ds+\frac{1}{h}\int_{t}^{t+h}(K(t+h,s)-1)p_\mathcal{T}(s;y)\, ds+\frac{1}{h}\int_{t}^{t+h}p_\mathcal{T}(s;y)\, ds\\
			&=J_1(h)+J_2(h)+J_3(h).
		\end{align*}
		Let us first handle $J_1(h)$. By the mean value theorem, we have
		\begin{equation*}
			\frac{K(t+h,s)-K(t,s)}{h}=\partial_t K(\zeta,s)
		\end{equation*}
		for some $\zeta$ depending on $t,h,s$ and belonging to $\left[t,\frac{t+t_{\sf const}}{2}\right]$. The derivative is well-defined for all $s<t$ since $\zeta>t$. Now we need to provide a bound on $\partial_t K(\zeta,s)$. To do this, we use the decomposition provided in \eqref{eq:partialtK}. First, we recall that
		\begin{equation}\label{eq:controlpartK1}
			p_{\Phi}(\zeta-s,\varphi(\zeta)-\varphi(s);0) \le \frac{U_{-\frac{1}{2}}(\zeta-s)}{\sqrt{2\pi}} \le \frac{U_{-\frac{1}{2}}(t-s)}{\sqrt{2\pi}}
		\end{equation}
		where we recall that the function $U_{-\frac{1}{2}}$ is decreasing, while, at the same time, $\varphi'$ is bounded in $\left[t,\frac{t+t_{\sf const}}{2}\right]$ since $\varphi$ is locally Lipschitz. For the second term in \eqref{eq:partialtK}, fix $c_1,c_2>0$ and observe that, by the reflection principle,
		\begin{equation*}
			\dP_0(\mathcal{T}_{c_1} \le c_2)=2\int_{c_1}^{+\infty}p_\Phi(c_2,z;0)\, dz.
		\end{equation*}
		With exactly the same argument we used for $K(t,s)$ and $A_y(t)$, we can take the derivative inside the integral sign, getting
		\begin{equation*}
			p_{\mathcal{T}_{c_1}}(c_2;0)=2\int_{c_1}^{+\infty}\partial_t p_\Phi(c_2,z;0)\, dz.
		\end{equation*}
		Setting $c_1=\varphi(\zeta)-\varphi(s)$ and $c_2=\zeta-s$ we get
		\begin{equation*}
			2\int_{\varphi(\zeta)-\varphi(s)}^{+\infty}\partial_t p_\Phi(\zeta-s,z;0)\, dz=p_{\mathcal{T}_{\varphi(\zeta)-\varphi(s)}}(\zeta-s;0).
		\end{equation*}
		Now recall that $(\mathcal{T}_c)_{c \ge 0}$ is a subordinator with Laplace exponent $\widetilde{\Phi}(\lambda)=\sqrt{2\Phi(\lambda)}$ by Proposition~\ref{prop:subordinatorT}, hence, denoting by $\widetilde{\sigma}$ a subordinator with such a Laplace exponent, we can write
		\begin{equation*}
			2\int_{\varphi(\zeta)-\varphi(s)}^{+\infty}\partial_t p_\Phi(\zeta-s,z;0)\, dz=g_{\widetilde{\sigma}}(\zeta-s,\varphi(\zeta)-\varphi(s)).
		\end{equation*}
		By Assumption \eqref{ass4}, we know that $\Phi$ admits a holomorphic extension on the complex sector $\mathbb{C}(\theta)$ for some $\theta \in \left(\frac{\pi}{2},\pi\right)$ and that such extension is continuous in $\overline{\mathbb{C}(\theta)}$, thus the same holds for $\widetilde{\Phi}$. Next, notice that since $\lim_{z \to \infty}z^{-1}\Phi(z)=0$ uniformly in $\overline{\mathbb{C}(\theta)}$, the same must hold for $\widetilde{\Phi}$. Next, since ${\rm Arg}(\Phi(z))<\pi$, we have ${\rm Arg}(\widetilde{\Phi}(z))<\frac{\pi}{2}$ and thus $\Re(\widetilde{\Phi}(z)) \ge 0$ for all $z \in \overline{\mathbb{C}(\theta)}$. Furthermore, observe that for fixed $a>0$ and $b \in \R$
		\begin{equation*}
			\Re(\widetilde{\Phi}(a+ib))=\sqrt{|\Phi(a+ib)|+\Re(\Phi(a+ib))} \ge \sqrt{2\Re\Phi(a+ib)}.
		\end{equation*}
		For $|b|>M_\gamma$, we have by Assumption \eqref{orcond} and Lemma \ref{lemmaintzero}
		\begin{equation*}
			\Re(\widetilde{\Phi}(a+ib)) \ge \sqrt{2C_\gamma}e^{-\frac{t_\gamma}{2}a}|b|^{1-\frac{\gamma}{2}},
		\end{equation*}
		so it is clear that
		\begin{equation*}
			\lim_{b \to +\infty}\frac{\Re(\widetilde{\Phi}(a+ib))}{\log(b)}=+\infty.
		\end{equation*}
		Finally, we recall that since $\lim_{z \to +\infty}z^{-1}\Phi(z)=0$ uniformly, it holds
		\begin{equation*}
			\lim_{r \to +\infty}\frac{|\Phi(re^{i\theta})|}{r}=0 \qquad \mbox{ that implies } \qquad \lim_{r \to +\infty}\frac{|\widetilde{\Phi}(re^{i\theta})|}{\sqrt{r}}=0.
		\end{equation*}
		Hence, by Proposition \ref{prop:keyhole}, denoting by $C_3$ the constant in \eqref{eq:controlgst}, we have
		\begin{equation*}
			2\int_{\varphi(\zeta)-\varphi(s)}^{+\infty}\partial_t p_\Phi(\zeta-s,z;0)\, dz\le C_3(\zeta-s)\left(1+\frac{1}{(\varphi(\zeta)-\varphi(s))^{\alpha+1}}\right).
		\end{equation*}
		Next, recall that $\varphi$ belongs to $C^1[0,t_{\sf const})$ with positive derivative, hence, by the mean value theorem
		\begin{equation*}
			\varphi(\zeta)-\varphi(s) \ge C_4(\zeta-s), \qquad \mbox{ where }\qquad C_4=\min_{w \in \left[0,\frac{t+t_{\sf const}}{2}\right]}\varphi'(w)
		\end{equation*}
		and then
		\begin{equation}\label{eq:controlpartK2}
			2\int_{\varphi(\zeta)-\varphi(s)}^{+\infty}\partial_t p_\Phi(\zeta-s,z;0)\, dz\le C_3\left(\zeta-s+\frac{1}{C_4^{\alpha+1}(\zeta-s)^{\alpha}}\right) \le C_3\left(\frac{t+t_{\sf const}}{2}+\frac{1}{C_4^{\alpha+1}(t-s)^{\alpha}}\right).
		\end{equation}
		Combining \eqref{eq:controlpartK1} and \eqref{eq:controlpartK2} we get
		\begin{equation*}
			|\partial_t K(\zeta,s)| \le \frac{\sqrt{2}C_5}{\sqrt{\pi}}\frac{U_{-\frac{1}{2}}(t-s)}+C_3\left(\frac{t+t_{\sf const}}{2}+\frac{1}{C_4^{\alpha+1}(t-s)^{\alpha}}\right), \quad \mbox{ where } \quad C_5=\max_{w \in \left[0,\frac{t+t_{\sf const}}{2}\right]}\varphi'(w).
		\end{equation*}
		Hence, recalling also that $p_{\mathcal{T}}(s;y)$ is bounded by \eqref{eq:STIT} since $y<\varphi(0)$, by the dominated convergence theorem we have 
		\begin{equation*}
			\lim_{h \downarrow 0}J_1(h)=\int_0^t \partial_t K(t,s)p_{\mathcal{T}}(s;y)\, ds.
		\end{equation*}
		Next, we need to handle $J_2(h)$. To do this, we use again the reflection principle to observe that
		\begin{equation*}
			K(t+h,s)=2\dP_0(X_\Phi(t+h-s) \ge \varphi(t+h)-\varphi(s)).	
		\end{equation*}
		Denoting by
		\begin{equation*}
			G(x)=\frac{1}{\sqrt{2\pi}}\int_{-\infty}^{x}e^{-\frac{z^2}{2}}\, dz,
		\end{equation*}
		we have, by a simple conditioning argument,
		\begin{align*}
			1-K(t+h,s)=\E_0\left[2G\left(\frac{\varphi(t+h)-\varphi(s)}{\sqrt{L(t+h-s)}}\right)-1\right] \le \E_0\left[2G\left(\frac{\varphi(t+h)-\varphi(t)}{\sqrt{L(h)}}\right)-1\right].
		\end{align*}
		Now let us recall that, as a consequence of, \cite[Proposition III.8]{bertoinb}, together with the fact that $\varphi$ is locally Lipschitz, we have
		\begin{equation*}
			\lim_{h \downarrow 0}\frac{\varphi(t+h)-\varphi(t)}{\sqrt{L(h)}}=0
		\end{equation*}
		almost surely. Thus we have
		\begin{equation*}
			J_2(h) \le \E_0\left[2G\left(\frac{\varphi(t+h)-\varphi(t)}{\sqrt{L(h)}}\right)-1\right]S_{\mathcal{T}}(\{y\};[0,t]) \to 0
		\end{equation*}
		as $h \to 0$. Finally, if $t$ is a right Lebesgue point of $p_{\mathcal{T}}(s;y)$, we get
		\begin{equation*}
			\lim_{h \downarrow 0}J_3(h)=p_\mathcal{T}(t;y).
		\end{equation*}
		Hence, taking the right derivative in \eqref{eq:prevolterra} on the right Lebesgue points of $p_{\mathcal{T}}(\cdot;y)$, we have
		\begin{equation*}
			p_{\mathcal{T}}(t;y)=\partial_t A_y(t)-\int_0^t \partial_t K(t,s)p_{\mathcal{T}}(s;y)\, ds,
		\end{equation*}
		where the right-hand side is, in any case, well-defined for all $t \in [0,t_{\sf const})$. Now let
		\begin{equation*}
			\widetilde{p}_{\mathcal{T}}(t;y)=\partial_t A_y(t)-\int_0^t \partial_t K(t,s)p_{\mathcal{T}}(s;y)\, ds
		\end{equation*}
		and notice that $\widetilde{p}_{\mathcal{T}}(t;y)=p_{\mathcal{T}}(t;y)$ for all right Lebesgue points fo $p_{\mathcal{T}}(\cdot;y)$, hence it is a version of the density of $\mathcal{T}$ under $\dP_y$. It remains to show that $\widetilde{p}_{\mathcal{T}}(\cdot;y)$ is continuous in $[0,t_{\sf const})$. We have already shown this property for the first summand. For the second summand, we recall the bound
		\begin{equation*}
			\left|\partial_t K(t,s)\right| \le C_5\frac{U_{-\frac{1}{2}}(t-s)}{\sqrt{2\pi}}+C_3\left(\frac{t+t_{\sf const}}{2}+\frac{1}{C_4^{\alpha+1}(t-s)^\alpha}\right)=:H(t-s).
		\end{equation*}
		To prove that the second summand is continuous, let us first consider the case $h \in \left[0,\frac{t_{\sf const}-t}{2}\right]$ and
		\begin{align*}
			\int_0^{t+h} &\partial_t K(t+h,s)p_{\mathcal{T}}(s;y)\, ds-\int_0^{t} \partial_t K(t,s)p_{\mathcal{T}}(s;y)\, ds\\
			&=\int_t^{t+h} \partial_t K(t+h,s)p_{\mathcal{T}}(s;y)\, ds+\int_{0}^{t}\left(\partial_t K(t+h,s)-\partial_t K(t,s)\right)p_{\mathcal{T}}(s;y)\, ds. 
		\end{align*}
		To handle the first summand, we notice that, with the same argument as before,
		\begin{equation*}
			\left|\partial_t K(t+h,s)\right| \le C_5\frac{U_{-\frac{1}{2}}(t+h-s)}{\sqrt{2\pi}}+C_3\left(\frac{t+t_{\sf const}}{2}+\frac{1}{C_4^{\alpha+1}(t+h-s)^{\alpha}}\right).
		\end{equation*}
		Hence
		\begin{align*}
			\int_t^{t+h} &\left|\partial_t K(t+h,s)\right|p_{\mathcal{T}}(s;y)\, ds \\
			&\le S_{\mathcal{T}}\left(\left[0,\frac{t+t_{\sf const}}{2}\right],\{y\}\right)\int_t^{t+h}\left(C_5\frac{U_{-\frac{1}{2}}(t+h-s)}{\sqrt{2\pi}}+C_3\left(\frac{t+t_{\sf const}}{2}+\frac{1}{C_4^{\alpha+1}(t+h-s)^{\alpha}}\right)\right)\, ds\\
			&=S_{\mathcal{T}}(\left[0,\frac{t+t_{\sf const}}{2}\right],\{y\})\left(\frac{C_5}{\sqrt{2\pi}}\int_{0}^{h}U_{-\frac{1}{2}}(s)\,ds+C_3\frac{t+t_{\sf const}}{2}h+\frac{C_3h^{1-\alpha}}{(1-\alpha)C_4^{\alpha+1}}\right) \to 0,
		\end{align*}
		as $h \downarrow 0$. On the other hand, in the second integral, since $s \in (0,t)$, we have
		\begin{equation*}
			|\partial_t K(t+h,s)-\partial_t K(t,s)| \le 2C_5\frac{U_{-\frac{1}{2}}(t-s)}{\sqrt{2\pi}}+C_3\left(\frac{t+t_{\sf const}}{2}+\frac{1}{C_4^{\alpha+1}(t-s)^{\alpha}}\right),
		\end{equation*}
		that is integrable against $p_{\mathcal{T}}(s;y)$ in $[0,t]$. Thus we can use the dominated convergence theorem to guarantee that
		\begin{equation*}
			\lim_{h \downarrow 0}\int_0^{t}(\partial_t K(t+h,s)-\partial_t K(t,s))p_{\mathcal{T}}(s;y)\, ds=0.
		\end{equation*}
		When $h \in \left[-\eta,0\right]$, where $\eta<\frac{t}{2}$, we have
		\begin{align*}
			\int_0^{t} &\partial_t K(t,s)p_{\mathcal{T}}(s;y)\, ds-\int_0^{t+h} \partial_t K(t+h,s)p_{\mathcal{T}}(s;y)\, ds\\
			&=\int_{t+h}^{t} \partial_t K(t,s)p_{\mathcal{T}}(s;y)\, ds+\int_{0}^{t+h}\left(\partial_t K(t,s)-\partial_t K(t+h,s)\right)p_{\mathcal{T}}(s;y)\, ds. 
		\end{align*}
		This time it is immediate to check that
		\begin{equation*}
			\lim_{h \uparrow 0}\int_{t+h}^{t}\partial_t K(t,s)p_{\mathcal{T}}(s;y)\, ds=0.
		\end{equation*}
		For the second integral, let us split furthermore the terms into
		\begin{align*}
			\int_{0}^{t+h}&\left(\partial_t K(t,s)-\partial_t K(t+h,s)\right)p_{\mathcal{T}}(s;y)\, ds\\
			&=\int_{0}^{t-\eta}\left(\partial_t K(t,s)-\partial_t K(t+h,s)\right)p_{\mathcal{T}}(s;y)\, ds+\int_{t-\eta}^{t+h}\left(\partial_t K(t,s)-\partial_t K(t+h,s)\right)p_{\mathcal{T}}(s;y)\, ds.
		\end{align*}
		For $s \in (0,t-\eta)$, we have
		\begin{equation*}
			|\partial_t K(t+h,s)-\partial_t K(t,s)| \le 2C_5\frac{U_{-\frac{1}{2}}(t-\eta-s)}{\sqrt{2\pi}}+C_3\left(\frac{t+t_{\sf const}}{2}+\frac{1}{C_4^{\alpha+1}(t-\eta-s)^{\alpha}}\right),
		\end{equation*}
		hence we can use the dominated convergence theorem to show that
		\begin{equation*}
			\lim_{h \to 0}\int_{0}^{t-\eta}\left|\partial_t K(t,s)-\partial_t K(t+h,s)\right|p_{\mathcal{T}}(s;y)\, ds=0.
		\end{equation*}
		For the second subpart, we only have a bound as follows:
		\begin{equation*}
			|\partial_t K(t+h,s)-\partial_t K(t,s)|1_{[0,t+h]}(s) \le 2C_5\frac{U_{-\frac{1}{2}}(t+h-s)}{\sqrt{2\pi}}+C_3\left(\frac{t+t_{\sf const}}{2}+\frac{1}{C_4^{\alpha+1}(t+h-s)^{\alpha}}\right)=:H(t+h-s),
		\end{equation*}
		so 
		\begin{equation*}
			\int_{t-\eta}^{t+h}\left|\partial_t K(t,s)-\partial_t K(t+h,s)\right|p_{\mathcal{T}}(s;y)\, ds \le S_{\mathcal{T}}(\{y\};[0,t])\int_{t-\eta}^{t+h}H(t+h-s)\, ds=S_{\mathcal{T}}(\{y\};[0,t])\int_{0}^{h+\eta}H(s)\, ds.
		\end{equation*}
		Hence, taking the limit superior as $h \uparrow 0$ we
		have
		\begin{equation*}
			\limsup_{h \uparrow 0}\left|\int_0^{t+h}\partial_t K(t+h,s)p_{\mathcal{T}}(s;y)\, ds-\int_0^{t}\partial_t K(t,s)p_{\mathcal{T}}(s;y)\, ds\right| \le S_{\mathcal{T}}(\{y\};[0,t])\int_{0}^{\eta}H(s)\, ds.
		\end{equation*}
		However, we can now take the limit as $\eta \downarrow 0$ on the right-hand side, finally obtaining
		\begin{equation*}
			\lim_{h \uparrow 0}\left|\int_0^{t+h}\partial_t K(t+h,s)p_{\mathcal{T}}(s;y)\, ds-\int_0^{t}\partial_t K(t,s)p_{\mathcal{T}}(s;y)\, ds\right|=0.
		\end{equation*}
		This proves that $\widetilde{p}_{\mathcal{T}}(\cdot;y)$ is continuous in $[0,t_{\sf const})$.
	}
\end{proof}
\textcolor{black}{\begin{rmk}
	As a side effect of the previous theorem, we actually proved that the continuous version $p_{\mathcal{T}}$ of the first passage time density solves a Volterra equation
	\begin{equation*}
		p_{\mathcal{T}}(t;y)=\partial_t A_y(t)-\int_0^{t}\partial_t K(t,s)p_{\mathcal{T}}(s;y)\, ds.
	\end{equation*}
	These kind of equations were already well-known in the context of Gauss-Markov processes, see for instance \cite{dinardo}, and have already been employed for the study of density of time-changed processes in \cite{leonenko}, for the crossing time of a time-changed Brownian motion and a threshold subject to the same time-change.
\end{rmk}}
\textcolor{black}{From now on, $p_{\mathcal{T}}$ will always denote the continuous version of the density of $\mathcal{T}$ if it exists.}
\textcolor{black}{To handle the case in which $x \ge \varphi(0)$, we need the additional $C^1$ condition on $\varphi$. 
\begin{prop}\label{prop:der2qphi}
	Suppose that Assumptions \ref{ass2}, \ref{orcond} and \ref{ass4} hold. Let $\varphi:\R^+_0 \to \R$ be non-decreasing, locally Lipschitz and assume further that if there exists $t_1<t_2$ such that $\varphi(t_1)=\varphi(t_2)$, then $\varphi(t)=\varphi(t_1)$ for all $t \ge t_1$. Let
	\begin{equation*}
		t_{\sf const}:=\inf\{t>0: \ \varphi(t)=\varphi(t+1)\}, \quad \mbox{ where } \inf\emptyset=+\infty,
	\end{equation*}
	and assume that $\varphi \in C^1[0,t_{\sf const})$ with $\varphi'(t)>0$ for all $t \in [0,t_{\sf const})$. Then for any $t>0$, $y<\varphi(0)$ and $x \in (\varphi(0),\varphi(t))$ it holds
	\begin{equation*}
		\partial_x^2 q_\Phi(t,x;y)=\partial_x^2 p_\Phi(t,x;y)-\int_0^t \partial_x^2p_\Phi(t-w,x;\varphi(w))p_{\mathcal{T}}(w;y)\, dw+\frac{2\overline{\nu}(t-\varphi^{-1}(x))}{\varphi'(\varphi^{-1}(x))}p_{\mathcal{T}}(\varphi^{-1}(x);y).
	\end{equation*}
\end{prop}
\begin{proof}
	Let us define
	\begin{equation*}
		\widetilde{p}_\Phi(t,x;y)=\partial_xp_\Phi(t,x;y)+2\overline{\nu}(t)1_{\R^+_0}(x-y)
	\end{equation*}
	to compensate the jump discontinuity of $\partial_x p_\Phi$. Then we have, from \eqref{eq:derxr}
	\begin{align*}
	\partial_x r_\Phi(t,x;y)&=\int_0^t \widetilde{p}_\Phi(t-w,x;\varphi(w))p_{\mathcal{T}}(w;y)\, dw-2\int_0^t \overline{\nu}(t-w)1_{\R^+_0}(x-\varphi(w))p_{\mathcal{T}}(w;y)\, dw\\
	&=:r_1(t,x;y)-r_2(t,x;y).
	\end{align*}
	Since $\widetilde{p}_\Phi(t,\cdot;y) \in C^1(\R)$, the same argument as in Proposition \ref{prop:derqphix} shows that
	\begin{equation*}
		\partial_x r_1(t,x;y)=\int_0^t \partial_x^2p_\Phi(t-w,x;\varphi(w))p_{\mathcal{T}}(w;y)\, dw.
	\end{equation*}
	Concerning $r_2(r,x;y)$, recall that $\varphi$ is strictly increasing and continuous up to $t_{\sf const}$ so that
	\begin{equation*}
		r_2(t,x;y)=2\int_0^{\varphi^{-1}(x)} \overline{\nu}(t-w)p_{\mathcal{T}}(w;y)\, dw
	\end{equation*}
	and taking the derivative we get
	\begin{equation*}
		\partial_x r_2(t,x;y)=\frac{2\overline{\nu}(t-\varphi^{-1}(x))}{\varphi'(\varphi^{-1}(x))}p_{\mathcal{T}}(\varphi^{-1}(x);y).
	\end{equation*}
	This ends the proof.
\end{proof}}
Next, we investigate the action of the nonlocal operator $\partial_t^\Phi$ on $q_\Phi$.
\begin{prop}\label{prop:timeder}
	Suppose Assumption \eqref{ass2}, \eqref{orcond} and \eqref{ass4} hold. Let $\varphi$ be non-decreasing, locally Lipschitz \textcolor{black}{ and such that if there exists $0 \le t_1<t_2<\infty$ such that $\varphi(t_1)=\varphi(t_2)$ then $\varphi(t_1)=\varphi(t)$ for all $t \ge t_1$.} Then, for $t>0$ and \textcolor{black}{$x,y \in \R$ with $x<\varphi(t)$ } it holds
	\textcolor{black}{\begin{equation*}
		D_t^\Phi q_\Phi(t,x;y)=D_t^\Phi p_\Phi(t,x;y)-\int_0^{t}D_t^\Phi p_\Phi(t-w,x;\varphi(w))p_{\mathcal{T}}(w;y)\, dw.
	\end{equation*}}
\end{prop}
\begin{proof}
	Define
	\begin{equation}\label{eq:kPhi}
		k_\Phi(t,w,x):=1_{\R^+}(t-w)\mathcal{I}^\Phi_t p_\Phi(t-w,x;\varphi(w)).
	\end{equation}
	Notice that
	\begin{align*}
		\mathcal{I}^\Phi_t r_\Phi(t,x;y)\, ds&=\int_0^t \int_0^{t-s}\overline{\nu}(s)p_\Phi(t-s-w,x;\varphi(w))p_{\mathcal{T}}(w;y)\, dw \, ds\\
		&=\int_0^t \int_0^{t-w}\overline{\nu}(s)p_\Phi(t-s-w,x;\varphi(w))p_{\mathcal{T}}(w;y)\, ds \, dw\\
		&=\int_0^{+\infty} 1_{\R^+}(t-w)k_\Phi(t,w,x)p_{\mathcal{T}}(w;y)\, dw.
	\end{align*}
	Since $x<\varphi(t)$, then either $x<\varphi(0)$ and then $\varphi^{-1}(x)=\emptyset$ or $x \in [\varphi(0),\varphi(t))$ and then $\varphi^{-1}(x)=\{w_x\}$. For $w \not = w_x$ and $w \not = t$ we have, for $w<t$,
	\begin{equation}\label{eq:partialtkPhi}
		\partial_t k_\Phi(t,w,x)=D_t^\Phi p_{\Phi}(t-w,x;\varphi(w))=\frac{1}{2}\partial_x^2p_\Phi(t-w,x;\varphi(w)).
	\end{equation}
	Before proceeding, let us provide a further bound on $\partial_x^2p_\Phi$. Consider two compact sets $K_1,K_2 \subset \R$ with $K_1 \cap K_2 =\emptyset$ and let $\varepsilon=\min_{(\xi,\eta) \in K_1 \times K_2}|\xi-\eta|$. Then we have, for $(\xi,\eta) \in K_1 \times K_2$,
	\begin{align}\label{eq:precontrol}
		\begin{split}
		\left|\partial^2_xp_\Phi(t,\xi;\eta)\right|&\le \sqrt{\frac{2}{\pi}}\E\left[(L(t))^{-\frac{3}{2}}|(\xi-\eta)^2(L(t))^{-1}-1|e^{-\frac{(\xi-\eta)^2}{2L(t)}}\right]\\
		&\le \sqrt{\frac{2}{\pi}}\E\left[(\xi-\eta)^2(L(t))^{-\frac{5}{2}}e^{-\frac{(\xi-\eta)^2}{2L(t)}}\right]+\sqrt{\frac{2}{\pi}}\E\left[(L(t))^{-\frac{3}{2}}e^{-\frac{(\xi-\eta)^2}{2L(t)}}\right].	
		\end{split}
	\end{align}
	Now observe that for all $r \ge 0$ and $\lambda>0$ it holds
	\begin{equation*}
		r^{\frac{5}{2}}e^{-\lambda r} \le \left(\frac{5}{2e\lambda }\right)^{\frac{5}{2}} \qquad \mbox{ and } r^{\frac{3}{2}}e^{-\lambda r} \le \left(\frac{3}{2e\lambda }\right)^{\frac{3}{2}}
	\end{equation*}
	hence we get, setting $r=(L(t))^{-1}$ and $\lambda=\frac{(\xi-\eta)^2}{4}$, by \eqref{eq:precontrol},  
	\begin{align}
		\left|\partial^2_xp_\Phi(t,\xi;\eta)\right|&\le \sqrt{\frac{2}{\pi}}\left(\left(\frac{10}{e}\right)^{\frac{5}{2}}+\left(\frac{6}{e}\right)^{\frac{3}{2}}\right)(\xi-\eta)^{-3}\E\left[e^{-\frac{(\xi-\eta)^2}{4L(t)}}\right]\\
		&\le \sqrt{\frac{2}{\pi}}\left(\left(\frac{10}{e}\right)^{\frac{5}{2}}+\left(\frac{6}{e}\right)^{\frac{3}{2}}\right)\varepsilon^{-3}\E\left[e^{-\frac{\varepsilon^2}{4L(t)}}\right].
		\label{estim23}
	\end{align}
	Furthermore, by Theorem \ref{lemmaeqfraz} we know that
	\begin{align*}
		\left|\partial^\Phi_tp_\Phi(t,\xi;\eta)\right|
		&\le \frac{1}{\sqrt{2\pi}}\left(\left(\frac{10}{e}\right)^{\frac{5}{2}}+\left(\frac{6}{e}\right)^{\frac{3}{2}}\right)\varepsilon^{-3}\E\left[e^{-\frac{\varepsilon^2}{4L(t)}}\right].
	\end{align*}
	Going back to \eqref{eq:partialtkPhi}, notice that for $w \to t$ we have that $\varphi(w) \to \varphi(t)>x$. Hence, there exists $\delta>0$ and two compact sets $K_1, K_2 \subset \R$ such that $K_1 \cap K_2=\emptyset$,  $\varphi(w) \in K_1$ for all $w \in [t-\delta,t)$ and $x \in K_2$. Hence, if $\varepsilon=\min_{(\xi,\eta) \in K_1 \times K_2}|\xi-\eta|$, we have
	\begin{equation}\label{eq:partialtkPhi2}
		\left|\partial_t k_\Phi(t,w,x)\right| \le \frac{1}{\sqrt{2\pi}}\left(\left(\frac{10}{e}\right)^{\frac{5}{2}}+\left(\frac{6}{e}\right)^{\frac{3}{2}}\right)\varepsilon^{-3}\E\left[e^{-\frac{\varepsilon^2}{4L(t-w)}}\right]
	\end{equation}
	and then
	\begin{equation*}
		\lim_{w \uparrow t}\left|\partial_t k_\Phi(t,w,x)\right|=0.
	\end{equation*}
	 If $x<\varphi(0)$, exactly the same argument guarantees that $\partial_t k_\Phi(\cdot,\cdot,x)$ is bounded and contiuous and that for any compact set $K \subset (0,+\infty)$ it holds $\sup_{(t,w) \in K \times \R^+}|\partial_t k_\Phi(t,w,x)|<\infty$. Hence, a simple dominated convergence argument, together with the fact that $\lim_{w \to t}k_\Phi(t,w,x)=0$, shows that
	\begin{align}\label{eq:finalequality}
		D_t^\Phi r_\Phi(t,x;y)=\int_0^{+\infty} \partial_tk_\Phi(t,w,x)p_{\mathcal{T}}(w;y)\, dw=\int_0^{t}D_t^\Phi p_\Phi(t-w,x;\varphi(w))p_{\mathcal{T}}(w;y)\, dw.
	\end{align}
	Now let us handle the case $x \in [\varphi(0),\varphi(t))$. In this case, the function $\partial_t k_\Phi(t,\cdot,x)$ is defined in $w_x$ as
	\begin{equation}\label{eq:rhsbound}
	\partial_t k_\Phi(t,w_x,x)=-\frac{1}{\sqrt{2\pi}}\int_0^{+\infty}s^{-\frac{1}{2}}\partial_sf_L(s,t-w_x)\, ds.
	\end{equation}
	Furthermore, by Assumption \eqref{ass4}, we know that the right-hand side of \eqref{eq:rhsbound} is locally bounded with respect to $t$. As a consequence we notice that for any compact set $K \subset (0,+\infty)$ it holds $\sup_{(t,w) \in K \times \R^+}|\partial_t k_\Phi(t,w,x)|<\infty$ to obtain again \eqref{eq:finalequality}. Thus, applying the operator $D_t^\Phi$ on $q_\Phi(\cdot,x;y)$ and using \eqref{reprq} and \eqref{eq:finalequality} we get the desired result.




\end{proof}
\begin{rmk}\label{rmkcont2}
	 \textcolor{black}{Notice that since $r_\Phi(0,x;y)=0$, then $D_t^\Phi r_\Phi(t,x;y)=\partial_t^\Phi r_\Phi(t,x;y)$ for all $x,y \in \R$. Furthermore, $D_t^\Phi p_\Phi(\cdot,x;y) \in C(0,+\infty)$ and, for fixed $x \not = y$ we have $\partial_t^\Phi p_{\Phi}(\cdot,x,y) \in C(0,+\infty)$ by Remark \ref{rmkcont}. On the other hand, observe that we have shown that $\partial_t k_\Phi(\cdot,\cdot,x)$ is bounded and continuous in $(0,+\infty)\times (0,+\infty)$ for any $x \in E_2$ and, at the same time,
	 \begin{equation*}
	 	\partial_t^\Phi r_\Phi(t,x;y)=\E_y[\partial_t k_\Phi(t,\mathcal{T},x)].
	 \end{equation*}
	 Hence, for any $t_0>0$, we can use the dominated convergence theorem to guarantee that
	 \begin{equation*}
	 	\lim_{t \to t_0}\partial_t^\Phi r_\Phi(t,x;y)=\lim_{t \to t_0}\E_y[\partial_t k_\Phi(t,\mathcal{T},x)]=\E_y\left[\lim_{t \to t_0}\partial_t k_\Phi(t,\mathcal{T},x)\right]=\E_y[\partial_t k_\Phi(t_0,\mathcal{T},x)]=\partial_t^\Phi r_\Phi(t_0,x;y).
	 \end{equation*}
	 Hence also $\partial_t^\Phi r_\Phi(\cdot,x;y) \in C(0,+\infty)$. As a consequence, we finally get that $D_t^\Phi q_\Phi(\cdot,x;y) \in C(0,+\infty)$ for all $x,y \in \R$ and $\partial_t^\Phi q_\Phi(\cdot,x;y) \in C(0,+\infty)$ when $x \not = y$.}
\end{rmk}

\section{Main result}
\label{sec:mainres}
We are now ready to state and prove our main result which will imply, as a particular case, the existence result in Theorem \ref{thm:mainfrac}. Uniqueness will be dealt with separately.
\begin{thm}\label{SPthm}
	Suppose that Assumptions \eqref{ass2}, \eqref{orcond} and \eqref{ass4} hold. Let $\varphi:\R_0^+ \to \R$ be non-decreasing and locally Lipschitz. Assume further that if there exist $t_1 < t_2$ such that $\varphi(t_1)=\varphi(t_2)$, then $\varphi(t)=\varphi(t_1)$ for all $t \ge t_1$. Let
	\begin{equation*}
		t_{\sf const}:=\inf\{t>0: \ \varphi(t)=\varphi(t+1)\}, \quad \mbox{ where } \inf\emptyset=+\infty,
	\end{equation*}
	and assume that $\varphi \in C^1[0,t_{\sf const})$ with $\varphi'(t)>0$ for all $t \in [0,t_{\sf const})$. Assume further that $p_{\mathcal{T}}(\cdot;y)$ is continuous in $(\varphi^{-1}(0),\varphi^{-1}(t_{\sf const}))$. Let also $f \in C_c(-\infty,\varphi(0))$ \textcolor{black}{be Dini-continuous} and consider the time-nonlocal Cauchy-Dirichlet problem
	\begin{equation}\label{eq:nonlocmovb}
		\begin{cases}
			\displaystyle \partial_t^\Phi u(t,x)=\frac{1}{2}\partial_x^2 u(t,x) & t>0, \ x<\varphi(t)\\
			u(t,x)=0 & t \ge 0, \ x \ge \btrev{\varphi(t)}\\
			u(0,x)=f(x) & x<\varphi(0)\\
			\lim_{x \to -\infty}u(t,x)=0 & \mbox{locally uniformly with respect to }t>0.
		\end{cases}
	\end{equation}
	Then 	there exists a function $u:\R_0^+ \times \R \to \R$ such that, setting $E=\{(t,x) \in \R^+ \times \R: \ x<\varphi(t)\}$:
	\begin{enumerate}
		\item $u \in \btrev{C}(\overline{E})$;
		\item For all $t \in E_1$ it holds $u(t,\cdot) \in C^2(E_2(t))$;
		\item For all $x \in E_2$, it holds $\partial_t^\Phi u(\cdot,x) \in C(E_1(x))$;
		\item  For all $x \in E_2\setminus \{\varphi(0)\}$, it holds $\partial_t^\Phi u(\cdot,x) \in L^1_{\rm loc}(\overline{E_1(x)})$;
		\item $u$ satisfies \eqref{eq:nonlocmovb}.
	\end{enumerate} 
	In particular, such a function $u$ is given by
	\begin{equation}\label{eq:usol}
		u(t,x)=\int_{-\infty}^{\varphi(0)}q_\Phi(t,x;y)f(y)\, dy,
	\end{equation}
	where $q_\Phi$ is defined in Theorem \ref{lem:intrep} and we set $u(0,x)=f(x)$.
\end{thm}
In order to prove this theorem, we state and prove the following further result (which is, however of independent interest).
\begin{thm}
	\label{thm52}
Suppose that Assumptions \eqref{ass2}, \eqref{orcond} and \eqref{ass4} hold. Let $\varphi:\R_0^+ \to \R$ be as in Theorem \ref{SPthm}. 
Assume further that $p_{\mathcal{T}}(\cdot;y)$ is continuous in $(\varphi^{-1}(0),\varphi^{-1}(t_{\sf const}))$. Then the function $q_\Phi$ defined in Theorem \ref{lem:intrep} satisfies
\begin{equation}\label{eq:nonlocmovbq}
	\begin{cases}
		\displaystyle \partial_t^\Phi q_\Phi(t,x;y)=\frac{1}{2}\partial_x^2 q_\Phi(t,x;y) & t>0, \ x<\varphi(t), y<\varphi(0), \ x \not = y\\
		q_\Phi(t,x;y)=0 & t \ge 0, \ x \ge \varphi(t) \mbox{ or }y \ge \varphi(0)\\
		\lim_{x \to -\infty}q_\Phi(t,x;y)=0 & \mbox{locally uniformly with respect to }t>0 \mbox{ and } y \in \R\\
		\lim_{t \to 0^+}q_\Phi(t,x;y)=0 & x\not = y\\
		\lim_{t \to 0^+}q_\Phi(t,x;y)dy=\delta_x(dy) & \mbox{vaguely as a measure on $(-\infty,\varphi(0))$.}
	\end{cases}
\end{equation}
\end{thm}
\begin{proof}
First notice that the second equality in \eqref{eq:nonlocmovbq} is verified for $x \ge \varphi(t)$ by Theorem \ref{thm:cont}, while for $y \ge \varphi(0)$ this is implied by the definition of $q_\Phi$. Furthermore, the fourth equality in \eqref{eq:nonlocmovbq} is proved in Proposition \ref{eq:extendto0}. Concerning the last equality in \eqref{eq:nonlocmovbq}, recall that, by Proposition \ref{prop:regx}, $\lim_{t \downarrow 0}p_\Phi(t,x;y)dy=\delta_x(dy)$ weakly. Now let $f \in C_c(-\infty,\varphi(0))$ and observe that, since ${\sf supp}(f) \subset (-\infty,\varphi(0))$,
\begin{equation*}
	\int_{-\infty}^{\varphi(0)}q_{\Phi}(t,x;y)f(y)\, dy= \int_{-\infty}^{+\infty}p_{\Phi}(t,x;y)f(y)\, dy-\int_{{\sf supp}(f)}r_\Phi(t,x;y)f(y)\, dy.
\end{equation*}
Furthermore, observe that, for $t \le 1$, since $\varphi$ is locally Lipschitz, we can use Theorem \ref{thm:density} to obtain
\begin{equation*}
\left| \int_{{\sf supp}(f)}r_\Phi(t,x;y)f(y)\, dy \right| \le \frac{|{\sf supp(f)}|S_{\mathcal{T}}({\sf supp}(f),\btrev{1})\max_{y \in {\sf supp(f)}}|f(y)|}{\sqrt{2\pi}} \int_0^t U_{-\frac{1}{2}}(w)\, dw,
\end{equation*}
\btrev{where $U_p$ is defined in \eqref{Up},} that implies
\begin{equation*}
	\lim_{t \downarrow 0}\int_{{\sf supp}(f)}r_\Phi(t,x;y)f(y)\, dy=0.
\end{equation*}
Hence
\begin{multline*}
	\lim_{t \downarrow 0}\int_{-\infty}^{\varphi(0)}q_{\Phi}(t,x;y)f(y)\, dy= \lim_{t \downarrow 0}\int_{-\infty}^{+\infty}p_{\Phi}(t,x;y)f(y)\, dy-\lim_{t \downarrow 0}\int_{{\sf supp}(f)}r_\Phi(t,x;y)f(y)\, dy=f(x).
\end{multline*}
Since $f \in C_c(-\infty,\varphi(0))$ is arbitrary, this proves the fifth equality in \eqref{eq:nonlocmovbq}.

To prove the third equality in \eqref{eq:nonlocmovbq}, notice that for fixed $T>0$ and $t \in [0,T]$ we have 
\begin{equation*}
	r_\Phi(t,x;y) \le \frac{S_{\mathcal{T}}(\{y\},T)}{\sqrt{2\pi}} \int_0^t \btrev{\E_y}\left[(L(t-w))^{-\frac{1}{2}}e^{-\frac{(x-\varphi(w))^2}{2L(t-w)}}\right]\, dw.
\end{equation*}
Since we want to study the limit as $x \to -\infty$, we can assume, without loss of generality, that $(x-\varphi(0))^2 \le (x-\varphi(w))^2$. Thus it holds
\begin{align*}
	r_\Phi(t,x;y) &\le \frac{S_{\mathcal{T}}(\{y\},T)}{\sqrt{2\pi}} \int_0^t \btrev{\E_y}\left[(L(w))^{-\frac{1}{2}}e^{-\frac{(x-\varphi(0))^2}{2L(w)}}\right]\, dw \\
	&\le \frac{S_{\mathcal{T}}(\{y\},T)}{\sqrt{2\pi}} \int_0^T \btrev{\E_y}\left[(L(w))^{-\frac{1}{2}}e^{-\frac{(x-\varphi(0))^2}{2L(w)}}\right]\, dw.
\end{align*}
Taking the supremum over $t \in [0,T]$ we have
\begin{equation*}
	\sup_{t \in [0,T]}r_\Phi(t,x;y) \le \frac{S_{\mathcal{T}}(\{y\},T)}{\sqrt{2\pi}} \int_0^T \btrev{\E_y}\left[(L(w))^{-\frac{1}{2}}e^{-\frac{(x-\varphi(0))^2}{2L(w)}}\right]\, dw.
\end{equation*}
Now, notice that the integrand can be controlled \btrev{(for any $y$, independently on $y$)} as 
\begin{equation*}
	\btrev{\E_y}\left[(L(w))^{-\frac{1}{2}}e^{-\frac{(x-\varphi(0))^2}{2L(w)}}\right] \le U_{-\frac{1}{2}}(w),
\end{equation*}
 and thus the right-hand side belongs to $L^1[0,T]$, so that we can use the dominated convergence theorem to conclude that
	\begin{equation}\label{eq:limitrPhi}
		\lim_{x \to -\infty}\sup_{\btrev{y \in K}}\sup_{t \in [0,T]}r_\Phi(t,x;y)=0.
	\end{equation}
	On the other hand, observing that for $r \ge 0$ and $\lambda>0$
	\begin{equation*}
		\sqrt{r}e^{-\lambda r} \le \sqrt{\frac{2}{\lambda e}},
	\end{equation*}
	we know that, for $x \not = y$,
	\begin{equation*}
		\sup_{t \in [0,T]}p_\Phi(t,x;y) \le \frac{1}{|x-y|} \btrev{\sqrt{\frac{2}{e \pi}}}  
	\end{equation*}
	and then, for any compact set $K \subset \R$,
	\begin{equation}\label{eq:limitpPhi}
		\lim_{x \to -\infty}\sup_{y \in K}\sup_{t \in [0,T]}p_\Phi(t,x;y)=0. 
	\end{equation}
	Combining \eqref{eq:limitrPhi} and \eqref{eq:limitpPhi} we get
	\begin{equation*}
		\lim_{x \to -\infty}\sup_{y \in K}\sup_{t \in [0,T]}q_\Phi(t,x;y)=0,
	\end{equation*}
	that, since $T>0$ and $K \subset \R$ are arbitrary, implies the third equality in \eqref{eq:nonlocmovbq}.
	
It only remains to show the first equality in \eqref{eq:nonlocmovbq}. This is done as follows: by Proposition \ref{prop:timeder} we know that for all $x<\varphi(t)$ with $x \not = y$ it holds
	\begin{equation}\label{eq:equality1}
		\partial_t^\Phi q_\Phi(t,x;y)=\partial_t^\Phi p_\Phi(t,x;y)-\int_0^t D_t^\Phi p_\Phi(t-w,x;\varphi(w))p_{\mathcal{T}}(w;y)\, dw.
	\end{equation}
	\textcolor{black}{On the second integral, we apply a change of variables to get $z=\varphi(w)$ $w=\varphi^{-1}(z)$
	\begin{equation}\label{eq:equality12}
	\partial_t^\Phi q_\Phi(t,x;y)=\partial_t^\Phi p_\Phi(t,x;y)-\int_{\varphi(0)}^{\varphi(t)} D_t^\Phi p_\Phi(t-\varphi^{-1}(z),x;z)\frac{p_{\mathcal{T}}(\varphi^{-1}(z);y)}{\varphi^\prime(\varphi^{-1}(z))}\, dz.
	\end{equation}
	}
	Hence, by Theorem \ref{lemmaeqfraz} (see, in particular, Remark \ref{adessobasta}), we know that for fixed $(t,x) \in \R^+ \times \R$ with $x<\varphi(t)$,
	\begin{equation}\label{eq:equality2}
		D_t^\Phi p_\Phi(t-\varphi^{-1}(z),x;z)=\frac{1}{2}\partial_x^2p_\Phi(t-\varphi^{-1}(z),x;z)-\overline{\nu}(t-\varphi^{-1}(z))\delta_{\{x\}}(dz)
	\end{equation}
	for all $z \in [\varphi(0),\varphi(t))$. Now we distinguish among two cases. If $x<\varphi(0)$, then $\delta_{\{x\}}([\varphi(0),\varphi(t)])=0$ and thus
	we have, by using \eqref{eq:equality2}, together with Theorem \ref{lemmaeqfraz} again, in \eqref{eq:equality12} to get
	\begin{equation*}
		\partial_t^\Phi q_\Phi(t,x;y)=\frac{1}{2}\left(\partial_x^2 p_\Phi(t,x;y)-\int_0^t \partial_x^2 p_\Phi(t-w,x;\varphi(w))p_{\mathcal{T}}(w;y)\, dw\right)=\frac{1}{2}\partial_x^2q_\Phi(t,x;y),
	\end{equation*}
	where we also used Proposition \ref{prop:derqphix}. If instead $x \in [\varphi(0),\varphi(t))$, then we gain an additional contribution due to the Dirac delta term, leading to
	\begin{align*}
		\partial_t^\Phi q_\Phi(t,x;y)&=\frac{1}{2}\left(\partial_x^2 p_\Phi(t,x;y)-\int_0^t \partial_x^2 p_\Phi(t-w,x;\varphi(w))p_{\mathcal{T}}(w;y)\, dw+\frac{2\overline{\nu}(t-\varphi^{-1}(x))}{\varphi^\prime(\varphi^{-1}(x))}p_{\mathcal{T}}(\varphi^{-1}(x);y)\right)\\
		&=\frac{1}{2}\partial_x^2q_\Phi(t,x;y),
	\end{align*}
	where this time we used Proposition \ref{prop:der2qphi}. Since $t>0$ and $x \in (-\infty,\varphi(t)) \setminus \{y\}$ is arbitrary, this ends the proof.
	\end{proof}	
	\begin{rmk}
		Notice that we also proved that for any $t>0$ and any $x,y \in \R$ it holds
		\begin{equation}\label{eq:eqphir}
			\partial_t^\Phi r_\Phi(t,x;y)=\frac{1}{2} \partial_x^2 r_\Phi(t,x;y).
		\end{equation}
	\end{rmk}
	\begin{proof}[Proof of Theorem \ref{SPthm}]
	Let $u$ be as in \eqref{eq:usol}. Let us first show that $u \in C(\R_0^+ \times \R)$. For $(t,x) \in \R^+ \times \R$, \btrev{we recall that, by Theorem \ref{thm:cont},} $q_\Phi(\cdot,\cdot;y) \in C(\R^+ \times \R)$. Furthermore, consider any $t_0 \in (0,t)$ and notice that
	\begin{equation*}
		\sup_{(\tau,\xi) \in [t_0,+\infty) \times \R}|f(y)q_\Phi(\tau,\xi;y)| \le \frac{\Norm{f}{L^\infty(\R)}}{\sqrt{2\pi t_0}}\mathbf{1}_{{\sf supp}(f)}(y),
	\end{equation*}
	hence the continuity of $u$ in $(t,x)$ follows by a simple application of the dominated convergence theorem. Concerning points of the form $(0,x)$ for $x \in \R$, let us split the solution as follows:
	\begin{equation}\label{eq:solusplit}
		u(t,x)=\widetilde{u}(t,x)+u_r(t,x),
	\end{equation}
	where
	\begin{equation*}
		\widetilde{u}(t,x)=\E[f(X_\Phi(t)+x)] \mbox{ and } u_r(t,x)=-\int_{\R}r_\Phi(t,x;y)f(y)\, dy,
	\end{equation*}
 \btrev{where we are using again the notation $\E:=\E_0$.}
	Concerning $\widetilde{u}(t,x)$, by dominated convergence it is clear that
	\begin{equation*}
	\lim_{(s,\xi) \to (0,x)}\widetilde{u}\btrev{(s,\xi)}=\E[f(x)]=u(0,x).
	\end{equation*}
	On the other hand, arguing as in \btrev{in the proof of Theorem \ref{thm52}}, we have, for $s \in (0,1]$ and $\xi \in \R$,
	\begin{equation*}
		|u_r(s,\xi)| \le  \frac{\btrev{(\max_{y \in {\sf supp}(f)} f(y))} \, |{\sf supp}(f)| \mathcal{S}_{\mathcal{T}}({\sf supp}(f),1)}{\sqrt{2\pi}}\int_0^s U_{-\frac{1}{2}}(w)\, dw
	\end{equation*}
	where $U_p$ is defined in \eqref{Up},	and then
	\begin{equation*}
		\lim_{(s,\xi) \to (0,x)}|u_r(s,\xi)|=0.
	\end{equation*}
	This shows that $\lim_{(s,\xi) \to (0,x)}u(s,\xi)=u(0,x)$ and then $u \in C(\R_0^+ \times \R)$ (and then, in particular, belongs to $C(\overline{E})$). \\
	Next, fix $t>0$ and $x \in \R \setminus \{\varphi(t)\}$ and observe that, arguing as in Proposition \ref{prop:derqphix}, there exists a compact set $K \subset \R_0^+ \times \R$ such that the curve $w \in [0,t] \mapsto (t-w,x-\varphi(w)) \in \R_0^+ \times \R$ lies in $\mathring{K}$ and $(0,0) \not \in K$. Furthermore, there exists a $\delta>0$ such that for $\xi \in [x-\delta,x+\delta]=:I_\delta$ the curves $w \in [0,t] \mapsto (t-w,\xi-\varphi(w)) \in \R_0^+ \times \R$ lie into $\mathring{K}$. Hence, we have that, for $\xi \in I_\delta$,
	\begin{equation*}
		\Norm{\partial_x r_\Phi(t,\cdot;y)}{L^\infty[x-\delta,x+\delta]} \le \Norm{\partial_x p_\Phi(\cdot,\cdot;0)}{L^\infty(K)},
	\end{equation*} 
	where the right-hand side is independent of $y$. Furthermore, for $y \in {\sf supp}(f)$ and $\xi \in I_\delta$ it is clear that $\xi-y \in I_\delta-{\sf supp}(f)$, where the latter is the Minkowski sum of two compact subsets of $\R$ and thus it is compact. As a consequence
	\begin{equation*}
		\Norm{\partial_x q_\Phi(t,\cdot;y)}{L^\infty(I_\delta)} \le \Norm{\partial_x p_\Phi(t,\cdot;0)}{L^\infty(I_\delta-{\sf supp}(f))}+\Norm{\partial_x p_\Phi(\cdot,\cdot;0)}{L^\infty(K)}.
	\end{equation*}
	Hence, a simple application of the dominated converence theorem guarantees that for $x \not = \varphi(t)$
	\begin{equation*}
		\partial_x u(t,x)=\int_{{\sf supp}(f)}\partial_x q_\Phi(t,x;y)f(y)\, dy
	\end{equation*}
	and $\partial_x u(t,\cdot) \in C(\R \setminus \{\varphi(t)\})$.

	\textcolor{black}{Notice, in particular, that the same argument shows that
	\begin{equation}\label{eq:derj}
		\partial_x u_r(t,x)=-\int_{{\sf supp}(f)}\partial_x r_\Phi(t,x;y)f(y)\, dy.
	\end{equation}
	To show that $u(t,\cdot) \in C^2(E_2(t))$ for all $t \in E_1$, we now prove that
	\begin{equation}\label{eq:secondderur}
		\partial_x^2 u_r(t,x)=-\int_{{\sf supp}(f)}\partial^2_x r_\Phi(t,x;y)f(y)\, dy
	\end{equation}
	and that $u_r(t,\cdot) \in C^2(E_2(t))$. This will imply, together with Theorem \ref{thm:Dini1}, that $u(t,\cdot) \in C^2(E_2(t))$. Let us distinguish two cases. If $x<\varphi(0)$, then we consider $\delta>0$ such that $x+\delta<\varphi(0)$ and $0<t_1<t<t_2$. Then, for $w \in [0,t]$ we have $t-w \in [0,t_2]$ and then we notice that, for any $t^\prime \in [t_1,t_2]$ and $x^\prime \in [x-\delta,x+\delta]$.
	\begin{equation*}
		\left|\partial_x^2 r_\Phi(t^\prime,x^\prime;y)\right| \le \sup_{\substack{\tau \in [0,t_2] \\ z \in [x-\delta,x+\delta] \\ h \in [\varphi(0),\varphi(t_2)]}}\left|\partial_x^2 p_\Phi(\tau,z,h)\right|<\infty.
	\end{equation*}
	If $x \in (\varphi(0),\varphi(t))$, then consider $\delta>0$ such that $x+\delta<\varphi(t-\delta)$. Since $\varphi \in C^1[0,t)$ and $\varphi'(0)>0$, there exists $\varepsilon>0$ and a function $C^1(-\varepsilon,t)$, that we still denote $\varphi$, has positive derivative on $(-\varepsilon, t)$ and extends $\varphi$. We can then suppose that $x-\delta>\varphi(-\varepsilon)$. Furthermore, let $t_\star=\frac{\varphi^{-1}(x+\delta)+t-\delta}{2}$. We split the integral defining $\partial_x^2 r_\Phi$ as follows:
	\begin{equation*}
		\partial_x^2 r_\Phi(t,x;y)=\int_0^{t_\star}\partial_x^2 p_\Phi(t-w,x;\varphi(w))p_{\mathcal{T}}(w;y)\, dw+\int_{t_\star}^{t}\partial_x^2 p_\Phi(t-w,x;\varphi(w))p_{\mathcal{T}}(w;y)\, dw-\frac{2 \overline{\nu}(t-\varphi^{-1}(x))}{\varphi'(\varphi^{-1}(x))}p_{\mathcal{T}}(\varphi^{-1}(x);y).
	\end{equation*} 
	To control the first integral, notice that as $w \in [0,t_\star]$, then for any $t^\prime \in [t-\delta,t+\delta]$ $t^\prime-w \in [\frac{t-\delta-\varphi^{-1}(x+\delta)}{2},t+\delta]=:I$, where $\frac{t-\delta-\varphi^{-1}(x+\delta)}{2}>0$. Hence we can set, using the fact that $\partial_x^2 p_\Phi \in L^\infty_{\sf loc}(\R^+ \times \R^2)$,
	\begin{equation*}
		M_1:=\sup_{\substack{\tau \in I \\ z \in [x-\delta,x+\delta] \\ h \in [\varphi(0),\varphi(t+\delta)]}}\left|\partial_x^2 p_\Phi(\tau,z;h)\right|<\infty.
	\end{equation*}
	Next, for the second integral, we notice that as $w \in [t_\star,t]$, then $\varphi(w) \in [\varphi(t_\star),\varphi(t)]$. However, since $t_\star>\varphi^{-1}(x+\delta)$, we have that $\varphi(t_\star)>x+\delta$ and thus for all $z \in [x-\delta,x+\delta]$ and $w \in [t_\star,t]$ the point $(z,\varphi(w))$ is separated from the diagonal. Let $K$ be a compact neighbourhood of $\{(z,\varphi(w)): \ z \in [x-\delta,x+\delta], \ w \in [t_\star,t]\}$ such that $K \cap {\sf diag}(\R^2)=\emptyset$. Furthermore, notice that for all $t^\prime \in [t-\delta,t+\delta]$ and $w \in [t_\star,t^\prime]$ it holds $t^\prime-w \in [0,t+\delta]$. Then, by \eqref{eq:locunifpartpphi2}, we have that
	\begin{equation*}
		M_2:=\sup_{\substack{\tau \in [0,t+\delta] \\ (z,h) \in K}}\left|\partial_x^2 p_\Phi(\tau,z;h)\right|<\infty.
	\end{equation*}
	Finally, let
	\begin{equation*}
		\varphi^\prime_\star:=\min_{z \in [x-\delta,x+\delta]}\varphi'(\varphi^{-1}(x))>0
	\end{equation*}
	so that for all $t^\prime \in [t-\delta,t+\delta]$, $z \in [x-\delta,x+\delta]$ and $y \in {\sf supp}(f)$, it holds
	\begin{equation*}
		M_3:=\frac{2\overline{\nu}(t-\delta-\varphi^{-1}(x+\delta))}{\varphi^\prime_\star}S([t-\delta,t+\delta];{\sf supp}(f))<\infty.
	\end{equation*}
	Using the definition of $M_j$ for $j=1,2,3$, we have for all $\tau \in [t-\delta,t+\delta]$, $z \in [x-\delta,x+\delta]$ and $y \in {\sf supp}(f)$
	\begin{equation*}
		\left|\partial_x^2 r_\Phi(\tau,z;y)\right| \le M_1+M_2+M_3.
	\end{equation*}
	With this in mind, \eqref{eq:secondderur} and the fact that $u_r(t,\cdot) \in C^2(E_2(t))$ come from a simple application of the dominated convergence theorem.}
	
	 \textcolor{black}{Next, we want to show that $\partial_t^\Phi u(t,x)$ is well-defined for any $t>0$ and $x \in E_2(t)$. To do this, let us split again $u$ as in \eqref{eq:solusplit}, so that, if the involved quantities exist,
	\begin{equation*}
		\partial_t^\Phi u(t,x)=\partial_t^\Phi \widetilde{u}(t,x)+\partial_t^\Phi u_r(t,x).
	\end{equation*}
	We notice that, by Theorem \ref{thm:Dini1}, $\partial_t^\Phi \widetilde{u} \in C(\R^+ \times \R)$ and $\partial_t^\Phi \widetilde{u}(\cdot,x) \in L^1_{\sf loc}(\R^+_0)$ for $x \in \R$. Hence, we only need to prove that $\partial_t^\Phi u_r(t,x)$ is well-defined for $t>0$ and $x \in E_2(t)$. To do this, notice that, since $u_r(0,\cdot)\equiv 0$, we have
	\begin{equation*}
		\partial_t^\Phi u_r(t,x)=-\partial_t \left(\int_{0}^t \int_{\R}\overline{\nu}(t-s) r_\Phi(s,x;y)f(y)\, dy\, ds\right).
	\end{equation*}
	Let us first show that we can exchange the order of the integrals. Indeed, by \eqref{eq:boundrPhi} we have
	\begin{equation*}
	\int_{0}^t \int_{\R}\overline{\nu}(t-s) r_\Phi(s,x;y)|f(y)|\, dy\, ds	\le |{\sf supp}(f)|\Norm{f}{L^\infty(\R)}\frac{S_{\mathcal{T}}(\{y\},1)}{\sqrt{2\pi}}I_\Phi(t)\int_0^t U_{-\frac{1}{2}}(w)\, dw.
	\end{equation*}
	Hence we can rewrite
	\begin{equation*}
		\partial_t^\Phi u_r(t,x)=-\partial_t \left(\int_{\R}f(y)\left(\int_{0}^t \overline{\nu}(t-s) r_\Phi(s,x;y)\, ds\right)\, dy\right).
	\end{equation*}
	Now we want to prove that we can take the derivative inside the integral sign. To do this, it is sufficient to provide, for fixed $x$, a bound for $|\partial_t^\Phi r_\Phi(\cdot,x;y)|$ in a suitable neighbourhood of any $t>0$ with $\varphi(t)>x$ that is independent of $y \in {\sf supp}(f)$. Fix $y \in {\sf supp}(f)$ and $x<\varphi(t)$ with $x \not = y$. We notice that, by the assumptions on $\varphi$, there exists $\delta>0$ such that $x<\varphi(\tau)$ for all $\tau \in [t-\delta,t+\delta]$ where $t-\delta>0$. Let $k_\Phi$ as in \eqref{eq:kPhi} and recall that
	\begin{equation*}
		\partial_t^\Phi r(t,x;y)=\int_0^{t}\partial_t k_\Phi(t,w,x)p_{\mathcal{T}}(w;y)\, dw.
	\end{equation*}
	Arguing as in Proposition \ref{prop:timeder}, we know that $\partial_t k_\Phi(\cdot,\cdot,x)$ is continuous and \linebreak $\sup_{(\tau,w) \in [t-\delta,t+\delta] \times (0,)}|\partial_t k_\Phi(\tau,w,x)|<\infty$. Hence 
	\begin{equation*}
		\sup_{(\tau,y) \in [t-\delta,t+\delta]\times {\sf supp}(f)}|\partial_t^\Phi r(\btrev{\tau},x;y)| \le \sup_{(\tau,w) \in [t-\delta,t+\delta] \times \R^+}|\partial_t k_\Phi(\tau,w,x)|.
	\end{equation*}
	Hence, by dominated convergence we get
	\begin{equation*}
		\partial_t^\Phi u_r(t,x)=-\int_{\R} f(y)\partial_t^\Phi r_\Phi(t,x;y)\, dy.
	\end{equation*}
	By Remark \ref{rmkcont2}, the same dominated convergence argument proves that $\partial_t^\Phi u_r(\cdot,x) \in C(E_1(x))$. Hence, we have shown that $\partial_t^\Phi u(t,x)$ is well defined for any $t>0$ and $x<\varphi(t)$ and belongs to $C(E_1(x))$ for any $x \in E_2$. 
	}

	Next, we verify that $u$ satisfies \eqref{eq:nonlocmovb}. \textcolor{black}{Indeed, let us split again $u$ as in \eqref{eq:solusplit} and notice that, by Theorem~\ref{thm:Dini1}
		\begin{equation}\label{eq:eqtildeu}
			\partial_t^\Phi \widetilde{u}(t,x)=\frac{1}{2} \partial_x^2 \widetilde{u}(t,x), \ (t,x) \in \R^+ \times \R.
		\end{equation}
	On the other hand, by} \eqref{eq:eqphir}
	\begin{equation}\label{eq:equr}
		\partial_t^\Phi u_r(t,x)=-\int_{\R}\partial_t^\Phi r_\Phi(t,x;y)f(y)\, dy=-\frac{1}{2}\int_{\R}f(y)\partial_x^2 r_\Phi(t,x;y) \, dy=\frac{1}{2}\partial_x^2 u_r(t,x).
	\end{equation}
\btrev{	Combining \eqref{eq:eqtildeu} and \eqref{eq:equr} we have
	\begin{equation*}
		\partial_t^\Phi u(t,x)=\frac{1}{2}\partial_x^2 u(t,x), \quad t>0, \ x<\varphi(t).
	\end{equation*}
}
	Furthermore, if $x \ge \varphi(t)$, then, by the second equality in \eqref{eq:nonlocmovbq}, it holds $u(t,x)=0$. The condition $u(0,x)=f(x)$ is guaranteed by assumption. Next, we have, for any $T>0$,
	\begin{equation*}
		\sup_{t \in [0,T]}|u(t,x)| \le |{\sf supp}(f)|\Norm{f}{L^\infty(\R)}\sup_{y \in {\sf supp}(f)}\sup_{t \in [0,T]}q_\Phi(t,x;y)
	\end{equation*}
	and then, taking the limit as $x \to -\infty$, by the third equality in \eqref{eq:nonlocmovbq}, we get
	\begin{equation*}
		\lim_{x \to -\infty}\sup_{t \in [0,T]}|u(t,x)|=0.
	\end{equation*}
	This proves that $u$ satisfies \eqref{eq:nonlocmovb}.
		
	Finally, we need to show that $\partial_t^\Phi u(\cdot,x) \in L^1_{\rm loc}(\btrev{\overline{E_1(x)}})$ for all $x \in E_2 \setminus \{\varphi(0)\}$. To do this, let us split again $u$ as in \eqref{eq:solusplit} and notice that
	\begin{equation*}
		\partial_t^\Phi u(t,x)=\partial_t^\Phi \widetilde{u}(t,x)+\partial_t^\Phi u_r(t,x).
	\end{equation*}
	Concerning $\partial_t^\Phi \widetilde{u}(t,x)$, \textcolor{black}{we already know that $\partial_t^\Phi \widetilde{u}(\cdot,x) \in L^1_{\rm loc}(\overline{E_1(x)})$ for all $x \in E_2$ by Theorem~\ref{thm:Dini1}.}
	 Concerning $u_r$ let us distinguish among two cases. If $x<\varphi(0)$, then consider any $T>0$. Notice that
	\begin{align*}
		\int_0^T &\int_0^t |\partial_t^\Phi p_\Phi(t-w,x;\varphi(w))|p_{\mathcal{T}}(w;y)\, dw \, dt=\int_0^T \left(\int_w^T  \left|\partial_t^\Phi p_\Phi(t-w,x;\varphi(w))\right|\, dt\right) p_{\mathcal{T}}(w;y)\, dw\\
		&= \frac{1}{2}\int_0^T \left(\int_w^T  \left|\partial_x^2 p_\Phi(t-w,x;\varphi(w))\right|\, dt\right) p_{\mathcal{T}}(w;y)\, dw\\
		&\le \frac{1}{2}\sqrt{\frac{2}{\pi}}\left(\left(\frac{10}{e}\right)^{\frac{5}{2}}+\left(\frac{6}{e}\right)^{\frac{3}{2}}\right)\int_0^T (x-\varphi(w))^{-3} \left(\int_w^T \E\left[e^{-\frac{(x-\varphi(w))^2}{4L(t)}}\right] \, dt\right) p_{\mathcal{T}}(w;y)\, dw\\
		&\le \frac{(x-\varphi(0))^{-3}T}{2}\sqrt{\frac{2}{\pi}}\left(\left(\frac{10}{e}\right)^{\frac{5}{2}}+\left(\frac{6}{e}\right)^{\frac{3}{2}}\right),
	\end{align*}
\btrev{where we used the estimates \eqref{estim23}.}
		Integrating this against $|f(y)|$ we get that $u_r(\cdot,x) \in L^1_{\rm loc}(\btrev{\overline{E_1(x)}})$. If $\varphi$ is not constant and \textcolor{black}{$x>\varphi(0)$, then let $\varphi^{-1}(x)=\{w_x\}$ and observe that $w_x>0$. Let also $k_\Phi$ be as in \eqref{eq:kPhi} and recall \eqref{eq:finalequality}, so that
			\begin{equation*}
				\partial^\Phi_t u_r(t,x)=-\int_{-\infty}^{\varphi(0)}\int_0^{+\infty}\partial_t k_\Phi(t,w,x)p_{\cT}(w;y)f(y)\, dw\, dy.
			\end{equation*}
		Hence we have, for any $T>w_x$,
		\begin{align*}
			\int_{w_x}^{T}|\partial^\Phi_t u_r(t,x)|\, dt &\le \int_{-\infty}^{\varphi(0)}\int_0^{+\infty}\left|\partial_t k_\Phi(t,w,x)\right|p_{\cT}(w;y)|f(y)|\, dw\, dy\\
			&\le \Norm{f}{L^\infty(\R)}\int_{{\sf supp}(f)}\int_0^{+\infty}\left|\partial_t k_\Phi(t,w,x)\right|p_{\cT}(w;y)\, dw\, dy.
		\end{align*}
		Now let recall that if we define $\partial_t k_\Phi(t,w,x)$ as in \eqref{eq:rhsbound} and $\partial_t k_\Phi(t,t,x)=0$, where both the limits are uniform in $t \in [w_x,T]$, then the function $\partial_t k_\Phi(\cdot,\cdot,x)$ is continuous on $(t,w) \in [w_x,T] \times \R$. Furthermore, we recall that for $w \ge t$ we have $\partial_t k_\Phi(t,w,x)=0$, while \eqref{eq:partialtkPhi2} holds for $w<t$. Hence we finally get
		\begin{align*}
			\int_{w_x}^{T}|\partial^\Phi_t u_r(t,x)|\, dt 
			&\le \Norm{f}{L^\infty(\R)}\left(\sup_{(t,w) \in [w_x,T] \times (0,+\infty)}k_\Phi(t,w,x)\right)|{\sf supp}(f)|,
		\end{align*}
		that implies $\partial_t^\Phi u_r(\cdot,x) \in L^1_{\sf loc}(\overline{E_1(x)})$.}

\end{proof}
\textcolor{black}{\begin{rmk}
	Notice that Dini-continuity has been only used to guarantee that $\partial_t^\Phi \widetilde{u}(\cdot,x) \in L^1_{\sf loc}(\overline{E_1(x)})$. Theorem \ref{SPthm} still holds if $f$ is not Dini-continuous, except for condition (4).
\end{rmk}}
Concerning uniqueness, let us first show the following general result.
\begin{thm}\label{thm:unique1}
	Under the assumptions of Theorem \ref{SPthm}, there exists at most one function $u:\R_0^+ \times \R \to \R$ satisfying items $(1),(2),(3)$ and $(5)$ of Theorem \ref{SPthm} and
	\begin{itemize}
		\item[$(4^\prime)$] For all $x \in E_2$ it holds $\partial_t^\Phi u(\cdot,x) \in L^1_{\rm loc}(\overline{E_1(x)})$.
	\end{itemize}
\end{thm}
\begin{proof}
	Since the equation is linear, it is sufficient to prove the statement for $f \equiv 0$. Precisely, we only need to prove that if $u$ satisfies  $(1),(2),(3),(4^\prime),(5)$ then $u \equiv 0$. This is clear since in our case the set $E=\{(t,x) \in \R^+_0 \times \R: \ x<\varphi(t)\}$ is unbounded and $\lim_{x \to -\infty}u(t,x)=0$ locally uniformly with respect to $t\ge 0$, hence we can apply Theorem \ref{thm:weakmax} to both $u$ and $-u$ to guarantee that
	\begin{equation*}
		\inf_{(t,x) \in \overline{E}}u(t,x)=\sup_{(t,x) \in \overline{E}}u(t,x)=0.
	\end{equation*}
\end{proof}
A plain corollary, that implies uniqueness in Theorem \ref{thm:mainfrac} for constant $\varphi$, is the following one.
\begin{coro}
	Under the assumptions of Theorem \ref{SPthm}, if $\varphi$ is constant, then \eqref{eq:usol} is the unique solution of \eqref{eq:nonlocmovb} satisfying items $(1),(2),(3)$ and $(5)$ of Theorem \ref{SPthm} and $(4^\prime)$ of Theorem \ref{thm:unique1}.
\end{coro}
\begin{proof}
	Just notice that in such a case $\varphi(0)\not \in E_2$ and use Theorems \ref{SPthm} and \ref{thm:unique1}.
\end{proof}
\begin{rmk}
	If $\varphi$ is constant, notice that $u(t,\varphi(0))=0$, that implies
	\begin{equation*}
		\partial_t^\Phi u_r(\cdot,\varphi(0))=-\partial_t^\Phi \widetilde{u}(\cdot,\varphi(0)) \in L^1_{\rm loc}(\R_0^+).	
	\end{equation*}
	Hence, in such a case, $\partial_t^\Phi u_r(\cdot,x) \in L^1_{\rm loc}(\R_0^+)$ for all $x \in \R$.
\end{rmk}
Now, let us consider the case in which $\varphi$ is non-constant (and thus non-decreasing). In such a case, we need to require the additional regularity given by conditions \eqref{eq:integr01} and \eqref{eq:varphicond} below: however, these conditions are not too restrictive and can be checked quite easily, as we show in Propositions \ref{prop:suffcond1} and \ref{prop:suffcond} below.
\begin{coro}\label{coro:unique2}
	Suppose the assumptions of Theorem \ref{SPthm} hold true. Assume further that
	\begin{equation}\label{eq:integr01}
		\int_0^1 \frac{|\Psi(\xi)|}{\sqrt{\Re(\Psi(\xi))}}\, d\xi<\infty
	\end{equation}
	and that for any $y<\varphi(0)$ and any fixed threshold $c>0$ it holds
	\begin{equation}\label{eq:varphicond}
		\E_y\left[(\varphi(\mathcal{T}_c)-\varphi(0))^{-\frac{4+\gamma}{2-\gamma}}\right]<\infty,
	\end{equation}
	where \btrev{$\gamma$ is that in Assumption \ref{orcond} and}
	\begin{equation*}
		\mathcal{T}_c:=\inf\{t>0: \ X(t) \ge c\}.
	\end{equation*}
	Then \eqref{eq:usol} is the unique solution of \eqref{eq:nonlocmovb} satisfying items $(1),(2),(3)$ and $(5)$ of Theorem \ref{SPthm} and $(4^\prime)$ of Theorem \ref{thm:unique1}.
\end{coro} 
\begin{proof}
	We only need to prove that $u$ as in \eqref{eq:usol} is such that $\partial_t^\Phi u(\cdot,\varphi(0)) \in L^1_{\rm loc}(\overline{E_1(\varphi(0))})$, so that the statement follows by Theorems \ref{SPthm} and \ref{thm:unique1}. To do this, we split $u$ as in \eqref{eq:solusplit} and we notice that the condition $\partial_t^\Phi \widetilde{u}(\cdot,\varphi(0)) \in L^1_{\rm loc}(\R_0^+)$ has been already proved in Theorem \ref{SPthm}. Hence, we only need to show that $\partial_t^\Phi u_r(\cdot,\varphi(0)) \in L^1_{\rm loc}(\R_0^+)$. However, since $\partial_t^\Phi u_r(\cdot,\varphi(0)) \in C(\R^+)$, it is sufficient to show that $\partial_t^\Phi u_r(\cdot,\varphi(0)) \in L^1[0,1]$. We set
		\begin{align*}
			\mathcal{P}_1(s)&=\sqrt{2}\int_0^{M_\gamma}|\Psi(\xi)|e^{-s\Re(\Psi(\xi))}\, d\xi \\
			\mathcal{P}_2(s)&= \int_{M_\gamma}^{+\infty}\left(3\overline{\nu}(1)+\xi \int_0^1 \tau \nu(d\tau)+\frac{\xi^2}{2}\int_0^1 \tau^2\nu(d\tau)\right)e^{-s\btrev{C_\gamma}\xi^{2-\gamma}}\, d\xi
		\end{align*}
		and we notice that \btrev{(recall eq. \eqref{diffPhipPhi})}
		\begin{align*}
			\int_{0}^{1}&\left|\partial_t^\Phi u_r(t,\varphi(0))\right|\, dt \\
			& \le  \int_0^{1}\int_{-\infty}^{\varphi(0)}\int_0^t \int_0^{+\infty} p(s,\varphi(0)-\varphi(w);0)\left|\partial_s f_L(s;t-w)\right| p_{\mathcal{T}}(w;y)|f(y)|  \, ds \, dw \, dy \, dt\\
			&= \int_{-\infty}^{\varphi(0)}\int_0^{1} \int_0^{+\infty} p(s,\varphi(0)-\varphi(w);0) \left(\int_w^{1} \left|\partial_s f_L(s;t-w)\right|\, dt\right) p_{\mathcal{T}}(w;y)|f(y)| \, ds \, dw  \, dy\\
			&\le \frac{1}{\pi}\int_{-\infty}^{\varphi(0)}\int_0^{1} \int_0^{+\infty} p(s,\varphi(0)-\varphi(w);0)(\mathcal{P}_1(s)+\mathcal{P}_2(s))\left(\int_w^{1} I_\Phi(t-w)\, dt\right) p_{\mathcal{T}}(w;y)|f(y)| \, ds \, dw  \, dy\\
			&\le \frac{\Norm{I_\Phi}{L^1[0,1]}\Norm{f}{L^\infty(\R)}}{\pi}\int_{{\sf supp}(f)}\int_0^{1} \int_0^{+\infty}(\mathcal{P}_1(s)+\mathcal{P}_2(s)) p(s,\varphi(0)-\varphi(w);0) p_{\mathcal{T}}(w;y)\, ds  \, dw \, dy\\
			&=\frac{\Norm{I_\Phi}{L^1[0,1]}\Norm{f}{L^\infty(\R)}}{\pi}\int_{{\sf supp}(f)}\int_0^{1} \int_0^{+\infty}\mathcal{P}_1(s) p(s,\varphi(0)-\varphi(w);0) p_{\mathcal{T}}(w;y)\, ds  \, dw \, dy\\
			&+\frac{\Norm{I_\Phi}{L^1[0,1]}\Norm{f}{L^\infty(\R)}}{\pi}\int_{{\sf supp}(f)}\int_0^{1} \int_0^{+\infty}\mathcal{P}_2(s) p(s,\varphi(0)-\varphi(w);0) p_{\mathcal{T}}(w;y)\, ds  \, dw \, dy\\
			&=:I_1+I_2,
		\end{align*}
		where we also used \eqref{upbound2}. To prove that $I_1<\infty$, notice that
		\begin{align*}
			I_1 & \le \frac{\Norm{I_\Phi}{L^1[0,1]}\Norm{f}{L^\infty(\R)}|{\sf supp}(f)|S_{\mathcal{T}}({\sf supp}(f);1)}{\pi\sqrt{\pi}} \int_0^{M_\gamma}\int_0^{+\infty}s^{-\frac{1}{2}}|\Psi(\xi)|e^{-s\Re(\Psi(\xi))}\, ds \, d\xi\\
			&=\frac{\Norm{I_\Phi}{L^1[0,1]}\Norm{f}{L^\infty(\R)}|{\sf supp}(f)|S_{\mathcal{T}}({\sf supp}(f);1)}{\pi} \int_0^{M_\gamma}\frac{|\Psi(\xi)|}{\sqrt{\Re(\Psi(\xi))}} d\xi<\infty.
		\end{align*}
		Concerning $I_2$, notice that there exists a constant $C>0$ such that
		\begin{equation*}
			\mathcal{P}_2(s) \le C \int_{M_\gamma}^{+\infty}\xi^2e^{-s\btrev{C_\gamma}\xi^{2-\gamma}}\, d\xi
		\end{equation*}	
		and then
		\begin{equation*}
			I_2 \le C\frac{\Norm{I_\Phi}{L^1[0,1]}\Norm{f}{L^\infty(\R)}}{\pi}\int_{{\sf supp}(f)}\int_0^{1} \int_{M_\gamma}^{+\infty}\xi^2 \left(\int_0^{+\infty} e^{-s\btrev{C_\gamma}\xi^{2-\gamma}} p(s,\varphi(0)-\varphi(w);0)\, ds\right) p_{\mathcal{T}}(w;y) \, d\xi \, dw \, dy
		\end{equation*}
		Now recall that, by \eqref{notint},
		\begin{equation*}
			\int_0^{+\infty}e^{-s\btrev{C_\gamma}\xi^{2-\gamma}}p(s,\varphi(0)-\varphi(w);0)\, ds=\frac{1}{\sqrt{2\btrev{C_\gamma}\xi^{2-\gamma}}}e^{-(\varphi(w)-\varphi(0))\sqrt{2}\xi^{1-\frac{\gamma}{2}}}
		\end{equation*}
		hence
		\begin{align*}
			I_2 &\le C\frac{\Norm{I_\Phi}{L^1[0,1]}\Norm{f}{L^\infty(\R)}}{\pi\sqrt{2\btrev{C_\gamma}}}\int_{{\sf supp}(f)}\int_0^{1} \left(\int_{0}^{+\infty}\xi^{1+\frac{\gamma}{2}} e^{-(\varphi(w)-\varphi(0))\sqrt{2}\xi^{1-\frac{\gamma}{2}}}\, d\xi \right) p_{\mathcal{T}}(w;y) \, dw \, dy\\
			&\le \frac{C\Norm{I_\Phi}{L^1[0,1]}\Norm{f}{L^\infty(\R)}}{(2-\gamma)\pi \btrev{2^{\frac{1+\gamma}{2-\gamma}}}\btrev{\sqrt{C_\gamma}}}\Gamma\left(\frac{4+\gamma}{2-\gamma}\right)\int_{{\sf supp}(f)}\int_0^{1} (\varphi(w)-\varphi(0))^{-\frac{4+\gamma}{2-\gamma}}p_{\mathcal{T}}(w;y) \, dw \, dy\\
			&\le \frac{C\Norm{I_\Phi}{L^1[0,1]}\Norm{f}{L^\infty(\R)}}{(2-\gamma)\pi \btrev{2^{\frac{1+\gamma}{2-\gamma}}}\btrev{\sqrt{C_\gamma}}}\Gamma\left(\frac{4+\gamma}{2-\gamma}\right)\int_{{\sf supp}(f)}\E_y\left[(\varphi(\mathcal{T})-\varphi(0))^{-\frac{4+\gamma}{2-\gamma}}\right] \, dy.
		\end{align*}
		Now let
		\begin{equation*}
			\mathcal{T}_{\varphi(0)}=\inf\{t>0: \ X(t) \ge \varphi(0)\}
		\end{equation*}
		and observe that $\mathcal{T}_{\varphi(0)} \le \mathcal{T}$ a.s. Thus
		\begin{align*}
			I_2 &\le \frac{C\Norm{I_\Phi}{L^1[0,1]}\Norm{f}{L^\infty(\R)}}{(2-\gamma)\pi \btrev{2^{\frac{1+\gamma}{2-\gamma}}}\btrev{\sqrt{C_\gamma}}}\Gamma\left(\frac{4+\gamma}{2-\gamma}\right)\int_{{\sf supp}(f)}\E_y\left[(\varphi(\mathcal{T}_{\varphi(0)})-\varphi(0))^{-\frac{4+\gamma}{2-\gamma}}\right] \, dy.
		\end{align*}
		Next, define
		\begin{equation*}
				\mathcal{T}_{\varphi(0)}(y)=\inf\{t>0: \ X(t)+y \ge \varphi(0)\}
		\end{equation*}
		so that
		\begin{equation*}
			\E_y\left[(\varphi(\mathcal{T}_{\varphi(0)})-\varphi(0))^{-\frac{4+\gamma}{2-\gamma}}\right]=\btrev{\E_0}\left[(\varphi(\mathcal{T}_{\varphi(0)}(y))-\varphi(0))^{-\frac{4+\gamma}{2-\gamma}}\right].
		\end{equation*}
		Furthermore, setting $\overline{y}=\max {\sf supp}(f)$, we have $\mathcal{T}_{\varphi(0)}(\overline{y}) \le \mathcal{T}_{\varphi(0)}(y)$ a.s. As a consequence
		\begin{align*}
			I_2 &\le \frac{C\Norm{I_\Phi}{L^1[0,1]}\Norm{f}{L^\infty(\R)}|{\sf supp}(f)|}{(2-\gamma)\pi \btrev{2^{\frac{1+\gamma}{2-\gamma}}}\btrev{\sqrt{C_\gamma}}}\Gamma\left(\frac{4+\gamma}{2-\gamma}\right)\E\left[(\varphi(\mathcal{T}_{\varphi(0)}(\overline{y}))-\varphi(0))^{-\frac{4+\gamma}{2-\gamma}}\right]<\infty
		\end{align*}
		by assumption.
	\end{proof}
	{\begin{rmk}
		With the same exact proof, under the assumptions of Corollary \ref{coro:unique2}, we also have continuity with respect to initial conditions. Indeed, consider $u_1$ and $u_2$ to be the unique solutions, satisfying items $(1)$, $(2)$, $(3)$ and $(5)$ of Theorem \ref{SPthm} and $(4^\prime)$ of Theorem \ref{thm:unique1}, of
		\begin{equation*}
			\begin{cases}
				\displaystyle \partial_t^\Phi u_j(t,x)=\frac{1}{2}\partial^2_xu_j(t,x) & t>0, \ x<\varphi(t)\\
				u_j(t,x)=0 & t \ge 0, \ x\ge \btrev{\varphi(t)}\\
				u_j(0,x)=f_j(x) & x<\varphi(0)\\
				\lim_{x \to -\infty}u_j(t,x)=0 & \mbox{locally uniformly with respect to $t>0$},
			\end{cases}
		\end{equation*}
		where $f_j \in C_{\rm c}(-\infty,\varphi(0))$ for $j=1,2$. Then, by linearity of the involved operators, $w=u_1-u_2$ is the unique solution of
		\begin{equation*}
			\begin{cases}
				\displaystyle \partial_t^\Phi w(t,x)=\frac{1}{2}\partial^2_xw(t,x) & t>0, \ x<\varphi(t)\\
				w(t,x)=0 & t \ge 0, \ x\ge \btrev{\varphi(t)}\\
				w(0,x)=f_1(x)-f_2(x) & x<\varphi(0)\\
				\lim_{x \to -\infty}w(t,x)=0 & \mbox{locally uniformly with respect to $t>0$},
			\end{cases}
		\end{equation*}
		satisfying items $(1)$, $(2)$, $(3)$ and $(5)$ of Theorem \ref{SPthm} and $(4^\prime)$ of Theorem \ref{thm:unique1}. Hence, by the positive maximum principle given in Theorem \ref{thm:weakmax}, we get that for all $T>0$ it holds
		\begin{equation*}
			\sup_{(t,x) \in [0,T]\times \R}|u_1(t,x)-u_2(t,x)|=\sup_{(t,x) \in [0,T]\times \R}|w(t,x)|=\max_{\btrev{x<\varphi (0)}}|f_1(x)-f_2(x)|.
		\end{equation*}
	\end{rmk}}
	We state here some sufficient conditions that guarantee \eqref{eq:integr01} and \eqref{eq:varphicond}.
	\begin{prop}\label{prop:suffcond1}
		Suppose the assumptions of Theorem \ref{SPthm} hold true. Assume further that
		\begin{equation*}
				\int_1^{+\infty} \log(t)\nu(dt)<\infty;
			\end{equation*}
		Then \eqref{eq:integr01} is satisfied.	
	\end{prop}
	\begin{proof}
		Let us observe that, by Lemma \ref{lemmaintzero}, for $\xi \in (0,1)$,
		\begin{equation*}
			\frac{|\Psi(\xi)|}{\sqrt{\Re(\Psi(\xi))}} \le \sqrt{\Re(\Psi(\xi))}+\frac{1}{\btrev{\sqrt{C_0}}}\frac{|\Im(\Psi(\xi))|}{\xi}.
		\end{equation*}
		Since $\Re(\Psi(\cdot))$ is continuous,		
		it is clear that
		\begin{equation*}
			\int_0^1 \sqrt{\Re(\Psi(\xi))}\, d\xi<\infty. 
		\end{equation*}
		On the other hand,
		\begin{align*}
			&\int_0^1 \frac{|\Im(\Psi(\xi))|}{\xi}\, d\xi \le \int_0^1\int_0^{+\infty}\frac{|\sin(\xi t)|}{\xi}\nu(dt)\, d\xi =\int_0^{+\infty} t \left(\int_0^1 \frac{|\sin(\xi t)|}{\xi t}\, d\xi \right) \nu(dt)\\
			&=\int_0^{+\infty} \left(\int_0^t \frac{|\sin(z)|}{z}\, dz \right) \nu(dt)=I_1+I_2+I_3
	\end{align*}
	where
	\begin{equation*}
		I_1=\int_0^{1} \left(\int_0^t \frac{|\sin(z)|}{z}\, dz \right) \nu(dt) \qquad I_2=\int_1^{+\infty} \left(\int_0^1 \frac{|\sin(z)|}{z}\, dz \right) \nu(dt)
	\end{equation*}
	\begin{equation*}
		I_3=\int_1^{+\infty} \left(\int_1^t \frac{|\sin(z)|}{z}\, dz \right) \nu(dt).
	\end{equation*}	
	Now recall that there exists a constant $C>0$ such that for $z \in (0,1)$
	\begin{equation*}
		\frac{|\sin(z)|}{z} \le C
	\end{equation*}	
		and then
	\begin{equation*}
		I_1 \le C \int_0^1 t\nu(dt)<\infty, \qquad I_2 \le C \overline{\nu}(1)<\infty.
	\end{equation*}	
	Concerning $I_3$, notice that
	\begin{equation*}
		I_3 \le \int_1^{+\infty} \int_1^t \frac{1}{z}\, dz \, \nu(dt)=\int_1^{+\infty}\log(t)\nu(dt)<\infty.
	\end{equation*}
	\end{proof}
	\begin{prop}\label{prop:suffcond}
		Suppose the assumptions of Theorem \ref{SPthm} hold true. Assume further that there exist two constants $\beta_1,\beta_2>0$ such that
		\begin{equation*}
			\lim_{\lambda \to +\infty}\frac{\phi^{\beta_1}(\lambda)}{\lambda}=+\infty \qquad \mbox{ and } \qquad  \liminf_{w \to 0^+}\frac{\varphi(w)-\varphi(0)}{w^{\beta_2}}>0.
		\end{equation*}
		Then \eqref{eq:varphicond} holds true.
	\end{prop}
	\begin{proof}
		Notice that there exist two constants $C,\delta>0$ such that for $w \in (0,\delta)$ it holds
		\begin{equation*}
			\varphi(w)-\varphi(0) \ge Cw^{\beta_2}
		\end{equation*}
		and then \eqref{eq:varphicond} is implied by
		\begin{equation*}
			\E\left[\mathcal{T}_c^{-\beta_2\frac{4+\gamma}{2-\gamma}}\mathbf{1}_{(0,\delta]}(\mathcal{T}_c)\right]<\infty.
		\end{equation*}
		The latter turns out to be true by employing \cite[Theorem 2.16]{annals2020}; notice that despite the assumptions considered here are lighter, the very same proof still holds.
	\end{proof}
	In case $\phi(\lambda)=\lambda^\alpha$, notice that if $\varphi^\prime(0+)>0$, then the conditions of Proposition \ref{prop:suffcond} hold with $\beta_1=\frac{2}{\alpha}$ and $\beta_2=1$. At the same time, while it is clear that the condition of Proposition \ref{prop:suffcond1} holds true, one can also verify \eqref{eq:integr01} by hands. As a consequence, combining Theorems \ref{SPthm} and \ref{thm:unique1} we get  Theorem \ref{thm:mainfrac}.
	\begin{rmk}
		Notice that under the assumptions of Theorems \ref{thm314} and \ref{thm:unique1}, one can show that $q_\phi(t,x;y)$ is the unique fundamental solution of \eqref{eq:nonlocmovb}, in the sense that if $\widetilde{q}_\phi:\R^+ \times \R \times \R \to \R$ is a continuous function such that for any $f \in C_c^\infty(-\infty,\varphi(0))$ the function $u(t,x)=\int_{-\infty}^{\varphi(0)}\widetilde{q}_\phi(t,x;y)f(y)\, dy$ satisfies \eqref{eq:nonlocmovb}, then $\widetilde{q}_\phi(t,x;y)=q_\phi(t,x;y)$ for $(t,x) \in \R^+ \times \R$ and $y \in (-\infty,\varphi(0))$. Indeed, by Theorem \ref{thm:unique1}, we get that for all $(t,x) \in \R^+\times \R$ and $f \in C^\infty_c(-\infty,\varphi(0))$ it holds
		\begin{equation*}
			\int_{-\infty}^{\varphi(0)}(\widetilde{q}_\phi(t,x;y)-q_\phi(t,x;y))f(y)\, dy=0,
		\end{equation*}
		that implies, by the Fundamental Lemma of the Calculus of Variations (see \cite[Theorem 1.24]{dac}), Theorem \ref{thm314} and the required continuity of $\widetilde{q}_\phi$ in the $y$ variable, the desired equality.
	\end{rmk}
\section{Anomalous diffusive behaviour}
\label{secmsd}
In this section, we focus on the anomalous diffusive behaviour of the processes $X_\Phi$ and $X_\Phi^\dagger$. Let us stress that the non-killed process $X$ could be an anomalous diffusion. Indeed, if we fix, without loss of generality, the initial value $y=0$, we have, by the independence of the Brownian motion $B$ and the inverse subordinator $L$ and by \cite[Proposition III.1]{bertoinb},
\begin{equation*}
	\E_0[|X_\Phi(t)|^2]=\E_0[\E_0[|B(L(t))|^2 \mid L(t)]]=\E_0[L(t)]=U_1(t) \asymp \frac{1}{\Phi(1/t)},
\end{equation*}
where $U_p$ is defined in \eqref{Up} and $f(t) \asymp g(t)$ means that there exists a constant $C>1$ such that 
\begin{equation*}
C^{-1}g(t) \le f(t) \le Cg(t).	
\end{equation*}
\textcolor{black}{Since $\sigma$ has $0$ drift, a necessary and sufficient condition in order to have a subdiffusion is as follows.
\begin{prop}
	Under Assumption \ref{ass2}, $X_\Phi$ is a subdiffusion if and only if
	\begin{equation}\label{eq:firstmoment}
		\int_{1}^{\infty}t\nu(dt)=\infty.
	\end{equation}
\end{prop}
\begin{proof}
	Notice that under Assumption \ref{ass2}, we have, by monotone convergence,
	\begin{equation*}
		\lim_{\lambda \downarrow 0}\frac{\Phi(\lambda)}{\lambda}=\int_{0}^{+\infty}\tau \nu(d\tau), \quad \mbox{ hence }\quad 
		\lim_{t \to +\infty}\frac{1/t}{\Phi\left(1/t\right)}=\frac{1}{\int_0^{+\infty}\tau \nu(d\tau)}.
	\end{equation*}
	Hence the statement is clear once we notice that as $t \to +\infty$ it holds
	\begin{equation*}
		t^{-1}\E_0\left[\left|X_\Phi(t)\right|^2\right] \asymp \frac{1/t}{\Phi(1/t)}.
	\end{equation*}
\end{proof}
}
It remains to show that the same holds also for the killed process $X^\dagger$. To do this, we need to handle separately the constant boundary case. We first recall the formula for the mean square displacement of the killed Brownian motion.
\begin{prop}\label{prop:kBM}
Let $c>0$ and
\begin{equation*}
	T_c:=\inf\{t>0: \ B(t)>c\}.
\end{equation*}	
Then
\begin{align}\label{eq:msdbM}
	\E_0[|B(t)|^2\mathbf{1}_{T_c>t}]=t\, {\sf erf}\left(\frac{c}{\sqrt{2t}}\right)+\textcolor{black}{2c\sqrt{\frac{t}{2\pi}}e^{-\frac{c^2}{2t}}}-2c^2\left(1-{\sf erf}\left(\frac{c}{\sqrt{2t}}\right)\right),
\end{align}	
where
\begin{equation*}
{\sf erf}(u)=\frac{2}{\sqrt{\pi}}\int_0^u e^{-s^2}\, ds.
\end{equation*}
\end{prop}
The proof easily follows by the reflection principle and is given in Appendix \ref{app:kBM}. It is also worth noticing that
\begin{equation*}
	{\sf erf}(u)=\frac{2}{\sqrt{\pi}}\sum_{n=0}^{+\infty}\frac{(-1)^nu^{2n+1}}{n!(2n+1)}=\frac{2 u}{\sqrt{\pi}}\left(1+\sum_{n=1}^{+\infty}\frac{(-1)^nu^{2n}}{n!(2n+1)}\right)
\end{equation*}
and then
\begin{equation*}
	\sqrt{t}\, {\sf erf}\left(\frac{c}{\sqrt{2t}}\right)=\frac{2 c}{\sqrt{2\pi }}\left(1+\sum_{n=1}^{+\infty}\frac{(-1)^n c^{2n}}{n!(2n+1)(2t)^n}\right).
\end{equation*}
In particular,
\begin{equation*}
	\lim_{t \to +\infty}\sqrt{t}\, {\sf erf}\left(\frac{c}{\sqrt{2t}}\right)=\frac{2 c}{\sqrt{2\pi }}.
\end{equation*}
At the same time
\begin{equation*}
	\lim_{t \to 0}\sqrt{t}\, {\sf erf}\left(\frac{c}{\sqrt{2t}}\right)=0
\end{equation*}
hence there exists a constant ${\sf E}^\star$, depending at most on $c$, such that
\begin{equation*}
	\sqrt{t}\, {\sf erf}\left(\frac{c}{\sqrt{2t}}\right) \le {\sf E}^\star.
\end{equation*}
Furthermore, for $t \ge 1$, there exists a constant ${\sf E}_\star > 0$, depending at most on $c$, such that 
\begin{equation*}
	\sqrt{t}\, {\sf erf}\left(\frac{c}{\sqrt{2t}}\right) \ge {\sf E}_\star.
\end{equation*}
With this in mind, we can prove the following result on the constant boundary case.
\begin{thm}\label{thm:andifconst}
	Let $\varphi(t) \equiv c>0$ and suppose Assumption \eqref{ass2} holds. Then
	\begin{equation*}
		\E_0[|X_\Phi(t)|^2 \mathbf{1}_{\mathcal{T}_c>t}] \underset{\infty}{\asymp} U_{\frac{1}{2}}(t),
	\end{equation*}
	where $U_p$ is defined in \eqref{Up} and, for two non-negative functions $f,g:\R^+ \to \R^+$, the symbol $f \underset{ \infty}{\asymp}g$ means
	\begin{equation*}
		0<\liminf_{t \to \infty} \frac{f(t)}{g(t)} \le \limsup_{t \to \infty} \frac{f(t)}{g(t)} <\infty. 
	\end{equation*}
\end{thm}
\begin{proof}
	By Propositions \ref{prop:refltc} and \ref{prop:kBM}, we clearly have
	\begin{equation*}
		\E_0[|X_\Phi(t)|^2\mathbf{1}_{\mathcal{T}_c>t}]=\E_0\left[L(t)\, {\sf erf}\left(\frac{c}{\sqrt{2L(t)}}\right)+\textcolor{black}{2c\sqrt{\frac{L(t)}{2\pi}}e^{-\frac{c^2}{2L(t)}}}-2c^2\left(1-{\sf erf}\left(\frac{c}{\sqrt{2L(t)}}\right)\right)\right].
	\end{equation*}
	On the one hand, we get
	\begin{equation*}
		\E_0[|X_\Phi(t)|^2\mathbf{1}_{\mathcal{T}_c>t}] \le \E_0\left[L(t)\, {\sf erf}\left(\frac{c}{\sqrt{2L(t)}}\right)+\textcolor{black}{2c\sqrt{\frac{L(t)}{2\pi}}e^{-\frac{c^2}{2L(t)}}}\right] \le \textcolor{black}{\left({\sf E}^\star+2\frac{c}{\sqrt{2\pi}}\right)}  U_{\frac{1}{2}}(t),
	\end{equation*}
	that clearly proves
	\begin{equation*}
		\limsup_{t \to \infty} \frac{\E_0[|X_\Phi(t)|^2\mathbf{1}_{\mathcal{T}_c>t}]}{U_{\frac{1}{2}}(t)} \le 2{\sf E}^\star<\infty.
	\end{equation*}
	On the other hand, notice that
	\begin{align*}
		\E_0[|X_\Phi(t)|^2\mathbf{1}_{\mathcal{T}_c>t}] &\ge 2{\sf E}_\star\E_0\left[\sqrt{L(t)}\mathbf{1}_{\{L(t) \ge 1\}}\right]-2c^2\\
		&=2{\sf E}_\star U_{\frac{1}{2}}(t)-2{\sf E}_\star\E_0\left[\sqrt{L(t)}\mathbf{1}_{\{L(t) \le 1\}}\right]-2c^2\\
		&\ge 2{\sf E}_\star U_{\frac{1}{2}}(t)-2({\sf E}_\star+c^2)
	\end{align*}
	hence
	\begin{equation*}
		\liminf_{t \to \infty} \frac{\E_0[|X_\Phi(t)|^2\mathbf{1}_{\mathcal{T}_c>t}]}{U_{\frac{1}{2}}(t)} \ge 2{\sf E}_\star>0.
	\end{equation*}
\end{proof}
\begin{rmk}
	Notice that differently from the non-killed $X_\Phi(t)$, which exhibits $U_1(t)$ as mean-square displacement, in case $\varphi$ is constant $X^\dagger_\Phi(t)$ admits mean-square displacement whose asymptotic behaviour is given by $U_{\frac{1}{2}}(t)$. This was expected, as it also hold in the case of the Brownian motion. Indeed, while $\E_0[|B(t)|^2]=t$, we have
	\begin{align*}
		\lim_{t \to +\infty}\frac{\E_0[|B(t)|^2\mathbf{1}_{T_c>t}]}{\sqrt{t}}&=\lim_{t \to +\infty}\left(\frac{2c}{\sqrt{2\pi }}\left(1+\sum_{n=1}^{+\infty}\frac{(-1)^n c^{2n}}{n!(2n+1)(2t)^n}\right)-\frac{2c^2}{\sqrt{t}}\left(1-{\sf erf}\left(\frac{c}{\sqrt{2t}}\right)\right)+\textcolor{black}{\frac{2c}{\sqrt{2\pi}}e^{-\frac{c^2}{2t}}}\right)\\
		&=\textcolor{black}{\frac{2c}{\sqrt{2\pi}}}.
	\end{align*}
\end{rmk}
In the general case, the following upper and lower bounds hold.
\begin{thm}
Let $\varphi$ be non-decreasing, continuous and such that $\varphi(0)>0$. Suppose that Assumption \eqref{ass2} holds. Then
\begin{equation}\label{eq:upperbound}
 \E_0[|X_\Phi^\dagger(t)|^2\mathbf{1}_{\{\mathcal{T} > t\}}] \le U_1(t) \le \frac{C}{\Phi(1/t)}
\end{equation}
where $C>1$ is a suitable constant. Furthermore,
\begin{equation}\label{eq:liminfgen}
	\liminf_{t \to \infty}\frac{\E_0[|X_\Phi^\dagger(t)|^2\mathbf{1}_{\{\mathcal{T}>t\}}]}{U_{\frac{1}{2}}(t)}>0.
\end{equation}
Finally, if $\lim_{t \to \infty}\varphi(t)=\ell \in \R^+$, then 
\begin{equation}\label{eq:limsupgen}
	\limsup_{t \to \infty}\frac{\E_0[|X_\Phi^\dagger(t)|^2\mathbf{1}_{\{\mathcal{T}>t\}}]}{U_{\frac{1}{2}}(t)}<\infty.
\end{equation}
\end{thm}
\begin{proof}
	To prove \eqref{eq:upperbound}, by \cite[Proposition III.1]{bertoinb}, it is sufficient to show the first inequality. \btrev{Let us use here the notation $\mathds{E}\coloneqq \mathds{E}_0$.} By Theorem \ref{lem:intrep} we know that
	\begin{align}\label{eq:preuplow}
		\begin{split}
		\E[|X_\Phi^\dagger(t)|^2\mathbf{1}_{\{\mathcal{T} > t\}}]&=\int_{\R}x^2q_\Phi(t,x;0)\, dx\\
		&=\int_{\R}x^2p_\Phi(t,x;0)\, dx-\int_0^t \int_{\R}x^2 p_\Phi(t-w,x;\varphi(w))\, \btrev{dx} \,\mu_\varphi(dw;0)\\
		&=\E_0[|X_\Phi(t)|^2]-\int_0^t\E_{\varphi(w)}[|X_\Phi(t-w)|^2]\mu_\varphi(dw;0)\\
		&=U_1(t)-\int_0^tU_1(t-w)\mu_\varphi(dw;0)-\E[\varphi^2(\mathcal{T})\mathbf{1}_{\{\mathcal{T} \le t\}}].	
		\end{split}
	\end{align}
	It is clear that the latter equality implies \eqref{eq:upperbound}. 
	To show \eqref{eq:liminfgen}, notice that, since $X^\dagger_\Phi$ and $X_\Phi$ coincide up to $\mathcal{T}$, we have
	\begin{equation*}
		\E[|X_\Phi^\dagger(t)|^2\mathbf{1}_{\{\mathcal{T} > t\}}]=\E[|X_\Phi(t)|^2\mathbf{1}_{\{\mathcal{T} > t\}}].
	\end{equation*}
	Now let
	\begin{equation*}
		\mathcal{T}_{\varphi(0)}:=\inf\{t>0, \ X_\Phi(t) \ge \varphi(0)\}
	\end{equation*}
	and notice that $\mathcal{T}_{\varphi(0)} \le \mathcal{T}$ a.s., that implies
	\begin{equation*}
		\mathbf{1}_{\{\mathcal{T}>t\}}=\mathbf{1}_{(t,+\infty)}(\mathcal{T}) \ge  \mathbf{1}_{(t,+\infty)}(\mathcal{T}_{\varphi(0)})=\mathbf{1}_{\mathcal{T}_{\varphi(0)}} \ \mbox{ a.s. }
	\end{equation*}
	since $\mathbf{1}_{(t,+\infty)}(\cdot)$ is non-decreasing. Hence
	\begin{equation*}
		\frac{\E[|X_\Phi^\dagger(t)|^2\mathbf{1}_{\{\mathcal{T} > t\}}]}{U_{\frac{1}{2}}(t)} \ge \frac{\E[|X_\Phi(t)|^2\mathbf{1}_{\{\mathcal{T}_{\varphi(0)} > t\}}]}{U_{\frac{1}{2}}(t)},
	\end{equation*}
	that proves \eqref{eq:liminfgen} by Theorem \ref{thm:andifconst}.
	
	Finally, if $\lim_{t \to +\infty}\varphi(t)=\ell$, we let
	\begin{equation*}
		\mathcal{T}_{\ell}:=\inf\{t>0, \ X_\Phi(t) \ge \ell\}
	\end{equation*}
	and notice that $\mathcal{T}_{\ell} \ge \mathcal{T}$ a.s., that implies, arguing as before,
	\begin{equation*}
		\mathbf{1}_{\{\mathcal{T}>t\}}\le \mathbf{1}_{\mathcal{T}_{\ell}} \ \mbox{ a.s. }
	\end{equation*}
	As a consequence,
	\begin{equation*}
		\frac{\E[|X_\Phi^\dagger(t)|^2\mathbf{1}_{\{\mathcal{T} > t\}}]}{U_{\frac{1}{2}}(t)} \le \frac{\E[|X_\Phi(t)|^2\mathbf{1}_{\{\mathcal{T}_{\ell} > t\}}]}{U_{\frac{1}{2}}(t)},
	\end{equation*}
	that implies \eqref{eq:limsupgen} by Theorem \ref{thm:andifconst}.
\end{proof}
\begin{rmk}
	Notice that the asymptotic behaviour provided by \eqref{eq:liminfgen} and \eqref{eq:limsupgen} guarantees that $X_\Phi^\dagger$ exhibits a subdiffusive behaviour \textcolor{black}{under \eqref{eq:firstmoment}}, since
	\begin{equation*}
		U_{\frac{1}{2}}(t) \le \sqrt{U_1(t)} \le \frac{C}{\sqrt{\Phi\left(\frac{1}{t}\right)}}
	\end{equation*}
	for some constant $C>1$.
\end{rmk}
While this is enough to guarantee that $X_\Phi^\dagger$ is a subdiffusion \textcolor{black}{under \eqref{eq:firstmoment}}, we are not aware of any explicit relation between the behaviour of $U_{\frac{1}{2}}(t)$ and $\Phi(z)$, except that in some specific cases. 




 Following \cite[Chapter $2$, Section $2.1.2$]{bingham}, for a positive function $f:\R^+ \to \R^+$ let us define $\alpha^\ast(f)$ as the infimum of the \btrev{$\alpha \in \mathbb{R}$} such that there exists a constant $D^\ast(\alpha)>0$ for which as $x \to \infty$
\begin{equation*}
	\frac{f(\lambda x)}{f(x)} \le D^\ast(\alpha)(1+o(1))\lambda^\alpha,
\end{equation*}
locally uniformly with respect to $\lambda \ge 1$. On the other hand, we define $\alpha_\ast(f)$ as the supremum of the $\alpha>0$ such that there exists a constant $D_\ast(\alpha)>0$ for which as $x \to \infty$
\begin{equation*}
	\frac{f(\lambda x)}{f(x)} \ge D_\ast(\alpha)(1+o(1))\lambda^\alpha,
\end{equation*}
locally uniformly with respect to $\lambda \ge 1$. The quantities $\alpha^\ast(f)$ and $\alpha_\ast(f)$ are called respectively the upper and lower Matuszewska indices. Notice that if $f(\lambda)=\lambda^\alpha$, then clearly $\alpha^\ast(f)=\alpha_\ast(f)=\alpha$. In general we say that $f$ is $O$-regularly varying if both Matuszewska indices are finite. One can easily check that if $f$ is of (extended) regular variation, then it is also $O$-regularly varying. For further details we refer to \cite[Chapter $2$]{bingham}. Here, we can prove the following statement.
\begin{prop}
Recall the definition of $U_p$ in \eqref{Up}.	If $\Phi(1/\cdot)$ is $O$-regularly varying then
	\begin{equation*}
		U_p(t) \underset{ \infty}{\asymp} \frac{1}{\left(\Phi\left(\frac{1}{t}\right)\right)^p}.
	\end{equation*}
\end{prop}
\begin{proof}
	Notice that, for $p > 0$, $U_{p}$ is non-decreasing and non-negative. Its Laplace-Stieltjes transform is
	\begin{equation*}
		\widetilde{U}_p(z):=\int_0^{+\infty}e^{-z t}U_p(dt)=z\int_0^{+\infty}e^{-zt}U_p(t)dt=\frac{\Gamma(p+1)}{(\Phi(z))^p},
	\end{equation*}
	where we used \eqref{eq:Laptrans}. Now let $-\infty<\alpha<\alpha_\ast(\Phi(1/\cdot))$. Then there exists a constant $D_\ast(\alpha)$ such that, as $z \to \infty$,
	\begin{equation*}
		\frac{\widetilde{U}_p\left(\frac{1}{\lambda z}\right)}{\widetilde{U}_p\left(\frac{1}{z}\right)}=\frac{(\Phi\left(\frac{1}{ z}\right))^p}{(\Phi\left(\frac{1}{\lambda z}\right))^p} \le \frac{1}{D^p_\ast(\alpha)}(1+o(1))\lambda^{-p\alpha}
	\end{equation*}
uniformly with respect to $\lambda \ge 0$. Hence $\alpha^\ast(\widetilde{U}_p(1/\cdot))=-p\alpha_\ast(\Phi(1/\cdot))$. Analogously, one can prove that $\alpha_\ast(\widetilde{U}_p(1/\cdot))=-p\alpha^\ast(\Phi(1/\cdot))$, hence $\widetilde{U}_p(1/\cdot)$ is $O$-regularly varying. Hence, by the de Haan-Stadtm\"{u}ller Tauberian theorem (see \cite[Theorem 2.10.2]{bingham}), it holds
\begin{equation*}
	U_p(t) \underset{\infty}{\asymp} \widetilde{U}_p(1/t)=\frac{\Gamma(p+1)}{\left(\Phi\left(\frac{1}{t}\right)\right)^p}.
\end{equation*}
\end{proof}
As a consequence, we have the following result.
\begin{coro}
	Let $\varphi$ be non-decreasing, continuous and such that $\varphi(0)>0$. Suppose that Assumption \eqref{ass2} holds and that $\Phi(1/\cdot)$ is $O$-regularly varying. Then there exist two constants $t_0,C$ such that
	\begin{equation*}
		\E_0[|X_\Phi^\dagger(t)|^2 \mathbf{1}_{\{\mathcal{T}>t\}}] \ge \frac{C}{\sqrt{\Phi\left(\frac{1}{t}\right)}}, \ t \ge t_0.
	\end{equation*}
	If furthermore $\varphi$ is bounded, then
	 \begin{equation*}
	 	\E_0[|X_\Phi^\dagger(t)|^2 \mathbf{1}_{\{\mathcal{T}>t\}}] \underset{\infty}{\asymp} \frac{1}{\sqrt{\Phi\left(\frac{1}{t}\right)}}.
	 \end{equation*}
\end{coro}
\appendix \section{Proofs of some technical results}
\label{appendix}
\subsection{Proof of Proposition \ref{derfL}}\label{App1True}
Let us first show \eqref{dert} and \eqref{upbound}. Consider $h>0$ and observe that, by \eqref{rapprdensfL},
\begin{align}
	\frac{f_L(s;t+h)-f_L(s;t)}{h}= \, &\int_0^t \overline{\nu}(\tau)\frac{g_\sigma(t+h-\tau;s)-g_\sigma(t-\tau;s)}{h}d\tau\\
	&+\int_t^{t+h} \overline{\nu}(\tau)\frac{g_\sigma(t+h-\tau;s)}{h}d\tau\\
	=: \, &I_1(h)+I_2(h).
\end{align}
{We want to take the limit as $h \to 0^+$.} Concerning $I_2(h)$, we have
\begin{align}
	I_2(h) \, = \, &\int_0^{h} \overline{\nu}(t+h-\tau)\frac{g_\sigma(\tau;s)}{h}d\tau\\
	= \, &\int_0^{h} (\overline{\nu}(t+h-\tau)-\overline{\nu}(t-\tau))\frac{g_\sigma(\tau;s)}{h}d\tau\\
	&+\int_0^{h} \overline{\nu}(t-\tau)\frac{g_\sigma(\tau;s)}{h}d\tau\\
	= \, &I_3(h)+I_4(h).
\end{align}
{It is clear that}
\begin{equation}
	\lim_{h \to 0^+}I_4(h)=\overline{\nu}(t)g_\sigma(0^+;s)=0,
\end{equation}
\btrev{since, under \ref{orcond}, the density $x \mapsto g_\sigma(x,s)$ is $C^\infty$ (see \cite{orey1968continuity}) and thus it must tends to zero as $x \to 0$.}
{To work with} $I_3(h)$, {let us observe that,} {by continuity of $\overline{\nu}$,} for any $\varepsilon>0$ there exists $\delta>0$ such that
\begin{equation}
	|\overline{\nu}(t+h-\tau)-\overline{\nu}(t-\tau)|\le 2\pi \varepsilon, \qquad \forall h \in (0,\delta).
\end{equation}
Moreover, {since}
\begin{equation}
	g_\sigma(t;s)=\frac{1}{2\pi}\int_{-\infty}^{+\infty}e^{-i\xi t-s\Psi(\xi)}d\xi \label{gsigma}
\end{equation}
and, in particular,
\begin{equation}
	|g_\sigma(t;s)|\le \frac{1}{2\pi}\int_{-\infty}^{+\infty}e^{-s\Re(\Psi(\xi))}d\xi,
\end{equation}
{we get}
\begin{equation}
	|I_3(h)|\le 2\varepsilon \int_0^{+\infty}e^{-s\Re(\Psi(\xi))}d\xi, \ h \in (0,\delta),
\end{equation}
{that implies}
\begin{equation}
	\limsup_{h \to 0^+}|I_3(h)|\le 2\varepsilon \int_0^{+\infty}e^{-s\Re(\Psi(\xi))}d\xi.
\end{equation}
Sending $\varepsilon \to 0^+$ we obtain $\lim_{h \to 0^+}I_3(h)=0$.\\
{Finally} let us consider $I_1(h)$, {whose} integrand is dominated {via the inequality}
\begin{equation}\label{domination}
	\overline{\nu}(\tau)\frac{|g_\sigma(t+h-\tau;s)-g_\sigma(t-\tau;s)|}{h}\le \frac{1}{\pi}\overline{\nu}(\tau)\int_0^{+\infty}\xi e^{-s\Re(\Psi(\xi))}d\xi
\end{equation}
that is integrable on $(0,t)$. {Hence} we can take the limit inside the integral sign {to} get
\begin{equation}
	\lim_{h \to 0^+}\frac{f_L(s;t+h)-f_L(s;t)}{h}=\int_0^t \overline{\nu}(\tau)\partial_t g_\sigma(t-\tau;s)d\tau.
\end{equation}
{The argument for $h<0$ is exactly the same.} Furthermore, \eqref{upbound} follows {from \eqref{domination}}. Concerning the continuity of $\partial_t f_L(s;\cdot)$, notice that we have show in general that
\begin{equation*}
	\left|\partial_t g_\sigma(t;s)\right| \le \frac{1}{\pi}\int_0^{+\infty}\xi e^{-s\Re(\btrev{\Psi(\xi)})}\, d\xi<\infty
\end{equation*}
hence $\partial_t f_L(s;\cdot)$ is the convolution product of a $L^1(\R^+)$ function with a bounded continuous function and thus it is continuous.

Now we show \eqref{ders} and \eqref{upbound2}. 
		Start again from \eqref{gsigma} to get
\begin{align}
	\left|\frac{g_\sigma(t;s)-g_\sigma(t;s^\prime)}{s-s^\prime}\right| \, \leq \, \frac{1}{2\pi} \int_{-\infty}^{+\infty} \left| \frac{e^{-s\Psi(\xi)}-e^{-s^\prime \Psi(\xi)}}{s-s^\prime} \right| d\xi.
	\label{this2}
\end{align}
Suppose $s>s^\prime$. {By} \eqref{this2} and the \btrev{since $|e^{iv}-e^{iu}|\leq |v-u|$ for any $u,v \in \mathbb{R}$ with $u\leq v$ and} $x^{-1}(1-e^{-x})\leq 1$ \textcolor{black}{for $x>0$}, we find
\begin{align}
	& \left|\frac{g_\sigma(t;s)-g_\sigma(t;s^\prime)}{s-s^\prime}\right| \notag \\ \leq \,& \frac{1}{2\pi} \int_{-\infty}^{+\infty} \left[ e^{-s\Re \l \Psi(\xi) \r} \left| \frac{e^{-is\Im\l \Psi(\xi) \r}- e^{-is^\prime \Im(\Psi(\xi))}}{\Im \l \Psi(\xi) \r (s-s^\prime)} \right| \left| \Im \l \Psi(\xi) \r \right| \right. \notag \\
	&+ \left. e^{-s^\prime \Re \l \Psi(\xi) \r} \left| \frac{e^{-(s-s^\prime) \Re \l \Psi(\xi) \r}-1}{\Re \l \Psi(\xi) \r (s-s^\prime)} \right| \Re \l \Psi(\xi) \r\right] d\xi \notag \\
	\leq \, & \frac{1}{2\pi} \int_{-\infty}^{+\infty} e^{-s\Re \l \Psi(\xi) \r} \left| \Im \l \Psi(\xi) \r \right|d\xi+ \frac{1}{2\pi} \int_{-\infty}^{+\infty} e^{-s^\prime\Re \l \Psi(\xi) \r} \Re \l \Psi(\xi) \r  d\xi \notag \\
	= \, & \frac{1}{2\pi} \int_{-M_\gamma}^{M_\gamma}  \l e^{-s\Re \l \Psi(\xi) \r}\left| \Im \l \Psi(\xi) \r \right| +e^{-s^\prime\Re \l \Psi(\xi) \r}\Re \l \Psi(\xi) \r\r d\xi \notag \\& + \frac{1}{2\pi} \int_{|\xi|>M_\gamma} e^{-s\Re \l \Psi(\xi) \r} \left| \Im \l \Psi(\xi) \r \right| d\xi \notag \\
	&+ \frac{1}{2\pi} \int_{|\xi|>M_\gamma} e^{-s^\prime\Re \l \Psi(\xi) \r} \Re \l \Psi(\xi) \r d\xi \notag \\
	=: \, & I_1+I_2+I_3
\end{align}
where $M_\gamma$ is defined in Lemma \ref{lemmaintzero}.
Note now that
\begin{align}
	\Im \l \Psi(\xi) \r \, = \, - \int_0^{+\infty} \sin(\xi \tau) \nu(d\tau)
\end{align}
is odd and therefore $|\Im \l \Psi(\xi) \r |$ is even. For $\xi >M_\gamma$ and $\tau<1$, it is clear that $\left| \sin (\xi \tau) \right| \leq \xi \tau$ {and then}
\begin{align}
	\left| \Im \l \Psi(\xi)\r \right| \, \leq \, \xi \int_0^1  \tau \nu(d\tau) + \overline{\nu}(1).
	\label{this3}
\end{align}
It follows by \eqref{this3} and \eqref{427} that
\begin{align}
	I_2 \leq \frac{1}{\pi} \int_{\xi>M_\gamma} \l \xi \int_0^1 \tau\nu(d\tau)+\overline{\nu}(1) \r e^{-s\textcolor{black}{C_\gamma}\xi^\alpha} d\xi < \infty,
\end{align}
where $\alpha = 2-\gamma$.
Observe now that $\Re \l \Psi(\xi) \r$ is even and that $1-\cos (\xi \tau) \leq \frac{\xi^2 \tau^2}{2}$, {thus}
\begin{align}
	\Re \l \Psi(\xi) \r \, = \, \int_0^{+\infty} \l 1-\cos (\xi \tau) \r \nu(d\tau) \, \leq \, \frac{\xi^2}{2} \int_0^1 \tau^2 \nu(d\tau) + 2 \overline{\nu}(1).
	\label{this4}
\end{align}
It follows by \eqref{this4} and \eqref{427} that
\begin{align}
	I_3 \leq \frac{1}{\pi} \int_{\xi>M_\gamma} \l \frac{\xi^2}{2} \int_0^1 \tau^2 \nu(d\tau) + 2 \overline{\nu}(1) \r e^{-s^\prime \btrev{C_\gamma}\xi^\alpha} d\xi < \infty.
\end{align}
It also clear that $I_1$ is finite and bounded as follows:
\begin{equation*}
	I_1 \le \frac{\sqrt{2}}{\pi}\int_{0}^{M_\gamma}|\Psi(\xi)|\left(e^{-s\Re(\Psi(\xi))}+e^{-s^\prime \Re(\Psi(\xi))}\right)\, d\xi.
\end{equation*}
Repeating the same argument for $s^\prime > s$ and combining the results we obtain, {for fixed $s>0$} and $s^\prime \in (s/2,3s/2)$,
\begin{align}\label{preupb2}
	\left| \frac{g(t,s)-g(t,s^\prime)}{s-s^\prime} \right| &\leq \frac{2\sqrt{2}}{\pi}\int_{0}^{M_\gamma}|\Psi(\xi)|e^{-\frac{s}{2}\Re(\Psi(\xi))}\, d\xi \nonumber\\
	&+\int_{M_\gamma}^{+\infty}\left(3\overline{\nu}(1)+\xi \int_0^1 \tau \nu(d\tau)+\frac{\xi^2}{2}\int_0^1 \tau^2\nu(d\tau)\right)e^{-\frac{s}{2}\btrev{C_\gamma}\xi^{2-\gamma}}\, d\xi 
\end{align}
Since $\overline{\nu} \in L^1_{\rm loc}(\R^+_0)$, then \eqref{ders} follows by dominated convergence and \eqref{upbound2} follows by \eqref{preupb2}. The latter also shows that
\begin{align*}
	\partial_s g(t,s) = -\frac{1}{2\pi}\int_{-\infty}^{+\infty}\Psi(\xi)e^{-it\xi-s\Psi(\xi)}\, d\xi
\end{align*}
that, for fixed $s>0$, is a continuous function of the variable $t \in \R^+$ by a simple application of the dominated convergence theorem. Furthermore,
\begin{align*}
	\left| \partial_s g(t,s)\right| &\leq \frac{\sqrt{2}}{\pi}\int_{0}^{M_\gamma}|\Psi(\xi)|e^{-s\Re(\Psi(\xi))}\, d\xi \nonumber\\
	&+\int_{M_\gamma}^{+\infty}\left(3\overline{\nu}(1)+\xi \int_0^1 \tau \nu(d\tau)+\frac{\xi^2}{2}\int_0^1 \tau^2\nu(d\tau)\right)e^{-s\btrev{C_\gamma}\xi^{2-\gamma}}\, d\xi. 
\end{align*}
hence $\partial_s f_L(s;\cdot)$ is the convolution product of a $L^1(\R^+)$ function with a bounded continuous function and thus it is continuous.

Now we prove the claimed behaviour at infinity. By \eqref{rapprdensfL} and \eqref{gsigma}, {it holds}
\begin{align}
	\sup_{t \in [a,b]}s f_L(s, t) \, \leq \, &\frac{1}{2\pi}\left(\int_0^b \overline{\nu}  (\tau)d\tau\right)  \, s\int_{-\infty}^{+\infty}  e^{-s \Re \l \Psi(\xi)\r} \, d\xi \notag \\
	\leq \, &\frac{1}{\pi}\left(\int_0^b \overline{\nu}  (\tau)d\tau\right)\left[ \int_0^{M_\gamma} s e^{-s\Re \l \Psi(\xi) \r} d\xi + s\int_{M_\gamma}^{+\infty} e^{-s \btrev{C_\gamma}\xi^{2-\gamma}} d\xi \right]\notag \\
	\leq \, & \frac{1}{\pi} \left(\int_0^b \overline{\nu}  (\tau)d\tau\right)\left[ \int_0^{M_\gamma} s e^{-s\Re \l \Psi(\xi) \r} d\xi + \frac{\btrev{C_\gamma^{-\frac{1}{2-\gamma}}}}{2-\gamma} s^{1-\frac{1}{2-\gamma}} \Gamma\left(\frac{1}{2-\gamma}\right)\right].
	\label{chep}
\end{align}
It is clear that \eqref{chep} tends to zero {as $s \to \infty$} by using dominated convergence on the first integral and the fact that $\gamma>1$. The last assertion follows by similar computation. Indeed, in the same spirit of \eqref{chep}, we obtain
\begin{align}
	\sup_{t \in [a,b]} s^2 | \partial_sf_L(s, t)| \, \leq \, \frac{1}{\pi} \left(\int_0^b \overline{\nu}  (\tau)d\tau\right) \left[ \int_0^{M_\gamma} | \Psi(\xi)|s^2 e^{-s\Re \l \Psi(\xi) \r} d\xi + s^2 \int_{M_\gamma}^\infty | \Psi(\xi)| e^{-s\btrev{C_\gamma}\xi^{2-\gamma}} d\xi\right].
	\label{535}
\end{align}
Using the estimates on $\Re (\Psi(\xi))$ and $\Im (\Psi(\xi))$ determined above we obtain
\begin{align}
	|\Psi(\xi)| \leq \left(\frac{3}{M_\gamma^2}\overline{\nu}(1)+\frac{1}{M_\gamma}\int_0^1 t\nu(dt)+\frac{1}{2}\int_0^1t^2\nu(dt)\right)\xi^2
\end{align}
for $\xi>M_\gamma$. It follows that
\begin{align}
	\begin{split}
	\sup_{t \in [a,b]} &s^2 |\partial_sf_L(s, t)| \, \leq \, 2 \left(\int_0^b \overline{\nu}  (\tau)d\tau\right) \left[ \int_0^{M_\gamma} | \Psi(\xi)|s^2 e^{-s\Re \l \Psi(\xi) \r} d\xi \right.\\
	&\left.+ \frac{1}{2-\gamma} \left(\frac{3}{M_\gamma^2}\overline{\nu}(1)+\frac{1}{M_\gamma}\int_0^1 \tau\nu(d\tau)+\frac{\btrev{C_\gamma^{-\frac{3}{2-\gamma}}}}{2}\int_0^1\tau^2\nu(d\tau)\right)  s^{2-\frac{3}{2-\gamma}}  \Gamma \left(\frac{3}{2-\gamma}\right) \right].
	\label{chep2}	
	\end{split}
\end{align}
It is clear that \eqref{chep2} tends to zero as $s \to \infty$ by using dominated convergence on the first integral and, again, the fact that $\gamma>1$. 

Finally, notice that if $s \in [s_0,s_1]$ for some $0<s_0<s_1$, then by \eqref{upbound} and \eqref{upbound2} we have
\begin{align}
	\notag \sup_{s \in [s_0,s_1]}\left|\partial_t f_L(s,t)\right| &\le \frac{I_\Phi(t)}{\pi}\int_0^{+\infty}\xi e^{-s_0\Re(\Psi(\xi))}\, d\xi\\
	\notag \sup_{s \in [s_0,s_1]}\left|\partial_s f_L(s,t)\right| &\le \frac{I_\Phi(t)}{\pi}\left(\sqrt{2}\int_0^{M_\gamma}|\Psi(\xi)|e^{-s_0\Re(\Psi(\xi))}\, d\xi\right.\\
	&\left.+\int_{M_\gamma}^{+\infty}\left(3\overline{\nu}(1)+\xi\int_0^{1}\tau \nu(d\tau)+\frac{\xi^2}{2}\int_0^1 \tau^2 \nu(d\tau)\right) e^{-s_0\btrev{C_\gamma}\xi^{2-\gamma}}\, d\xi\right),
\end{align}
where the right-hand side converges to $0$ as $t \downarrow 0$. This shows \eqref{eq:unifconvparfL}
\qed
\subsection{Proof of Proposition \ref{prop:keyhole}}\label{app:keyhole}
Let us preliminarily recall that $\Phi$ is holomorphic in $\mathbb{C}(\theta)$ and real on the real axis, hence it is Hermitian. As a consequence, we recall that $\Re(\Phi(a+ib))=\Re(\Phi(a-ib))$. Now let us show that 
\begin{equation*}
	g_\sigma(s,t)=\frac{1}{2\pi }\int_{-\infty}^{+\infty}e^{s(a+ib)-t\Phi(a+ib)}\, db,
\end{equation*}
for any $a>0$, where the integral is absolutely convergent. By \cite[Theorem 4.2.21.a and Theorem 4.1.2]{abhn}, it is sufficient to show that the integral on the right-hand side is absolutely convergent. To do this, fix $a>0$ and notice that
\begin{equation*}
	\int_{-\infty}^{+\infty}\left|e^{s(a+ib)-t\Phi(a+ib)}\right|\, db=e^{sa}\int_{-\infty}^{+\infty}e^{-t\Re(\Phi(a+ib))}\, db.
\end{equation*}
Now let $b_0$ be such that $\Re(\Phi(a+ib)) \ge \frac{2}{t}\log(b)$ for any $b$ such that $|b|>b_0$. Thus we can split the integral as
\begin{align*}
	\int_{-\infty}^{+\infty}\left|e^{s(a+ib)-t\Phi(a+ib)}\right|\, db&=2e^{sa}\int_{0}^{b_0}e^{-t\Re(\Phi(a+ib))}\, db+2e^{sa}\int_{b_0}^{+\infty}e^{-t\Re(\Phi(a+ib))}\, db\\
	&\le 2b_0e^{sa}+\frac{2e^{sa}}{b_0}<\infty.
\end{align*}
This proves the desired absolute integrability. Now consider the sequence
\begin{equation*}
	r_n=\sqrt{\frac{n^2\pi^2}{s^2}-a^2},
\end{equation*}
where $n$ is sufficiently big to have $\frac{n^2\pi^2}{s^2}>a^2$. We have
\begin{equation*}
	\int_{-r_n}^{r_n}e^{s(a+ib)}\, db=\frac{e^{sa}}{si}(e^{isr_k}-e^{-isr_k})=-\frac{2e^{sa}}{s}\sin(sr_n).
\end{equation*}
Now we use the Lipschitz property of the sine function to observe that
\begin{equation*}
	|\sin(sr_n)|=|\sin(sr_n)-\sin(n\pi)| \le |sr_n-n\pi|.
\end{equation*}
Now notice that
\begin{equation*}
	sr_n-n\pi=\sqrt{n^2\pi^2-s^2a^2}-n\pi=n\pi\left(\sqrt{1-\frac{s^2a^2}{n^2\pi^2}}-1\right) \to 0
\end{equation*}
hence taking the limit we have
\begin{equation*}
	\lim_{n \to +\infty}\int_{-r_n}^{r_n}e^{s(a+ib)}\, db=0.
\end{equation*}
Hence we can write
\begin{equation*}
	g_\sigma(s,t)=\frac{1}{2\pi}\lim_{n \to +\infty}\int_{-r_n}^{+r_n}e^{s(a+ib)}\left(e^{-t\Phi(a+ib)}-1\right)\, db.
\end{equation*}
Now we let ${\sf A}_n=a+ir_n$, ${\sf B}_n=a-ir_n$ and we construct an auxiliary contour. We set $R_n=\sqrt{a^2+r_n^2}=\frac{n\pi}{s}$, $\varepsilon<R_n$ and we define ${\sf C}_n=R_ne^{i\theta}$, ${\sf D}(\varepsilon)=\varepsilon e^{i\theta}$, ${\sf E}(\varepsilon)=\varepsilon e^{-i\theta}$, and ${\sf F}_n=R_ne^{-i\theta}$. Then we denote by $\Gamma_n^+$ the directed counter-clockwise circle arc, centred in the origin, connecting ${\sf A}_n$ to ${\sf C}_n$, $\ell^+_n$ the directed segment connecting ${\sf C}_n$ to ${\sf D}(\varepsilon)$, $\gamma_\varepsilon$ the directed clockwise circle arc, centred in the origin, connecting ${\sf D}(\varepsilon)$ to ${\sf E}(\varepsilon)$, $\ell^-_n$ the directed segment connecting ${\sf E}(\varepsilon)$ to ${\sf F}_n$, $\Gamma_n^-$ the directed counter-clockwise circle arc, centred in the origin, connecting ${\sf F}_n$ to ${\sf B}_n$ and finally the directed segment $\ell_n$ connecting ${\sf B}_n$ to ${\sf A}_n$. For a figure of such a contour, check \cite[Figure 1]{tlms2024}. Let $\mathcal{C}$ be the full closed contour and consider the function $F(z)=e^{sz}\left(e^{-t\Phi(z)}-1\right)$. Then $F$ is holomorphic in the domain enclosed by $\mathcal{C}$ and thus
\begin{equation*}
	\int_{\mathcal{C}}F(z)=0
\end{equation*}
or, equivalently,
\begin{align}\label{eq:contoursplit}
	\begin{split}
		\int_{\ell}F(z)\, dz&=-\left(\int_{\Gamma_n^+}F(z)\, dz+\int_{\ell^+_n}F(z)\, dz+\int_{\gamma_\varepsilon}F(z)\, dz+\int_{\ell^-_n}F(z)\, dz+\int_{\Gamma_n^-}F(z)\, dz\right)\\
		&=-\left(I_1^+(r_n)+I_2^+(r_n,\varepsilon)+I_3(\varepsilon)+I_2^-(r_n,\varepsilon)+I_1^-(r_n)\right).	
	\end{split}
\end{align}
Let us also preliminarily notice that
\begin{equation}\label{eq:exponentialcontrol}
	\left|e^{-t\Phi(z)}-1\right|=\left|t\Phi(z)\int_{0}^1e^{-\tau t\Phi(z)}\, d\tau\right| \le t\left|\Phi(z)\right|. 
\end{equation}
Concerning $I_2^\pm$, we have
\begin{align*}
	I_2(r_n,\varepsilon)=-\int_{\varepsilon}^{r_n}e^{i\theta+s\xi e^{i\theta}}\left(e^{-t\Phi(\xi e^{i\theta})}-1\right)\,d\xi, && I_2^-(r_n,\varepsilon)=\int_{\varepsilon}^{r_n}e^{-i\theta+s\xi e^{-i\theta}}\left(e^{-t\Phi(\xi e^{-i\theta})}-1\right)\,d\xi,
\end{align*}
hence, once we recall that $\Phi$ is  Hermitian, it is clear that
\begin{equation*}
	I_2^+(r_n,\varepsilon)+I_2^-(r_n,\varepsilon)=-2i\int_{\varepsilon}^{r_n}\Im\left(e^{i\theta+s\xi e^{i\theta}}\left(e^{-t\Phi(\xi e^{i\theta})}-1\right)\right)\, d\xi.
\end{equation*}
To show that the latter is absolutely convergent even for $r_n \to +\infty$ and $\varepsilon \to 0$, we notice that
\begin{equation*}
	\int_{0}^{+\infty}e^{s\xi \cos(\theta)}\left|e^{-t\Phi(\xi e^{i\theta})}-1\right|\, d\xi \le t\int_{0}^{+\infty}e^{s\tau \cos(\theta)}\left|\Phi(\xi e^{i\theta})\right|\, d\xi.
\end{equation*}
Since $\lim_{\xi \to +\infty}\frac{\left|\Phi(\xi e^{i\theta})\right|}{\xi}=0$, there exists $M>0$ such that for $\tau>M$ it holds $\left|\Phi(\xi e^{i\theta})\right| \le \xi$, hence
\begin{align*}
	\int_{0}^{+\infty}e^{s\xi \cos(\theta)}&\left|e^{-t\Phi(\xi e^{i\theta})}-1\right| \le t\int_{0}^{M}e^{s\xi \cos(\theta)}\left|\Phi(\xi e^{i\theta})\right|\, d\xi+t\int_{M}^{+\infty}\xi e^{s\xi \cos(\theta)}\, d\xi\\
	& \le \frac{MS_Mt}{s|\cos(\theta)|}\left(1-e^{sM \cos(\theta)}\right)+t\frac{e^{sM\cos(\theta)}(1-Ms\cos(\theta))}{s^2\cos^2(\theta)}<\infty,
\end{align*}
where, setting $\mathcal{S}_k:=\{z \in \overline{\mathbb{C}(\theta)}: \ |z| \le k\}$ for any positive constant $k>0$, it holds
\begin{equation*}
	S_k=\max_{z \in \mathcal{S}_k}|\Phi(z)|<\infty.
\end{equation*}
Let us now handle $I_3(\varepsilon)$. There we have, for $\varepsilon<M$,
\begin{equation*}
	\left|I_3(\varepsilon)\right| \le t\varepsilon\int_{-\theta}^{\theta}e^{s\varepsilon \cos(\xi)}\left|\Phi(\varepsilon e^{i\xi})\right|\, d\xi \le t\varepsilon S_\varepsilon e^{s\varepsilon}
\end{equation*}
where the right-hand side converges to $0$ as $\varepsilon \to 0$, since $\Phi(0)=0$. Hence we can take the limit as $\varepsilon \downarrow 0$ in \eqref{eq:contoursplit}, getting
\begin{align}\label{eq:contoursplit2}
	\begin{split}
		\int_{\ell}F(z)\, dz=-\left(I_1^+(r_n)+I_2^+(r_n)+I_2^-(r_n)+I_1^-(r_n)\right),
	\end{split}
\end{align}
where $I_2^\pm(r_n):=I_2^\pm(r_n,\varepsilon)$.
Concerning $I_1^\pm(r_n)$, notice that
\begin{equation*}
	I_1^\pm(r_n)=\int_{\Gamma_n^{\pm}}e^{sz-t \Phi(z)}\, dz+\int_{\Gamma_n^{\pm}}e^{sz}\, dz=I^{\pm}_4(r_n)+I^{\pm}_5(r_n).
\end{equation*}
For $I^{\pm}_4(r_n)$, the same argument as in the proof of \cite[Proposition 5.1, Eqs. (5.9) and (5.12)]{tlms2024} shows that
\begin{equation*}
	\lim_{n \to +\infty}I^{\pm}_4(r_n)=0.
\end{equation*}
Now we want to handle $I^{+}_5(r_n)+I^-_5(r_n)$. To do this, let $\widetilde{\theta}(r_n)=\arctan\left(\frac{r_n}{a}\right)$ and notice that
\begin{align*}
	I_5^+(r_n)=iR_n\int_{\widetilde{\theta}(r_n)}^{\theta}e^{i\xi+sRe^{i\xi}}\, d\xi &&
	I_5^-(r_n)=iR_n\int_{\widetilde{\theta}(r_n)}^{\theta}e^{-i\xi+sRe^{-i\xi}}\, d\xi
\end{align*}
hence
\begin{equation*}
	I_5^+(r_n)+I_5^-(r_n)=2iR_n\int_{\widetilde{\theta}(r_n)}^{\theta}\Re \left(e^{i\xi+sR_ne^{i\xi}}\right)\, d\tau=2\Im\left(iR_n \int_{\widetilde{\theta}(r_n)}^{\theta}e^{i\xi+sR_ne^{i\xi}}\, d\tau\right)=2\Im(I_5^+(r_n)). 
\end{equation*}
Evaluating $I_5^+(r_n)$ explicitly we have
\begin{align*}
	I_5^+(r_n)=\frac{e^{sR_ne^{i\theta}}-e^{sR_ne^{i\widetilde{\theta}(r_n)}}}{s}
\end{align*}
whose imaginary part is given by
\begin{equation*}
	\Im(I_5^+(r_n))=\frac{e^{n\pi\cos(\theta)}\sin(n\pi\sin(\theta))}{s}-\frac{e^{n\pi\cos(\widetilde{\theta}(r_n))}\sin(n\pi\sin(\widetilde{\theta}(r_n)))}{s}.
\end{equation*}
Since $\cos(\theta)<0$, it is clear that the first summand converges towards $0$. Let us handle the second summand. For the exponential, let us first write
\begin{equation*}
	\cos(\widetilde{\theta}(r_n))=\sin\left(\frac{\pi}{2}-\widetilde{\theta}(r_n)\right),
\end{equation*}
where it is clear that $\widetilde{\theta}(r_n) \to \frac{\pi}{2}$. In particular, it is immediate to see that $	\lim_{n \to +\infty}n\pi\cos(\widetilde{\theta}(r_n))=sa$. Concerning the sine function, we use the identity
\begin{equation*}
	\sin(\widetilde{\theta}(r_n))=\frac{r_n}{\sqrt{1+r_n^2}}=\sqrt{\frac{\frac{n^2\pi^2}{s^2}-a^2}{1+\frac{n^2\pi^2}{s^2}-a^2}}=\sqrt{1-\frac{s^2}{s^2+n^2\pi^2-a^2}}.
\end{equation*}
Now we notice that
\begin{equation*}
	n\pi\sqrt{1-\frac{s^2}{s^2+n^2\pi^2-a^2}}=n\pi+n\pi\left(\sqrt{1-\frac{s^2}{s^2+n^2\pi^2-a^2}}-1\right).
\end{equation*}
Again, by the Lipschitz property of the sine function, we have
\begin{equation*}
	\left|\sin(n\pi \sin(\widetilde{\theta}(r_n)))\right|=\left|\sin(n\pi \sin(\widetilde{\theta}(r_n)))-\sin(n\pi)\right| \le n\pi\left(\sqrt{1-\frac{s^2}{s^2+n^2\pi^2-a^2}}-1\right) \to 0.
\end{equation*}
Hence, taking the limit as $n \to +\infty$, we have
\begin{equation*}
	\lim_{n \to +\infty}(I_5^+(r_n)+I_5^-(r_n))=0 \qquad \mbox{ and then } \qquad \lim_{n \to +\infty}(I_1^+(r_n)+I_1^-(r_n))=0.
\end{equation*}
Now we go back to \eqref{eq:contoursplit2} and we write
\begin{equation*}
	\frac{1}{2 \pi}\int_{-r_n}^{r_n}e^{h(a+ib)}\left(e^{-t\Phi(a+ib)}-1\right)\, db=\frac{1}{\pi}\int_{0}^{r_n}\Im\left(e^{i\theta+s\tau e^{i\theta}}\left(e^{-t\Phi(\tau e^{i\theta})}-1\right)\right)\, d\tau-\frac{I_1^+(r_n)+I_1^-(r_n)}{2\pi}.
\end{equation*}
Taking then the limit as $n \to +\infty$ we get \eqref{eq:gstbound}. To prove \eqref{eq:controlgst}, we consider $M>0$ big enough to have $\left|\Phi(\xi e^{i \theta})\right| \le 2\xi^\alpha$ for all $\xi>M$. Then
\begin{align*}
	g_\sigma(s,t) &\le \frac{t}{\pi}\left(\int_0^Me^{s\xi \cos(\theta)}\left|\Phi(\xi e^{i \theta})\right|\, d\xi+\int_M^{+\infty}\xi^\alpha e^{s\xi \cos(\theta)}\, d\xi\right)\le \frac{t}{\pi}\left(MS_M+\frac{\Gamma(\alpha+1)}{s^{\alpha+1}|\cos(\theta)|^{\alpha+1}}\right),
\end{align*}
ending the proof. 
\qed
\subsection{Proof of Theorem \ref{thm:weakmax}}\label{App2new}
To prove Theorem \ref{thm:weakmax}, we first need to show the following auxiliary result.

\begin{lem}
	\label{lemma34}
	\textcolor{black}{Assume \eqref{ass2} holds}. Let $f:[0,T] \to \R$ be a continuous function and $t_0 \in (0,T]$ such that 
	\begin{equation*}
	f(t_0)>f(t), \ \forall t \in (0,t_0).
	\end{equation*}
	Assume that $\partial^\Phi_t f(t)$ is well-defined, continuous on $(0,t_0]$ and belongs to $L^1[0,t_0]$. Then $\partial_t^\Phi f(t_0)\ge 0$.
\end{lem}
\begin{proof}
	Let us first prove the statement under the additional assumption that $f \in {\rm AC}[0,t_0]\cap C^1(0,t_0]$. Consider the function $g(\tau)=f(t_0)-f(\tau)$ for $\tau \in [0,t_0]$, {so that}, being $t_0$ the maximum point of $f$, it holds $g(\tau)\ge 0$ for any $\tau \in [0,t_0]$ and $\partial_t^\Phi g(\tau)=-\partial_t^\Phi f(\tau)$ for any $\tau \in (0,t_0]$. Since $t_0>0$, $g$ cannot be constant. Furthermore, since $f \in C^1(0,t_0]$, $g \in C^1[\varepsilon,t_0]$ for any $\varepsilon \in (0,t_0)$, and then, in particular, $g$ is Lipschitz on $[\varepsilon,t_0]$, i.e.,
	\begin{equation}
		\forall \varepsilon \in (0,T) \ \exists C_\varepsilon>0: \ g(\tau)\le C_\varepsilon |t_0-\tau| \ \forall \tau \in [\varepsilon,\textcolor{black}{t_0}].
		\label{duepuntosette}
	\end{equation}
	Note that {by hypotheses} $g$ is absolutely continuous on {$[0,t_0]$}, thus it follows that we can use \eqref{derdentro} to say that 
	\begin{equation}
		\partial_t^\Phi g(t_0)=\int_0^{t_0}\overline{\nu}(t_0-\tau)\partial_\tau g(\tau)d\tau,
	\end{equation}

	Let us split the integral
	\begin{align}
		\partial_t^\Phi g(t_0)&=\int_0^{\varepsilon}\overline{\nu}(t_0-\tau)\partial_\tau g (\tau)d\tau+\int_\varepsilon^{t_0}\overline{\nu}(t_0-\tau)\partial_\tau g(\tau)d\tau\\
		&=I_1(\varepsilon)+I_2(\varepsilon),
	\end{align}
	where $\varepsilon \in (0,t_0)$.
	We can clearly use dominated convergence theorem to obtain
	\begin{equation}
		I_2(\varepsilon)=\lim_{a \to 0^+}\int_{a}^{t_0-\varepsilon}\overline{\nu}(\tau)\partial_\tau g(\textcolor{black}{t_0}-\tau)d\tau.
	\end{equation}
	Now let us recall that the function $\overline{\nu}$ is non-increasing and finite in $[a,t_0-\varepsilon]$, hence it is of bounded variation with distributional derivative given by $-\nu$. Thus we can apply integration by parts (see \cite[Theorem $3.3.1$]{hille}) to obtain
	\begin{equation}
		I_2(\varepsilon)=\lim_{a \to 0}\left[\overline{\nu}(a)g(t_0-a)-\overline{\nu}(t_0-\varepsilon)g(\varepsilon)-\int_{a}^{t_0-\varepsilon}g(t_0-z)\nu(dz)\right].
	\end{equation}
	Combining inequality \eqref{duepuntosette} with the fact that $t \overline{\nu}(t)\to 0$ as $t \to 0$ (see, e.g., \cite[Eq. (3.7)]{librobern}) and also using monotone convergence theorem, we get
	\begin{equation}
		I_2(\varepsilon)=-\overline{\nu}(t_0-\varepsilon)g(\varepsilon)-\int_0^{t_0-\varepsilon}g(\btrev{t_0}-\tau)\nu(d\tau).
	\end{equation}
	Since $g$ is not constant on $(0,t_0)$, we know there exists $\varepsilon_0 \in (0,t_0)$ such that $g(\varepsilon_0)>0$. Furthermore, we have for $\varepsilon \in (0,\varepsilon_0)$, by continuity of $g$ and the fact that $-\overline{\nu}(t_0-\varepsilon)g(\varepsilon)\le 0$,
	\begin{equation}
		I_2(\varepsilon)\le -\int_{0}^{t_0-\varepsilon_0}g(t_0-\tau)\nu(d\tau)=:-C<0,
	\end{equation}
	where $C$ is finite since, by \eqref{duepuntosette},
	\begin{align*}
		\int_0^{t_0-\varepsilon_0} g(t_0-\tau) \nu(d\tau) \, \leq \,C_{\varepsilon_0} \int_0^{t_0-\varepsilon_0} \tau \, \nu(d\tau) < +\infty.
	\end{align*}
	Now let us consider $I_1(\varepsilon)$. {By absolute continuity of $g$ and monotonicity of $\overline{\nu}$} we clearly have
	\begin{equation}
		I_1(\varepsilon) \le \overline{\nu}(t_0-\varepsilon_0)\int_0^\varepsilon \left| \partial_\tau g(\tau) \right| d\tau.
	\end{equation}
	By absolute continuity of the Lebesgue integral we have that there exists  $\varepsilon \in (0,\varepsilon_0)$ such that $I_1(\varepsilon)\le C/2$. Let us fix this $\varepsilon$ to achieve
	\begin{equation}
		-\partial_t^\Phi f(t_0)=\partial_t^\Phi g(t_0)=I_1(\varepsilon)+I_2(\varepsilon)\le -\frac{C}{2}<0
	\end{equation}
	concluding the proof in the case $f \in {\rm AC}[0,t_0] \cap C^1(0,t_0]$.
	
	Now let us consider a generic $f:[0,T] \to \R$ satisfying the assumptions in the statement. However, we assume further that $t_0<T$ and that $\partial_t^\Phi f(t)$ is well-defined and continuous in $(t_0-2\delta,t_0+2\delta)$ for some $\delta>0$. We denote by $\ast$ the convolution product. Let $\{\varphi_n\}_{n \in \N}$ be a family of Friedrich's mollifiers and consider the sequence $f_n=\varphi_n \ast f$. Clearly, $f_n \to f$ uniformly in $[0,T]$ (see \cite[Lemma 1.3.3]{abhn}). In particular, consider any sequence $t_0^n \in \argmax_{t \in [0,t_0]}f_n(t)$. Since $t_0^n \in [0,t_0]$ for all $n \in \N$, we can extract a (non-relabelled) subsequence such that $t_0^n \to t_\star \in [0,t_0]$. Furthermore, $f_n(t_0^n) \ge f_n(t_0)$ and then, taking the limit as $n \to \infty$, we get $f(t_\star) \ge f(t_0)$, that guarantees that $t_0=t_\star$. Next, notice that
	\begin{equation*}
		\varphi_n \ast \mathcal{I}^\Phi f=\mathcal{I}^\Phi f_n,
	\end{equation*}
	by associativity of the convolution product. By \cite[Proposition 1.3.6]{abhn} we get
	\begin{equation*}
		\varphi_n \ast \partial_t^\Phi f=\partial_t^\Phi f_n,
	\end{equation*}
	and then, again, $\partial_t^\Phi f_n \to \partial_t^\Phi f$ uniformly in $[t_0-\delta,t_0+\delta]$. In particular, since $f_n \in C^\infty_c(\R)$, we already know that $\partial_t^\Phi f_n(t_0^n) \ge 0$. Taking the limit and using the uniform convergence, we get $\partial_t^\Phi f(t_0) \ge 0$.
	
	Now we are ready to prove the full statement. Let $f$ as requested and consider $\delta_n=\frac{f(t_0)-f(0)}{n}$ for $n \ge 2$. Clearly, $f(0)< f(t_0)-\delta_n < f(t_0)$, hence we can set $t_n=\min\{t \ge 0: \ f(t)=f(t_0)-\delta_n\}$ and it holds $t_n \in (0,t_0)$. This guarantees that $\partial_t^\Phi f$ is continuous in an open neighbourhood of $t_n$. Furthermore $t_n=\min \argmax_{t \in [0,t_n]}f(t)$. Indeed, if $t_n^\prime=\min \argmax_{t \in [0,t_n]}f(t)<t_n$, then $f(t_n^\prime) \ge f(t_n)=f(t_0)-\delta_n>f(0)$ and then, by the intermediate value theorem, we know that there exists $t_n^{\prime \prime} \in [0,t_n^\prime]\subset [0,t_n)$ such that $f(t_n^{\prime \prime})=f(t_0)-\delta_n$ that is a contradiction with the definition of $t_n$. Furthermore, since $t_n \in [0,t_0]$, there exists a (non-relabelled) subsequence such that $t_n \to t^\star$. By continuity, we must have $f(t^\star)=f(t_0)$, hence, by definition of $t_0$, $t^\star=t_0$. However, since $\partial_t^\Phi f$ is continuous in $(0,t_n+\varepsilon)$ for some $\varepsilon>0$, we have, by the previous argument, $\partial_t^\Phi f(t_n) \ge 0$. Taking the limit we get $\partial_t^\Phi f(t_0) \ge 0$, which ends the proof.
\end{proof}
Now we are ready to prove the weak maximum principle.

\begin{proof}[Proof of Theorem \ref{thm:weakmax}]
We argue by contradiction. Set $M_1=\max_{(t,x)\in \partial_p E}u(t,x)$ if $E$ is bounded and $M_1=\max\left\{\max_{t \in [0,t]}u_\infty(t), \max_{(t,x)\in \partial_p E}u(t,x)\right\}$ if $E$ is unbounded. Suppose there exists $(t_0,x_0) \in E$ such that $u(t_0,x_0)>M_1$, define $\varepsilon=u(t_0,x_0)-M_1>0$ and consider the auxiliary function
\begin{equation}
	w(t,x)=u(t,x)+\frac{\varepsilon}{2}\frac{T-t}{T} \ \forall (t,x)\in \overline{E}
\end{equation}
Notice that for $x \in E_2$ it holds
\begin{equation} 
	w(t(x),x)=u(t(x),x)+\frac{\varepsilon}{2}\frac{T-t(x)}{T}.
\end{equation} 
Now we extend $w$ to the whole cylinder $[0,T] \times E_2$ by setting, for $(t,x) \not \in E$, $w(t,x)=w(t(x),x)$, so that $w(\cdot,x)$ is continuous on $[0,T]$ for any $x \in E_2$.  Observe also that if $E$ is unbounded $\lim_{x \to \infty}w(t,x)=u_\infty(t)+\frac{\varepsilon}{2}\frac{T-t}{T}=:w_\infty(t)$ uniformly with respect to $t \in [0,T]$.

Since $\frac{T-t}{T},\frac{T-t(x)}{T}<1$ for all $(t,x) \in [0,T]\times E_2$, we have that
\begin{equation}\label{eqw}
	w(t,x)\le u(t,x)+\frac{\varepsilon}{2} \quad \forall (t,x) \in  [0,T] \times \, E_2.
\end{equation}
Specifically, in $(t_0,x_0)$ it holds
\begin{equation}
	w(t_0,x_0)=u(t_0,x_0)+\frac{\varepsilon}{2}\frac{T-t_0}{T}\ge u(t_0,x_0)=\varepsilon+M_1
\end{equation}
For any $(t,x) \in \partial_pE$ we have $M_1 \ge u(t,x)$ and then, by  \eqref{eqw},
\begin{equation}
	w(t_0,x_0)\ge \varepsilon+M_1\ge \varepsilon +u(t,x)\ge \frac{\varepsilon}{2}+w(t,x), \quad \forall (t,x)\in \partial_p E.
\end{equation}
Furthermore, if $E$ is unbounded,
\begin{equation}
	w(t_0,x_0)\ge \varepsilon+M_1\ge \varepsilon +u_\infty(t)\ge \frac{\varepsilon}{2}+w_\infty(t), \quad \forall t\in [0,T].
\end{equation}
Now let us distinguish among two cases. If $E$ is bounded, since $w \in \btrev{C}(\overline{E})$, we know there exists a point $(t_1,x_1) \in \overline{E}$ such that $w(t_1,x_1)=\max_{(t,x) \in \overline{E}}w(t,x)$ and $t_1=\min \argmax_{t \in [0,T]}w(t,x_1)$. Furthermore, $w(t_1,x_1) \ge w(t_0,x_0)>w(t,x)$ for all $(t,x) \in \partial_p E$, hence $(t_1,x_1) \in E^\ast$. If $E$ is unbounded then there exists a sequence $(t^n,x^n)$ such that $\lim_{n \to \infty}w(t^n,x^n)=\sup_{(t,x) \in E}w(t,x)$. We can always extract a (non-relabelled) subsequence $(t^n,x^n)$ such that $|x^n|$ is monotone. If $|x^n| \uparrow +\infty$, then we can extract a (non-relabelled) subsequence $(t^n,x^n)$ such that $|x^n| \uparrow +\infty$ and $t^n \to t^\star \in [0,T]$ so that 
\begin{equation*}
\sup_{(t,x) \in E}w(t,x)=\lim_{n \to \infty}w(t^n,x^n)=w_\infty(t^\star)<w(t_0,x_0),	
\end{equation*}
that is absurd. Hence either $|x^n|$ is non-increasing or non-decreasing and bounded. Again, we can extract a (non-relabelled) subsequence $(t^n,x^n)$ converging towards a point $(t^\star,x_1) \in \overline{E}$ such that $w(t^\star,x_1)=\max_{(t,x) \in \overline{E}}w(t,x)$. If we set $t_1=\min \argmax_{t \in [0,T]} w(t,x_1)$, we still have $w(t_1,x_1)=\max_{(t,x) \in \overline{E}}w(t,x)$. Furthermore, arguing as before, $(t_1,x_1) \in E^\ast$. 

Since $w(t_1,\cdot) \in C^2(E_2(t_1))$, we have
\begin{equation}
	\begin{cases}
		\nabla w(t_1,x_1)=0,\\
		\Delta w(t_1,x_1)\le 0.
	\end{cases}
\end{equation}
Moreover, observe that
\begin{equation}\label{eq:A46}
	\partial_t^\Phi u(t,x)=\partial_t^\Phi w(t,x)+\frac{\varepsilon}{2T}I_\Phi(t), \, \forall (t,x) \in E.
\end{equation}
Hence, if $t(x_1)=0$, then $\partial_t^\Phi w(\cdot,x_1) \in C(0,t_1] \cap L^1[0,t_1]$ and then by Lemma \ref{lemma34} we have 
\begin{equation}\label{eq:ineqpart}
	\partial_t^\Phi w(t_1,x_1) \ge 0.
\end{equation}
If $t(x_1)\not =0$, notice that setting $\widetilde{w}(t)=w(t+t(x_1),x_1)$, it is clear that $\partial_t^\Phi w(t,x_1)=\partial_t^\Phi \widetilde{w}(t-t(x_1))$. Again \eqref{eq:A46} guarantees that $\partial_t^\Phi \widetilde{w}(\cdot) \in C(0,t_1-t(x_1)] \cap L^1[0,t_1-t(x_1)]$ and $(t_1-t(x_1))=\min\argmax_{t \in [0,t_1-t(x_1)]}\widetilde{w}(t)$, hence by Lemma \ref{lemma34} we get again \eqref{eq:ineqpart}.
Finally, we get
\begin{equation}
	\partial_t^\Phi u(t_1,x_1)-Lu(t_1,x_1)=\partial_t^\Phi w(t_1,x_1)+\frac{\varepsilon}{2T}I(t_1)-p_2(x_1)\Delta w(t_1,x_1)>0
\end{equation}
which is absurd.
\end{proof}

\subsection{Proof of Proposition \ref{prop:regx}}\label{Appreg}
Let $(t_n,x_n,y_n) \to (t,x,y)$ in $\R^+ \times \R^2$. Notice that
\begin{equation*}
	p_\Phi(t_n,x_n;y_n)=\E_0\left[p(L(t_n),x_n;y_n)\right].
\end{equation*}
Furthermore, there exists $t_0>0$ such that $t_n,t \ge t_0$ for all $n \in \N$ and then
\begin{equation*}
	p(L(t_n),x_n;y_n) \le \frac{1}{\sqrt{2\pi L(t_0)}}
\end{equation*}
where the right-hand side has finite expectation \textcolor{black}{by \eqref{Up}}. Hence, by dominated convergence theorem, we have that $p_\Phi \in C(\R^+ \times \R^2)$. Now consider a compact $K \subset \R^2 \setminus {\sf diag}(\R^2)$ and let $\varepsilon=\min_{(x,y) \in K}|x-y|$. Then it holds
\begin{equation}\label{eq:preconv}
	\sup_{(x,y) \in K}p_\Phi(t,x;y) \le \E_0\left[p(L(t),\varepsilon;0)\right].
\end{equation}
Recalling that, for all $r, \lambda>0$,
\begin{equation*}
	\sqrt{r}e^{-\lambda r} \le \frac{1}{\sqrt{2\lambda e}},
\end{equation*}
we can use the dominated convergence theorem to ensure that the right-hand side of \eqref{eq:preconv} converges to $0$, hence getting that
\begin{equation*}
	\lim_{t \downarrow 0}p_\Phi(t,x;y)=0 \mbox{ locally uniformly in } \R^2 \setminus {\sf diag}(\R^2). 
\end{equation*}
Next, consider any function $f \in C_{\rm b}(\R)$ and observe that, by a simple application of Fubini's theorem,
\begin{equation*}
	\int_{\R}p_\Phi(t,x;y)f(y)\, dy= \E_0\left[\int_{\R}p(L(t),x;y)f(y)\, dy\right].
\end{equation*}
Notice that
\begin{equation*}
	\int_{\R}p(L(t),x;y)\left|f(y)\right|\, dy \le \Norm{f}{\infty}
\end{equation*}
hence by dominated convergence theorem we have
\begin{equation*}
	\lim_{t \downarrow 0}\int_{\R}p_\Phi(t,x;y)f(y)\, dy=\E_0\left[\lim_{t \downarrow 0}\int_{\R}p(L(t),x;y)f(y)\, dy\right]=f(x),
\end{equation*}
where we also used that $p(t,x;y)\, dy \to \delta_x(dy)$ weakly as $t\downarrow 0$.

Now recall that
\begin{equation*}
	\partial_x p(t,x;y)=-\frac{1}{\sqrt{2\pi t^3}}(x-y)e^{-\frac{(x-y)^2}{2t}}.
\end{equation*}
Fix $(t,x,y) \in \R^+ \times (\R^2 \setminus {\sf diag}(\R^2))$ and let $\delta>0$ be such that $y \not \in [x-\delta,x+\delta]$. Let also $\varepsilon=\max\{x-y-\delta, y-x-\delta\}$. Since, for all $r,\lambda>0$,
\begin{equation*}
	re^{-\lambda r} \le \frac{1}{\lambda e},
\end{equation*}
we have, for $z \in [x-\delta,x+\delta]$,
\begin{equation*}
	\left|\partial_x p(L(t),z;y)\right|\le \sqrt{\frac{2}{ e^2 \pi L(t)}}|z-y|^{-1} \le \varepsilon^{-1}\sqrt{\frac{2}{ e^2 \pi L(t)}},
\end{equation*}
where we used the fact that $|z-y|>\varepsilon$. Notice that the right-hand side has finite expectation. Hence, by a simple application of the dominated convergence theorem, we can derive under the expectation sign in
\begin{equation*}
	p_\Phi(t,x;y)=\E_0\left[p(L(t),x;y)\right] 
\end{equation*}
to achieve
\begin{equation*}
	\partial_x p_\Phi(t,x;y)=\E_0\left[\partial_x p(L(t),x;y)\right], 
\end{equation*}
that is \eqref{derx}. 

Now let $(t_n,x_n,y_n) \to (t,x,y)$ in $\R^+ \times (\R^2 \setminus {\sf diag}(\R^2))$. Then there exist $t_0>0$ and a compact set $K \subset \R^2 \setminus {\sf diag}(\R^2)$ such that $(t_n,x_n,y_n),(t,x,y) \in [t_0,+\infty) \times K$ for all $n \in \N$. Let $\varepsilon=\min_{(x,y) \in K}|x-y|>0$ and notice that, as before
\begin{equation*}
	\left|\partial_x p(L(t_n),x_n;y_n)\right|\le \varepsilon^{-1}\sqrt{\frac{2}{ e^2 \pi L(t_0)}},
\end{equation*}
hence continuity in $(t,x,y)$ follows by dominated convergence. Next, consider a compact set $K \subset \R^2 \setminus {\sf diag}(\R^2)$ and let $\varepsilon=\min_{(x,y) \in K}|x-y|$ and $M=\max_{(x,y) \in K}|x-y|$. Then it holds
\begin{equation}\label{eq:preconvder1}
	\sup_{(x,y) \in K}\left|\partial_x p_\Phi(t,x;y)\right| \le \frac{1}{\sqrt{2\pi}}M\E_0\left[(L(t))^{-\frac{3}{2}}e^{-\frac{\varepsilon^2}{2L(t)}}\right].
\end{equation}
Recalling that, for all $r, \lambda>0$,
\begin{equation*}
	r^{\frac{3}{2}}e^{-\lambda r} \le \left(\frac{3}{2\lambda e}\right)^{\frac{3}{2}},
\end{equation*}
and then, a.s.,
\begin{equation*}
	(L(t))^{-\frac{3}{2}}e^{-\frac{\varepsilon^2}{2L(t)}} \le \left(\frac{3}{\varepsilon^2 e}\right)^{\frac{3}{2}}
\end{equation*}
we can use the dominated convergence theorem to ensure that the right-hand side of \eqref{eq:preconvder1} converges to $0$, hence getting that
\begin{equation*}
	\lim_{t \downarrow 0}\partial_x p_\Phi(t,x;y)=0 \mbox{ locally uniformly in } \R^2 \setminus {\sf diag}(\R^2). 
\end{equation*}

Concerning the second derivative, notice that
\begin{equation*}
	\partial^2_x p(t,x;y)=\frac{1}{\sqrt{2\pi t^3}}\left(\frac{(x-y)^2}{t}-1\right)e^{-\frac{(x-y)^2}{2t}}
\end{equation*}
and then
\begin{equation*}
	\left|\partial^2_x p(t,x;y)\right| \le \frac{1}{\sqrt{2\pi t^3}}\left(\frac{(x-y)^2}{t}+1\right)e^{-\frac{(x-y)^2}{2t}}
\end{equation*}
Fix $(t,x,y) \in \R^+ \times (\R^2 \setminus {\sf diag}(\R^2))$ and let $\delta>0$ be such that $y \not \in [x-\delta,x+\delta]$. Set also $\varepsilon=\max\{x-y-\delta,y-x-\delta\}$. Since, for all $r,\lambda>0$,
\begin{equation*}
	re^{-\lambda r} \le \frac{1}{\lambda e} \mbox{ and } r^2e^{-\lambda r} \le \frac{4}{\lambda^2 e^2},
\end{equation*}
we have, for $z \in [x-\delta,x+\delta]$,
\begin{equation*}
	\left|\partial^2_x p(L(t),z;y)\right|\le \varepsilon^{-2}\sqrt{\frac{2}{\pi  e^2 L(t)}}\left(1+\frac{8}{e}\right)
\end{equation*}
where the right-hand side has finite expectation. Hence, by a simple application of the dominated convergence theorem, we can derive under the expectation sign in
\begin{equation*}
	\partial_x p_\Phi(t,x;y)=\E_0\left[\partial_x p(L(t),x;y)\right] 
\end{equation*}
to achieve
\begin{equation*}
	\partial^2_x p_\Phi(t,x;y)=\E_0\left[\partial^2_x p(L(t),x;y)\right], 
\end{equation*}
that is \eqref{derx2}. 

Now let $(t_n,x_n,y_n) \to (t,x,y)$ in $\R^+ \times (\R^2 \setminus {\sf diag}(\R^2))$. Then there exist $t_0>0$ and a compact set $K \subset \R^2 \setminus {\sf diag}(\R^2)$ such that $(t_n,x_n,y_n),(t,x,y) \in [t_0,+\infty) \times K$ for all $n \in \N$. Let $\varepsilon=\min_{(x,y) \in K}|x-y|>0$ and notice that, as before
\begin{equation*}
	\left|\partial^2_x p(L(t_n),x_n;y_n)\right|\le \varepsilon^{-2}\sqrt{\frac{2}{\pi  e^2 L(t_0)}}\left(1+\frac{8}{e}\right),
\end{equation*}
hence continuity in $(t,x,y)$ follows by dominated convergence. Next, consider a compact set $K \subset \R^2 \setminus {\sf diag}(\R^2)$ and let $\varepsilon=\min_{(x,y) \in K}|x-y|$ and $M=\max_{(x,y) \in K}|x-y|$. Then it holds
\begin{equation}\label{eq:preconvder}
	\sup_{(x,y) \in K}\left|\partial^2_x p_\Phi(t,x;y)\right| \le \frac{1}{\sqrt{2\pi}}\E_0\left[\left(M^2(L(t))^{-\frac{5}{2}}+L(t)^{-\frac{3}{2}}\right)e^{-\frac{\varepsilon^2}{2L(t)}}\right].
\end{equation}
Recalling that, for all $r, \lambda>0$,
\begin{equation*}
	r^{\frac{3}{2}}e^{-\lambda r} \le \left(\frac{3}{2\lambda e}\right)^{\frac{3}{2}} \mbox{ and } r^{\frac{5}{2}}e^{-\lambda r} \le \left(\frac{5}{2\lambda e}\right)^{\frac{5}{2}},
\end{equation*}
and then, a.s.,
\begin{equation*}
	\left(M^2(L(t))^{-\frac{5}{2}}+L(t)^{-\frac{3}{2}}\right)e^{-\frac{\varepsilon^2}{2L(t)}} \le M^2\left(\frac{5}{\varepsilon^2 e}\right)^{\frac{5}{2}}+\left(\frac{3}{\varepsilon^2 e}\right)^{\frac{3}{2}}
\end{equation*}
we can use the dominated convergence theorem to ensure that the right-hand side of \eqref{eq:preconvder} converges to $0$, hence getting that
\begin{equation*}
	\lim_{t \downarrow 0}\partial^2_x p_\Phi(t,x;y)=0 \mbox{ locally uniformly in } \R^2 \setminus {\sf diag}(\R^2). 
\end{equation*}

Now we prove \eqref{diffpPhi} under the additional assumption \eqref{orcond}. Assume that $x \not = y $ and consider a compact set $[0,T] \subset (0,+\infty)$ and let $t \in [0,T]$ and $s \in (0,+\infty)$. By \eqref{upbound} we have that
\begin{equation}\label{eq:predertime}
p(s,x;y)\left|\partial_t f_{L}(s,t)\right| \le \frac{I_\Phi(T)}{\pi}p(s,x;y)\int_0^{+\infty}\xi e^{-s\Re(\Psi(\xi))}\, d\xi=:I_1(s,x;y),
\end{equation}
where the right-hand side is independent of $t \in [0,T]$. Next, notice that
\begin{equation}\label{eq:I1sxyint}
	\int_0^{+\infty} I_1(s,x;y)\, ds = \frac{I_\Phi(T)}{\sqrt{2\pi^3}}\int_0^{+\infty}\int_0^{+\infty}s^{-\frac{1}{2}}e^{-\frac{(x-y)^2}{2s}-s\Re(\Psi(\xi))}\xi\, ds \, d\xi.
\end{equation}
Let us first evaluate the inner integral. To do this, rewrite
\begin{multline*}
	\int_0^{+\infty}s^{-\frac{1}{2}}e^{-\frac{(x-y)^2}{2s}-s\Re(\Psi(\xi))}ds\\=2\left(\frac{|x-y|}{\sqrt{2\Re(\Psi(\xi))}}\right)^{\frac{1}{2}}\frac{1}{2}\left(\frac{|x-y|}{\sqrt{2\Re(\Psi(\xi))}}\right)^{-\frac{1}{2}}\int_0^{+\infty}s^{-\left(-\frac{1}{2}\right)-1}e^{-\frac{2\Re(\Psi(\xi))}{2}\left(\frac{(x-y)^2}{2\Re(\Psi(\xi))s}+s\right)}ds.
\end{multline*}
that, by \cite[Formula 3.471.9]{gradshteyn2014table}, becomes
\begin{equation}
	\int_0^{+\infty}s^{-\frac{1}{2}}e^{-\frac{(x-y)^2}{2s}-\Re(\Psi(\xi))s}ds=2\left(\frac{|x-y|}{\sqrt{2\Re(\Psi(\xi))}}\right)^{\frac{1}{2}}K_{\frac{1}{2}}(\sqrt{2\Re(\Psi(\xi))}|x-y|),
\end{equation}
where $K_\nu(\cdot)$ is the modified Bessel function of the third kind (see \cite{gradshteyn2014table}). In particular, it holds, by \cite[Formula 8.469.3]{gradshteyn2014table}
\begin{equation}
	K_{\frac{1}{2}}(\sqrt{2\Re(\Psi(\xi))}|x-y|)=\sqrt{\frac{\pi}{2\sqrt{2\Re(\Psi(\xi))}|x-y|}}e^{-|x-y|\sqrt{2\Re(\Psi(\xi))}},
\end{equation}
hence
\begin{equation}\label{notint}
	\int_0^{+\infty}s^{-\frac{1}{2}}e^{-\frac{(x-y)^2}{2s}-\Re(\Psi(\xi))s}ds=\sqrt{\frac{\pi}{\Re(\Psi(\xi))}}e^{-|x-y|\sqrt{2\Re(\Psi(\xi))}}.
\end{equation}
Going back to \eqref{eq:I1sxyint} we have
\begin{equation}\label{eq:I1sxyint2}
	\int_0^{+\infty} I_1(s,x;y)\, ds = \frac{I_\Phi(T)}{\pi\sqrt{2}}\int_0^{+\infty}\frac{\xi}{\sqrt{\Re(\Psi(\xi))}}e^{-|x-y|\sqrt{2\Re(\Psi(\xi))}} \, d\xi.
\end{equation}
Next, we split the integral as
\begin{align}\label{eq:boundbrutto}
	\begin{split}
	&\int_0^{+\infty} I_1(s,x;y)\, ds = \frac{I_\Phi(T)}{\pi\sqrt{2}}\left(\int_0^{1}\frac{\xi}{\sqrt{\Re(\Psi(\xi))}}e^{-|x-y|\sqrt{2\Re(\Psi(\xi))}} \, d\xi \right.\\
	&\quad \left.+\int_1^{M_\gamma}\frac{\xi}{\sqrt{\Re(\Psi(\xi))}}e^{-|x-y|\sqrt{2\Re(\Psi(\xi))}} \, d\xi +\int_{M_\gamma}^{+\infty}\frac{\xi}{\sqrt{\Re(\Psi(\xi))}}e^{-|x-y|\sqrt{2\Re(\Psi(\xi))}} \, d\xi\right)\\
	&\le \frac{I_\Phi(t_1)}{\pi\sqrt{2}}\left(\frac{1}{\sqrt{C_0}}+\frac{M_\gamma(M_\gamma-1)}{\sqrt{\min_{\xi \in [1,M_\gamma]}\Re(\Psi(\xi))}} +\frac{1}{\sqrt{C_\gamma}}\int_{0}^{+\infty}\xi^{\frac{\gamma}{2}}e^{-2|x-y|\xi^{1-\frac{\gamma}{2}}} \, d\xi\right)<\infty,	
	\end{split}
\end{align}
where we used Lemma \ref{lemmaintzero} and the fact that $\Re(\Psi(\xi))$ is continuous in $\xi$ and $\Re(\Psi(\xi))>0$ for $\xi>0$. Hence, observing that, for $t \in [t_0,t_1]$ and fixed $x,y$ such that $x\not=y$, we proposed an upper bound in \eqref{eq:predertime} that belongs to $L^1(\R^+)$, by a simple application of the dominated convergence theorem, we achieve \eqref{diffpPhi}. 

Concerning the continuity of the function $\partial_t p_\Phi(t,x;y)$ we proceed as follows. Consider a vector $(t,x;y) \in \R^+ \times (\R^2 \setminus {\sf diag}(\R^2))$ and $\{(t_n,x_n,y_n)\}_{n \in \N}\subset \R^+ \times (\R^2 \setminus {\sf diag}(\R^2))$ such that $(t_n,x_n,y_n) \to (t,x,y)$. Clearly there exists a compact $K \subset \R^2 \setminus {\sf diag}(\R^2)$ and $T>0$ such that $(x_n,y_n) \in K$ and $t_n \in [0,T]$ for all $n \in \N$. Let $\varepsilon=\min_{(z_1,z_2) \in K}|z_1-z_2|>0$. Then we have, by \eqref{upbound}
\begin{equation*}
	p(s,x_n;y_n)|\partial_t f_L(s;t_n)| \le \frac{I_\Phi(T)}{\pi}p(s,\varepsilon;0)\int_0^{+\infty}\xi e^{-s\Re(\Psi(\xi))}\, d\xi=:I_1(s,\varepsilon;0).
\end{equation*}
as in \eqref{eq:predertime}. Since we already know that $I_1(\cdot,\varepsilon;0) \in L^1(\R^+)$, we can use the dominated convergence theorem in \eqref{diffpPhi} to guarantee that $\lim_{n \to \infty}\partial_t p_\Phi(t_n,x_n;y_n)=p_\Phi(t,x;y)$. Since $(t,x;y) \in \R^+ \times (\R^2 \setminus {\sf diag}(\R^2))$ is arbitrary, this proves that $\partial_t p_\Phi \in C(\R^+ \times (\R^2 \setminus {\sf diag}(\R^2)))$.

Finally, in order to show \eqref{eq:locunifpartpphi} consider a compact $K \subset \R^2 \setminus {\sf diag}(\R^2)$ and let $\varepsilon=\min_{(x,y) \in K}|x-y|$. Then
\begin{equation}\label{eq:upboundsuppart2}
	\sup_{(x,y) \in K}|\partial_t p_\Phi(t,x;y)| \le \int_0^{+\infty} p(s,\varepsilon;0)|\partial_t f_L(s;t)|\, ds.
\end{equation}
Let $T>0$ and notice that for $t \le T$ it holds
\begin{equation*}
	p(s,\varepsilon;0)|\partial_t f_L(s;t)| \le I_1(s,\varepsilon;0)
\end{equation*}
where the right-hand side belongs to $L^1(\R^+)$. Hence, by dominated convergence theorem and \eqref{eq:upboundsuppart2}, we get
\begin{equation}\label{eq:upboundsuppart}
	\lim_{t \downarrow 0}\sup_{(x,y) \in K}|\partial_t p_\Phi(t,x;y)|=0.
\end{equation}

Now observe that for $x \not = y$ we have that $p_\Phi(\cdot,x;y) \in C^1(\R_0^+)$, once we set $\partial_t p_\Phi(0,x;y)=0$. Hence, by Lemma \ref{lem:derdentro} and \eqref{diffpPhi}, it holds
\begin{equation}\label{eq:derphiphi1}
	\partial_t^\Phi p_\Phi(t,x;y)=\int_0^t \overline{\nu}(t-s)\partial_tp_\Phi(s,x;y)\, ds=\int_0^t \overline{\nu}(t-s)\left(\int_0^{+\infty}p(\tau,x;y)\partial_t f_L(\tau;s)\, d\tau\right) \, ds.
\end{equation}
Notice that, by \eqref{upbound} and \eqref{notint},
\begin{align*}
	\int_0^t &\int_0^{+\infty} \overline{\nu}(t-s)p(\tau,x;y)\left|\partial_t f_L(\tau;s)\right|\, d\tau \, ds \\
	 &\le \frac{1}{\pi}\int_0^t \int_0^{+\infty}\int_0^{+\infty} \overline{\nu}(t-s)p(\tau,x;y)I_\Phi(s)\xi e^{-\tau \Re(\Psi(\xi))}\, d\xi \, d\tau \, ds\\
	 &=\frac{1}{\sqrt{2\pi^2}}\left(\int_0^{+\infty} \frac{\xi}{\sqrt{\Re(\Psi(\xi))}}e^{-|x-y|\sqrt{2 \Re(\Psi(\xi))}}\, d\xi\right)\left(\int_0^t \overline{\nu}(t-s)I_\Phi(s) \, ds\right)<\infty,
\end{align*}
where the finiteness of the first integral follows as in \eqref{eq:boundbrutto}. Hence we can use Fubini's theorem in \eqref{eq:derphiphi1} to achieve
\begin{equation*}
	\partial_t^\Phi p_\Phi(t,x;y)=\int_0^{+\infty}p(\tau,x;y)\left(\int_0^t \overline{\nu}(t-s)\partial_t f_L(\tau;s)\, ds\right) \, d\tau.
\end{equation*}
Again, setting $\partial_t f_L(\tau;0)=0$ for $\tau>0$, we know that $f_L(\tau;\cdot) \in C^1(\R^+_0)$ and we can use again Lemma \ref{lem:derdentro} and Proposition \ref{prop:fL} to get
\begin{equation*}
	\partial_t^\Phi p_\Phi(t,x;y)=\int_0^{+\infty}p(\tau,x;y)\partial^\Phi_t f_L(\tau;t) \, d\tau=-\int_0^{+\infty}p(s,x;y)\partial_s f_L(s;t) \, ds.
\end{equation*}

To show that $\partial_t^\Phi p_\Phi \in C(\R^+ \times (\R^2 \setminus {\sf diag}(\R^2)))$ fix $(t,x,y) \in \R^+ \times (\R^2 \setminus {\sf diag}(\R^2))$ and consider a sequence $\{(t_n,x_n,y_n)\}_{n \in \N}\subset\R^+ \times (\R^2 \setminus {\sf diag}(\R^2))$ such that $(t_n,x_n,y_n) \to (t,x,y)$. Then there exist a compact $K \subset \R^2 \setminus {\sf diag}(\R^2)$ and $T>0$ such that $t_n \in [0,T]$ and $(x_n,y_n) \in K$ for all $n \in \N$. Let also $\varepsilon=\min_{(z_1,z_2) \in K}|z_1-z_2|$. By \eqref{upbound2} we have
\begin{equation*}
	\left|p(s,x;y)\partial_s f_L(s;t)\right| \le I_2(s,\varepsilon;0),
\end{equation*}
where
\begin{align}
	I_2(s,x;y)&= \frac{p(s,x;y)I_\Phi(T)(\sqrt{2}J_1(s)+J_2(s))}{\pi}, \label{I2}\\
	J_1(s)&=\int_{0}^{M_\gamma}|\Psi(\xi)|e^{-s\Re(\Psi(\xi))}\, d\xi \label{J1g}\\
	J_2(s)&=\int_{M_\gamma}^{+\infty}\left(3\overline{\nu}(1)+\xi \int_0^1 \tau \nu(d\tau)+\frac{\xi^2}{2}\int_0^1 \tau^2\nu(d\tau)\right)e^{-s\xi^{2-\gamma}}\, d\xi\label{J2g}.
\end{align}
Now we need to show that $I_2(\cdot,\varepsilon;0) \in L^1(\R^+)$. To do this, notice that
\begin{equation*}
	\int_0^{+\infty}p(s,\varepsilon;0)J_1(s)\, ds \le M_\gamma \sup_{\xi \in [0,M_\gamma]}|\Psi(\xi)|\int_0^{+\infty}\frac{1}{\sqrt{2\pi s}}e^{-\frac{\varepsilon^2}{2s}}\, ds<\infty.
\end{equation*}
Concering $J_2(s)$, first notice that there exists a constant $C>0$ such that, for $\xi \ge M_\gamma$,
\begin{equation*}
	3\overline{\nu}(1)+\xi \int_0^1 \tau \nu(d\tau)+\frac{\xi^2}{2}\int_0^1 \tau^2\nu(d\tau) \le C \xi^2.
\end{equation*}
Furthermore, we have, arguing as in \eqref{notint},
\begin{align*}
	\int_0^{+\infty}p(s,\varepsilon;0)J_2(s)\, ds &\le C\int_{M_\gamma}^{+\infty}\frac{\xi^2}{\sqrt{2\pi}}\left(\int_0^{+\infty}s^{-\frac{1}{2}}e^{-s\xi^{2-\gamma}-\frac{\varepsilon^2}{2s}}\, ds\right) \, d\xi\\
	& \le \frac{C}{\sqrt{2}}\int_{0}^{+\infty}\xi^{1+\frac{\gamma}{2}}e^{-\varepsilon \sqrt{2\xi^{2-\gamma}}} \, d\xi<\infty.
\end{align*}
Hence $I_2(\cdot,\varepsilon;0) \in L^1(\R^+)$ and then $\lim_{n \to \infty}\partial_t^\Phi p_\Phi(t_n,x_n;y_n)=\partial_t^\Phi p_\Phi(t,x;y)$ by dominated convergence theorem. 

Finally, consider a compact set $K \subset \R^2 \setminus {\sf diag}(\R^2)$ and let $\varepsilon=\min_{(x,y) \in K}|x-y|$. Then it holds
\begin{equation*}
	\sup_{(x,y) \in K}\left|\partial_t^\Phi p_\Phi(t,x;y)\right| \le \frac{I_\Phi(t)}{\pi} \int_0^{+\infty}p(s,\varepsilon;0)(\sqrt{2}J_1(s)+J_2(s))\, ds
\end{equation*}
where the right-hand side clearly goes to $0$ as $t \downarrow 0$.
\qed
\subsection{Proof of Proposition \ref{prop:limitseconder}} \label{AppLimits}
\textcolor{black}{We will prove the statement under the milder assumptions of Remark \ref{rmk:lessrestrict}.} Since $p_\Phi(t,x;y)=p_\Phi(t,x-y;0)$, it is sufficient to prove the statement for $y=0$. Recall that, by \eqref{derx}, we have
\begin{equation}
	\partial_x p_\Phi(t,x;0)=-\frac{1}{\sqrt{2\pi}}\int_0^{+\infty}\frac{xe^{-\frac{x^2}{2s}}}{s^\frac{3}{2}}f_L(s;t)ds.
\end{equation}
%
%
First consider $x>0$. By using the change of variables $z=\frac{x}{\sqrt{2s}}$ {for $x>0$}, {the latter becomes}
\begin{equation}
	\partial_x p_\Phi(t,x;0)=-\frac{2}{\sqrt{\pi}}\int_0^{+\infty}e^{-z^2}f_L\left(\frac{x^2}{2z^2};t\right)dz.
\end{equation}
Notice further that
\begin{equation}
	\overline{\nu}(t)+\partial_x p_\Phi(t,x;0)=\frac{2}{\sqrt{\pi}}\int_0^{+\infty}e^{-z^2}\left(\overline{\nu}(t)-f_L\left(\frac{x^2}{2z^2};t\right)\right)dz.
\end{equation}
Now consider a compact set $K \subset \R^+$ and notice that
\begin{equation}\label{eq:estprelim}
	\sup_{t \in K}\left|\overline{\nu}(t)+\partial_x p_\Phi(t,x;0)\right|\le \frac{2}{\sqrt{\pi}}\int_0^{+\infty}e^{-z^2}\sup_{t \in K}\left|\overline{\nu}(t)-f_L\left(\frac{x^2}{2z^2};t\right)\right|dz.
\end{equation}
By Assumption \eqref{ass4} we know that
\begin{equation*}
	\lim_{s \downarrow 0}\sup_{t \in K}\left|\overline{\nu}(t)-f_L\left(s;t\right)\right|=0.
\end{equation*}
Furthermore, by Proposition \ref{derfL}, we know that $\lim_{s \to +\infty}\sup_{t \in K}\left|f_L\left(s;t\right)\right|=0$. Hence, there exists a constant $C>0$ such that
\begin{equation*}
	\sup_{t \in K}\left|\overline{\nu}(t)-f_L\left(s;t\right)\right| \le C, \ \forall s \ge 0.
\end{equation*} 
Hence, by dominated convergence theorem and \eqref{eq:estprelim} we achieve
\begin{equation*}
	\lim_{x \downarrow 0}\sup_{t \in K}\left|\overline{\nu}(t)+\partial_x p_\Phi(t,x;0)\right|=0.
\end{equation*}
The argument for $x<0$ is analogous.

Now recall that, by \eqref{diffPhipPhi}, for $x \not = 0$ it holds
\begin{equation*}
	\partial_t^\Phi p_\Phi(t,x;0)=-\int_{0}^{+\infty}p(s,x;0)\partial_s f_L(s,t)\, ds.
\end{equation*}
First, notice that by Proposition \ref{derfL} we know that $\lim_{s \to \infty}s^2\left|\partial_s f_L(s,t)\right|=0$ locally uniformly with respect to $t \in (0,+\infty)$. Now fix a compact set $K \subset (0,+\infty)$ and let $\delta>0$ and $h:(0,\delta) \to \R^+_0$ as in Assumption \eqref{ass4}. Then we know that there exists a constant $C>0$ such that
\begin{equation*}
	\sup_{t \in K}|\partial_s f_L(s,t)| \le Cs^{-2}, \ s \ge \delta
\end{equation*}
and then
\begin{equation}\label{eq:controlpdfLs}
	\int_0^{+\infty}s^{-\frac{1}{2}}\sup_{t \in K}|\partial_ s f_L(s,t)|\, ds \le \int_0^{\delta}s^{-\frac{1}{2}}h(s)\, ds+\int_{\delta}^{+\infty}s^{-\frac{5}{2}}\, ds<\infty.
\end{equation}
Now notice that
\begin{align*}
	\sup_{t \in K}\left|\partial_t^\Phi p_\Phi(t,x;0)+\frac{1}{\sqrt{2\pi}}\int_0^{+\infty}s^{-\frac{1}{2}}\partial_sf_L(s,t)\, ds\right| \le \frac{1}{\sqrt{2\pi}}\int_{0}^{+\infty}\left(1-e^{-\frac{x^2}{2s}}\right)s^{-\frac{1}{2}}\sup_{t \in K}\left|\partial_s f_L(s,t)\right|\, ds.
\end{align*}
The integrand can be controlled as follows
\begin{equation*}
	\left(1-e^{-\frac{x^2}{2s}}\right)s^{-\frac{1}{2}}\sup_{t \in K}\left|\partial_s f_L(s,t)\right| \le s^{-\frac{1}{2}}\sup_{t \in K}\left|\partial_s f_L(s,t)\right|,
\end{equation*}
where the right-hand side is integrable, hence by dominated convergence we have
\begin{equation}\label{eq:uniformconvat0}
	\lim_{x \to 0}\sup_{t \in K}\left|\partial_t^\Phi p_\Phi(t,x;0)+\frac{1}{\sqrt{2\pi}}\int_0^{+\infty}s^{-\frac{1}{2}}\partial_sf_L(s,t)\, ds\right|=0.
\end{equation}
Finally,
\begin{equation*}
	\lim_{x \to 0}\sup_{t \in K}\left|\partial_x^2 p_\Phi(t,x;0)+\frac{1}{\sqrt{2\pi}}\int_0^{+\infty}s^{-\frac{1}{2}}\partial_s f_L(s,t)\, ds\right|=0 \ \mbox{ for any compact }K\subset \R^+
\end{equation*}
follows by \eqref{eq:uniformconvat0} and Theorem \ref{lemmaeqfraz}. \qed
\subsection{Proof of Proposition \ref{prop:secder}}\label{Appsecder}
\btrev{
	Let us first observe that if we define
	\begin{equation*}
		v(t,x)=\int_{\R}f(y)p(t,x;y)\, dy,
	\end{equation*}
	that is a bounded and continuous function in $\R^+_0 \times \R$, then we have, by a simple application of Fubini's theorem,
	\begin{equation*}
		u(t,x)=\E[v(L_\Phi(t),x)], \ (t,x) \in \R_0^+ \times \R.
	\end{equation*}
	Thus, since $L_\Phi$ is a.s. continuous, $u \in C(\R^+_0 \times \R)$.}
	 
\btrev{	Now we move to the derivatives. Fix $x \in \R$ and $t>0$. Let $\chi \in C^\infty_{\sf c}(\R)$ be such that $\chi(z)=1$ for all $z \in \left[-\frac{1}{2},\frac{1}{2}\right]$, $\chi(z)=0$ for all $z \not \in [-1,1]$ and, in general, $\chi(z) \in [0,1]$. Let also $L=\Norm{\chi^\prime}{L^\infty(\R)}$. Consider $f \in {\sf C}_b(\R)$ and write
	\begin{align}\label{eq:uu1u2}
		\begin{split}
		u(t,x)&=\int_{\R}f(y)p_\Phi(t,x;y)\, dy\\
		&=\int_{\R}f(y)\chi(x-y)p_\Phi(t,x;y)\, dy-\int_{\R}f(y)(\chi(x-y)-1)p_\Phi(t,x;y)\, dy\\
		&=u_1(t,x)-u_2(t,x).
		\end{split}
	\end{align}
	Let us first determine $\partial_x u_1(t,x)$ and show that $\partial_x u_1 \in C(\R^+ \times \R)$. To do this, let $I_x:=[x-1,x+1]$ and $I_x^\prime:=[x-2,x+2]$. Recall that, by Proposition \ref{prop:regx}, $\partial_x p_\Phi \in C(\R^+ \times (\R^2 \setminus {\sf diag}(\R^2)))$ and notice that by \eqref{eq:limtder1} we have that for any $0<t_0<t_1$ and any compact set $K\subset \R^2$ it holds
	\begin{equation}\label{eq:upboundpartialpPhi}
		M:=\sup_{(t,z,y) \in [t_0,t_1]\times I_x \times I_x^\prime}|\partial_x p_\Phi(t,z;y)|<\infty.
	\end{equation}
	Now consider $h \in (-1,1)$ and define the incremental ratios
	\begin{equation}\label{eq:incrementalratio}
		\mathcal{R}^\Phi_h(t,x;y):=\frac{p_\Phi(t,x+h;y)-p_\Phi(t,x;y)}{h}, \quad \mathcal{R}^\chi_h(t,x;y):=\frac{\chi(x+h-y)-\chi(x-y)}{h}
	\end{equation}
	and observe that $x+h \in I_x$. Then the incremental ratio
	\begin{equation}\label{eq:incrratio2}
		\mathcal{R}_h(t,x;y):=\chi(x+h-y)\mathcal{R}^\Phi_h(t,x;y)+p_\Phi(t,x;y)\mathcal{R}^\chi_h(t,x;y).
	\end{equation}
	 Let us provide an upper bound for $\mathcal{R}^\Phi_h(t,x;y)$ as $y \in I^\prime_x$. To do this, notice that we have shown in Proposition \ref{prop:regx} that $\partial_x p_\Phi \in C(\R^+ \times (\R^2 \setminus {\sf diag}(\R^2)))$ and, in particular, $p_\Phi(t,\cdot;y)$ admits derivative almost everywhere. Furthermore, by \eqref{eq:upboundpartialpPhi}, we know that for fixed $t>0$ and $y \in I_x$, 
	\begin{equation*}
		\sup_{z \in I_x}\left|\partial_x p_\Phi(t,z;y)\right| \le M.
	\end{equation*}
	As a consequence, for $y \in I^\prime_x$, $p_\Phi(t,\cdot;y)$ is a Lipschitz function on $I_x$ with Lipschitz constant controlled by $M$, that only depends on $I_x, t_1, t_2$. Furthermore, notice that for any $h \in (-1,1)$, the function $f(\cdot)\chi(x+h-\cdot)$ is supported in $I^\prime_x$ and so also $f(y)\mathcal{R}_h^\chi(t,x;y)$. Moreover, as $t \in [t_1,t_2]$ and $x,y \in R$ we have
	\begin{equation*}
		p_\Phi(t,x;y) \le \frac{1}{\sqrt{2\pi}}U_{-\frac{1}{2}}(t_1),
	\end{equation*}
	hence
	\begin{equation*}
		\left|f(y)\mathcal{R}_h(t,x;y)\right| \le \Norm{f}{L^\infty(\R)}\left(M+\frac{L}{\sqrt{2\pi}}U_{-\frac{1}{2}}(t_1)\right)1_{I^\prime_x}(y),
	\end{equation*}
	where the right-hand side is integrable. Hence, by the dominated convergence theorem, 
	\begin{equation}\label{eq:deru1}
		\partial_x u_1(t,x)=\int_{R}f(y)\chi(x-y)\partial_x p_\Phi(t,x;y)\, dy+\int_{R}f(y)\partial_x\chi(x-y) p_\Phi(t,x;y)\, dy
	\end{equation}
	and the same estimates guarantee, through dominated convergence, that $\partial_x u_1 \in C(\R^+ \times \R)$. Next, let us consider $u_2$. Let $J_x=\left[x-\frac{1}{8},x+\frac{1}{8}\right]$, $J^\prime_x=\left[x-\frac{3}{8},x+\frac{3}{8}\right]$ and consider again \eqref{eq:incrementalratio}, this time with $h \in \left(-\frac{1}{8},\frac{1}{8}\right)$, and define
	\begin{equation*}
		\mathcal{R}^c_h(t,x;y)=(\chi(x+h-y)-1)\mathcal{R}^\Phi_h(t,x;y)+p_\Phi(t,x;y)\mathcal{R}^\chi_h(t,x;y).
	\end{equation*} 
	While the second summand can be controlled exactly as before, notice that the first summand has to be bounded only for $y \in \R \setminus J^\prime_x$, since $\chi(x+h-y)=1$ whenever $y \in J^\prime_x$. We observe that for $z \in J_x$ and $y \in J_x^\prime$ it holds
	\begin{equation}\label{eq:controldistance}
		|x-y|-\frac{1}{8} \le |z-y| \le |x-y|+\frac{1}{8}
	\end{equation}
	hence, for $z \in I_x$ and $y \in I_x^\prime$, we have
	\begin{align}\label{eq:aux1}
		\begin{split}
			\left|\partial_x p_\Phi(t,z;y)\right| &\le \frac{1}{\sqrt{2\pi}}\int_0^{+\infty}\frac{|z-y|e^{-\frac{(z-y)^2}{2s}}}{s^{\frac{3}{2}}}f_L(s;t)\, ds \\
			&\le \frac{1}{\sqrt{2\pi}}\int_0^{+\infty}\frac{\left(|x-y|+\frac{1}{8}\right)e^{-\frac{\left(|x-y|-\frac{1}{8}\right)^2}{2s}}}{s^{\frac{3}{2}}}f_L(s;t)\, ds. 	
		\end{split}
	\end{align}
	Now consider, for $r,s>0$, the function
	\begin{equation*}
		g_r(s)=\frac{\left(r+\frac{1}{8}\right)e^{-\frac{\left(r-\frac{1}{8}\right)^2}{2s}}}{s^{\frac{3}{2}}}.
	\end{equation*}
	One can easily check that
	\begin{equation*}
		g'_r(s)=\frac{\left(r+\frac{1}{8}\right)\left(\left(r-\frac18\right)^2-3s\right)}{2s^{\frac{7}{2}}}e^{-\frac{\left(r-\frac18\right)^2}{2s}},
	\end{equation*}
	hence $g_r(s)$ admits maximum in $3s=\left(r-\frac{1}{8}\right)^2$. As a consequence we have
	\begin{equation}\label{eq:grup}
		g_r(s) \le \sqrt{\frac{27}{e^3}}\frac{r+\frac{1}{8}}{\left(r-\frac{1}{8}\right)^3}.
	\end{equation}
	Using $r=|x-y|$ and plugging the previous inequality into \eqref{eq:aux1}, we get
	\begin{align*}
		\left|\partial_x p_\Phi(t,z;y)\right| \le \sqrt{\frac{27}{2\pi e^3}}\frac{|x-y|+\frac{1}{8}}{\left(|x-y|-\frac18\right)^3}.
	\end{align*}
	The fact that the right-hand side is integrable on $\R \setminus J^\prime_x$ easily follows from the relation 
	\begin{equation*}
		\int_{\frac{1}{4}}^{+\infty}\frac{y+2}{y^3}\, dy<+\infty.
	\end{equation*}
	In general, we have
	\begin{equation*}
		\left|f(y)\mathcal{R}^c_h(t,x;y)\right| \le \Norm{f}{L^\infty(\R)}\left(\sqrt{\frac{27}{2\pi e^3}}\frac{|x-y|+\frac{1}{8}}{\left(|x-y|-\frac18\right)^3}1_{\R \setminus J^\prime_x}(y)+\frac{L}{\sqrt{2\pi}}U_{-\frac{1}{2}}(t_1)1_{J_x}(y)\right).
	\end{equation*}
	Hence, by dominated convergence, we have
	\begin{equation}\label{eq:deru2}
		\partial_x u_2(t,x)=\int_{\R}f(y)(\chi(x-y)-1)\partial_x p_\Phi(t,x;y)\, dy+\int_{\R}f(y)\partial_x\chi(x-y) p_\Phi(t,x;y)\, dy.
	\end{equation}
	Again, the same estimates tell us that $\partial_x u_2 \in C(\R^+ \times \R)$. Plugging  \eqref{eq:deru1} and \eqref{eq:deru2} into \eqref{eq:uu1u2}, we get \eqref{eq:der1int} and $\partial_x u \in C(\R^+ \times \R)$. }
	
	\btrev{Now we want to do the same for the second derivative. Nevertheless, $\partial_xp_\Phi(t,\cdot;y)$ is not even a continuous function. To handle this problem, we define
	\begin{equation*}
		\widetilde{p}_\Phi(t,x;y)=\partial_xp_\Phi(t,x;y)+2\overline{\nu}(t)1_{[0,+\infty)}(x-y).
	\end{equation*}
	Let us first handle $u_1$. We rewrite
	\begin{align}\label{eq:der2u11}
		\begin{split}
		\partial_x u_1(t,x)&=\int_{\R}f(y)\chi(x-y)\widetilde{p}_\Phi(t,x;y)\, dy-2\nu(t)\int_{\R}f(y)\chi(x-y)1_{[0,+\infty)}(x-y)\, dy\\
		&=\int_{\R}f(y)\chi(x-y)\widetilde{p}_\Phi(t,x;y)\, dy-2\nu(t)\int_{-\infty}^xf(y)\chi(x-y)\, dy\\
		&=u_3(t,x)-2\nu(t)u_4(x).	
		\end{split}
	\end{align}
	It is clear that
	\begin{equation}\label{eq:deru4}
		\partial_x u_4(x)=f(x)+\int_{-\infty}^xf(y)\partial_x\chi(x-y)\, dy.
	\end{equation}
	On the other hand, we notice that since $\partial_x^2 p_\Phi \in C(\R^+ \times (\R^2 \setminus {\sf diag}(\R^2)))$ and \eqref{eq:limtder2} holds, then $\widetilde{p}_\Phi(t,\cdot;y) \in C^1(\R)$ with derivative
	\begin{equation*}
		\partial_x \widetilde{p}_\Phi(t,x;y)=\begin{cases} \displaystyle \partial_x^2 p_\Phi(t,x;y) & x \not = y \\
		\displaystyle -\sqrt{\frac{2}{\pi}}\int_0^{+\infty}s^{-\frac{1}{2}}\partial_s f_L(s;t)\, ds & x=y.
	\end{cases}
	\end{equation*}
	Furthermore, since the involved functions are continuous and the limits in \eqref{eq:limtder2} hold locally uniformly with respect to $t>0$, for $t \in [t_1,t_2]$, $z \in I_x$ and $y \in I^\prime_x$ we have
	\begin{equation*}
		\left|\partial_x \widetilde{p}_\Phi(t,z;y)\right|\le   \sup_{\substack{t \in [t_1,t_2] \\ z \in I_x \\ y \in I^\prime_x}}\left|\partial_x \widetilde{p}_\Phi(t,z;y)\right|:=\widetilde{M}.
	\end{equation*}
	Now we define, for $h \in (-1,1)$
	\begin{equation*}
		\widetilde{\cR}^\Phi_h(t,x;y)=\frac{\widetilde{p}_\Phi(t,x+h;y)-\widetilde{p}_\Phi(t,x;y)}{h}
	\end{equation*}
	and
	\begin{equation*}
		\widetilde{\cR}_h(t,x;y)=\chi(x+h-y)\widetilde{\cR}^\Phi_h(t,x;y)+\widetilde{p}_\Phi(t,x;y)\cR^\chi_h(t,x;y).
	\end{equation*}
	Notice that $\chi(x+h-\cdot)$ and $\cR^\chi_h(t,x;\cdot)$ is supported on $I^\prime_x$, whole $x+h \in I_x$ for any $h \in (-1,1)$. Hence we have
	\begin{equation*}
		\left|\widetilde{p}_\Phi(t,x;y)\right| \le M+\nu(t_1),
	\end{equation*}
	where $M$ is defined in \eqref{eq:upboundpartialpPhi}. Hence
	\begin{equation*}
		\left|f(y)\widetilde{\cR}_h(t,x;y)\right| \le \Norm{f}{L^\infty(\R)}1_{I^\prime_x}(y)\left(\widetilde{M}+(M+\nu(t_1))L\right),
	\end{equation*}
	where the right-hand side is integrable in $y$. Hence
	\begin{equation}\label{eq:deru3}
		\partial_x u_3(t,x)=\int_{\R}f(y)\chi(x-y)\partial^2_xp_\Phi(t,x;y)\, dy+\int_{\R}f(y)\partial_x\chi(x-y)\widetilde{p}_\Phi(t,x;y)\, dy
	\end{equation}
	and $\partial_x u_3 \in C(\R^+ \times \R)$. Plugging \eqref{eq:deru4} and \eqref{eq:deru3} into \eqref{eq:der2u11} we get
	\begin{equation}\label{eq:der2u1}
		\partial_x^2 u_1(t,x)=\int_{\R}f(y)\chi(x-y)\partial^2_xp_\Phi(t,x;y)\, dy+\int_{\R}f(y)\partial_x\chi(x-y)p_\Phi(t,x;y)\, dy-2\nu(t)f(x)-2\nu(t)\int_{-\infty}^xf(y)\partial_x \chi(x-y)\, dy
	\end{equation}
	and $\partial_x^2 u_1 \in C(\R^+ \times \R)$. }
	
\btrev{	Now we handle $u_2$. To do this, we split again $\partial_xp_\Phi(t,x;y)$ to get
	\begin{align}\label{eq:der2u21}
		\begin{split}
			\partial_x u_2(t,x)&=\int_{\R}f(y)(\chi(x-y)-1)\widetilde{p}_\Phi(t,x;y)\, dy-2\nu(t)\int_{\R}f(y)(\chi(x-y)-1)1_{[0,+\infty)}(x-y)\, dy\\
			&=\int_{\R}f(y)(\chi(x-y)-1)\widetilde{p}_\Phi(t,x;y)\, dy-2\nu(t)\int_{-\infty}^xf(y)(\chi(x-y)-1)\, dy\\
			&=u_5(t,x)-2\nu(t)u_6(x).	
		\end{split}
	\end{align}
	Again, it is clear that
	\begin{equation}\label{eq:deru6}
		\partial_x u_6(x)=\int_{-\infty}^{x}f(y)\partial_x \chi(x-y)\, dy.
	\end{equation}
	To handle $u_5(t,x)$, we define for $h \in \left(-\frac{1}{4},\frac{1}{4}\right)$
	\begin{equation*}
		\widetilde{R}_h^c(t,x;y)=(\chi(x+h-y)-1)\widetilde{\cR}^\Phi_h(t,x;y)+\widetilde{p}_\Phi(t,x;y)\cR^\chi_h(t,x;y).
	\end{equation*}
	We only need to provide an estimate of the first summand, which is $0$ except at most for $y \in \R \setminus J^\prime_x$. To do this, we notice that by \eqref{eq:controldistance} and \eqref{derx2} we have for $z \in J_x$ and $y \in \R \setminus J^\prime_x$,
	\begin{align*}
		\left|\partial^2_x p_\Phi(t,z;y)\right|&\le \frac{1}{\sqrt{2\pi}}\int_0^{+\infty}\frac{1}{ s^{\frac{3}{2}}}\left(\frac{\left(|x-y|+\frac{1}{8}\right)^2}{s}+1\right)e^{-\frac{\left(|x-y|-\frac{1}{8}\right)^2}{2s}}f_L(s;t)\, ds\\
		&=\frac{1}{\sqrt{2\pi}}\int_0^{+\infty}\frac{1}{ s^{\frac{3}{2}}}\frac{\left(|x-y|+\frac{1}{8}\right)^2}{s}e^{-\frac{\left(|x-y|-\frac{1}{8}\right)^2}{2s}}f_L(s;t)\, ds\\
		&+\frac{1}{\sqrt{2\pi}}\int_0^{+\infty}\frac{1}{ s^{\frac{3}{2}}}e^{-\frac{\left(|x-y|-\frac{1}{8}\right)^2}{2s}}f_L(s;t)\, ds.
	\end{align*}
	Now let us consider, for $r>0$, the function
	\begin{equation*}
		\widetilde{g}_r(s)=\frac{\left(r+\frac{1}{8}\right)^2}{ s^{\frac{5}{2}}}e^{-\frac{\left(r-\frac{1}{8}\right)^2}{2s}}
	\end{equation*}
	whose derivative is given by
	\begin{equation*}
		\widetilde{g}_r'(s)=\frac{\left(r+\frac{1}{8}\right)^2}{2s^{\frac{9}{2}}}e^{-\frac{\left(r-\frac{1}{8}\right)^2}{2s}}\left(\left(r-\frac{1}{8}\right)^2-5s\right)
	\end{equation*}
	and thus admits maximum in $s=\frac{1}{5}\left(r-\frac{1}{8}\right)^2$, leading to
	\begin{equation*}
		\widetilde{g}_r(s) \le \frac{\left(r+\frac{1}{8}\right)^2}{ \left(r-\frac{1}{8}\right)^{5}}\sqrt{\frac{5^5}{e^5}}.
	\end{equation*}
	Setting $r=|x-y|$ and using also \eqref{eq:grup} we finally get
	\begin{equation}\label{eq:controlsecderfar}
		\left|\partial^2_x p_\Phi(t,z;y)\right| \le \frac{\left(|x-y|+\frac{1}{8}\right)^2}{ \left(|x-y|-\frac{1}{8}\right)^{5}}\sqrt{\frac{5^5}{2\pi e^5}}+\frac{1}{ \left(|x-y|-\frac{1}{8}\right)^{3}}\sqrt{\frac{27}{2\pi e^3}}=:F(|x-y|)
	\end{equation}
	where the right-hand side is integrable in $\R \setminus J^\prime_x$ since
	\begin{equation*}
		\int_{\frac{1}{4}}^{+\infty}\left(\frac{\left(y+\frac{1}{4}\right)^2}{y^{5}}+\frac{1}{y^{3}}\right)\, dy<\infty.
	\end{equation*}
	Hence, once we observe that
	\begin{equation*}
		\left|f(y)\widetilde{\mathcal{R}}^c_h(t,x;y)\right| \le \Norm{f}{L^\infty(\R)}\left(1_{\R \setminus J^\prime_x}(y)F(|x-y|)+L(M+\nu(t_1))1_{J_x}(y)\right),
	\end{equation*}
	by dominated convergence we have
	\begin{equation}\label{eq:deru5}
		\partial_x u_5(t,x)=\int_{\R}f(y)(\chi(x-y)-1)\partial_x^2p_\Phi(t,x;y)\, dy+\int_{\R}f(y)\partial_x\chi(x-y)\widetilde{p}_\Phi(t,x;y) dy
	\end{equation}
	and with the same estimates we have that $\partial_x u_5 \in C(\R^+ \times \R)$. Using \eqref{eq:deru6} and \eqref{eq:deru5} into \eqref{eq:der2u21}, we have that $\partial_x^2 u_2 \in C(\R^+ \times \R)$ and
	\begin{equation}\label{eq:der2u2}
		\partial_x^2 u_2(t,x)=\int_{\R}f(y)\partial_x\chi(x-y)\widetilde{p}_\Phi(t,x;y)\,dy+\int_{\R}f(y)(\chi(x-y)-1)\partial_x^2p_\Phi(t,x;y)\,dy-2\nu(t)\int_{-\infty}^{x}f(y)\partial_x\chi(x-y)\, dy.
	\end{equation}
	Finally, combining \eqref{eq:der2u1} and \eqref{eq:der2u2}, we get that $\partial_x^2 u \in C(\R^+ \times \R)$ and \eqref{eq:der2int}.}
	
\btrev{	In particular, for $f=\equiv 1$, we have $u \equiv 1$
	\begin{equation*}
		0=\partial_x^2u(t,x)=\int_{\R}\partial_x^2p_\Phi(t,x;y)\, dy-2\nu(t)
	\end{equation*}
	leading to the desired result.}
\qed
\textcolor{black}{\subsection{Proof of Proposition \ref{prop:RLderpPhi}}\label{app:RLderpPhi}
	Notice that
\begin{equation}\label{eq:fracintpPhidiag}
	\mathcal{I}^{\Phi}p_\Phi(t,x;x)=\frac{1}{\sqrt{2\pi}}\int_{0}^{t}\overline{\nu}(t-\tau)\int_{0}^{+\infty}s^{-\frac{1}{2}}f_L(s;\tau)\, d\tau\, ds=\frac{1}{\sqrt{2\pi}}\int_{0}^{+\infty}s^{-\frac{1}{2}}\left(\mathcal{I}^\Phi_t f_L(s;\cdot)\right)\, ds.
\end{equation}
Next, recall that for all $s,t>0$ we have
\begin{equation*}
	\left(\mathcal{I}^\Phi_t f_L(s;\cdot)\right)(t)=-\int_0^t\partial_s f_L(s;w)\, dw.
\end{equation*}
To take the derivative inside the integral sign in \eqref{eq:fracintpPhidiag}, it is then sufficient to recall \eqref{eq:controlpdfLs} for any fixed $t>0$ and any compact neighbourhood $K$ of $t$. The latter also immediately proves \eqref{eq:boundDphi}.
\qed }
\textcolor{black}{\subsection{Proof of Proposition \ref{prop:exchangeRL}}\label{app:exchangeRL}
First observe that
\begin{equation*}
	\mathcal{I}^\Phi_t u(t,x)=\int_0^t \int_{\R} \overline{\nu}(t-\tau)f(y)p_\Phi(t,x;y)\, dy,
\end{equation*}
and we can exchange the order of the integrals since
\begin{equation*}
	\int_0^t \int_{\R} \overline{\nu}(t-\tau)\left|f(y)\right|p_\Phi(t,x;y)\, dy \le I_\Phi(t)\Norm{f}{L^\infty(\R)},
\end{equation*}
so that we have
\begin{equation*}
	\mathcal{I}^\Phi_t u(t,x)= \int_{\R}f(y)\mathcal{I}^\Phi_t p_\Phi(t,x;y)\, dy.
\end{equation*}
Now let $\chi$ as in Appendix~\ref{Appsecder}, $0<t_1<t<t_2$ and consider
\begin{align*}
	\mathcal{I}^\Phi_t u(t,x)&= \int_{\R}f(y)\chi(x-y)\mathcal{I}^\Phi_t p_\Phi(t,x;y)\, dy+\int_{\R}f(y)(1-\chi(x-y))\mathcal{I}^\Phi_t p_\Phi(t,x;y)\, dy\\
	&=:u_1(t,x)+u_2(t,x).
\end{align*}
Since we observed that $I^\Phi_t p_\Phi \in C^1(\R^+ \times \R^2)$, we notice that, for $y \in I_x:=[x-1,x+1]$ and $t \in [t_1,t_2]$ it holds
\begin{equation*}
\sup_{\substack{t \in [t_1,t_2] \\ y \in I_x}}\left|D^\Phi_tp_\Phi(t,x;y)\right|<\infty
\end{equation*}
and thus it is clear, since $f(\cdot)\chi(x-\cdot)$ is continuous and supported on $I_x$, that
\begin{equation}\label{eq:tderu1}
	\partial_t u_1(t,x)=\int_{\R}f(y)\chi(x-y)D^\Phi_t p_\Phi(t,x;y)\, dy.
\end{equation}
Concerning $u_2$, we use \eqref{eq:controlsecderfar} to show that for $y \in \R \setminus J_x$, where $J_x=\left[x-\frac{3}{8},x+\frac{3}{8}\right]$, it holds
\begin{equation*}
	\left|D^\Phi_t p_\Phi(t,z;y)\right| \le \frac{\left(|x-y|+\frac{1}{8}\right)^2}{ \left(|x-y|-\frac{1}{8}\right)^{5}}\sqrt{\frac{5^5}{8\pi e^5}}+\frac{1}{ \left(|x-y|-\frac{1}{8}\right)^{3}}\sqrt{\frac{27}{8\pi e^3}}
\end{equation*}
where the right-hand side is integrable in $\R \setminus J_x$. Hence, once we notice that $f(\cdot)(1-\chi(x-\cdot))$ is bounded and supported on $\R \setminus J_x$, we have 
\begin{equation}\label{eq:tderu2}
	\partial_t u_2(t,x)=\int_{\R}f(y)(1-\chi(x-y))D^\Phi_t p_\Phi(t,x;y)\, dy.
\end{equation}
Summing \eqref{eq:tderu1} and \eqref{eq:tderu2} we get the desired result.
\qed}
\subsection{Uniform convergence of monotone functions}\label{AppUnifConv}
\begin{prop}
	\label{uniform}
	Let $F_n,F: \R_0^+ \mapsto \R$ for any $n \in \N$ such that $F_n(t) \to F(t)$ as $n \to \infty$, for any $t \in \R_0^+$. Suppose further that $F_n, F$ are continuous, non-decreasing and $\lim_{t \to \infty} F_n(t) = \lim_{t \to \infty} F(t) = l \in \R$. Then $F_n \to F$ uniformly.
\end{prop}
\begin{proof}
	Without loss of generality we can suppose that $l=1$. Fix $\epsilon>0$. Then there exists $T>0$ such that for any $t \ge T$ it is true that $1-F(t)<\epsilon/2$. Note that $F(t)$ is continuous in $[0,T]$ and thus uniformly continuous. It follows that there exists $\delta>0$ such that for any choice of $t_1$ and $t_2$ in $[0,T]$ it holds
	\begin{equation}
		|t_2-t_1|< \delta \, \Longrightarrow \,|F(t_2)-F(t_1)|<\epsilon/2.
		\label{3118b}
	\end{equation}
	Take now a finite partition $\Pi$ of $[0,T]$ with diam$(\Pi)<\delta$. Since $F_n(t) \to F(t)$ for any $t$ we have by pointwise convergence that there exists $\nu>0$ such that, for any $n>\nu$ one has 
	\begin{equation}
		|F_n(t_i)-F(t_i)|< \epsilon/2 \text{ for all } t_i \in \Pi.
		\label{3118}
	\end{equation}
	Now consider $t>T$. We have that
	\begin{align}
		F_n(t) \leq 1 \leq F(t)+\epsilon/2 < F(t)+\epsilon
	\end{align}
	where we used the fact that since $t>T$ then $1<F(t)+\epsilon/2$. On the other hand, for $n > \nu$
	\begin{align}
		F_n(t) \geq \, &F_n(T) \notag \\
		\geq \, & F(T)-\epsilon /2 & \text{by \eqref{3118} since } T \in \Pi \notag \\
		\geq \, &  F(t)-(F(t)- F(T))-\epsilon /2 \notag \\ \geq \, &   F(t)-\epsilon & \text{since } F(t)-F(T)\leq 1-F(T)<\epsilon/2.
	\end{align}
	For $t \in [0,T]$ instead, we have that there exists $t_i, t_{i+1} \in \Pi$ such that $t \in [t_i, t_{i+1}]$. It follows that, for $n > \nu$,
	\begin{align}
		F_n(t) \leq &F_n(t_{i+1}) \notag \\
		\leq \, & F(t_{i+1}) + \epsilon/2  &\text{by \eqref{3118} since } t_{i+1} \in \Pi \notag \\
		\leq \, & F(t)+ (F(t_{i+1})-F(t))+\epsilon/2 \notag \\
		\leq \, &F(t) +\epsilon & \text{by \eqref{3118b} since diam}(\Pi)<\delta.
	\end{align}
	With the same argument it is possible to see that, for any $n > \nu$
	\begin{equation}
		F_n(t) \geq F(t) -\epsilon.
	\end{equation}
\end{proof}
\subsection{Proof of Proposition \ref{prop:kBM}}\label{app:kBM}
By \cite[Chapter 2, Problem 8.6]{karatzas}, we have that
\begin{align}\label{eq:premsd}
	\begin{split}
	\E_0[|B^\dagger(t)|^2\mathbf{1}_{T_c>t}]&=\int_{-\infty}^{c}x^2p(t,x;0)\, dx-\int_{-\infty}^{c}x^2p(t,x;2c)\, dx\\
	&=\int_{-\infty}^{c}x^2p(t,x;0)\, dx-\int_{-\infty}^{-c}x^2p(t,x;0)\, dx\\
	&\qquad -4c\int_{-\infty}^{-c}xp(t,x;0)\, dx-4c^2\int_{-\infty}^{-c}p(t,x;0)\, dx\\
	&=2\int_{0}^{c}x^2p(t,x;0)\, dx-4c\int_{-\infty}^{-c}xp(t,x;0)\, dx-4c^2\int_{-\infty}^{-c}p(t,x;0)\, dx\\
	&=2\int_{0}^{c}x^2p(t,x;0)\, dx+4c\int_{c}^{+\infty}xp(t,x;0)\, dx-4c^2\int_{c}^{+\infty}p(t,x;0)\, dx\\
	&=2\int_{0}^{c}x^2p(t,x;0)\, dx+4c\left(\int_{0}^{+\infty}xp(t,x;0)\, dx-\int_{0}^{c}xp(t,x;0)\, dx\right)\\
	&\qquad -4c^2\left(\frac{1}{2}-\int_{0}^{c}p(t,x;0)\, dx\right)	
	\end{split}
\end{align}
Now, by \cite[Formula 3.321.5]{gradshteyn2014table},
\begin{equation}\label{eq:premsd1}
	\int_{0}^{c}x^2p(t,x;0)\, dx=\btrev{\frac{1}{2}}t\, {\sf erf}\left(\frac{c}{\sqrt{2t}}\right)-c\sqrt{\frac{t}{2\pi}}e^{-\frac{c^2}{2t}}.
\end{equation}
Next, notice that	
\begin{equation}\label{eq:premsd2}
	\int_{0}^{+\infty}xp(t,x;0)\, dx=\sqrt{\frac{t}{2\pi}} \qquad \int_{0}^{c}xp(t,x;0)\, dx=\sqrt{\frac{t}{2\pi}}\left(1-e^{-\frac{c^2}{2t}}\right)
\end{equation}
and furthermore
\begin{equation}\label{eq:premsd3}
	\int_{0}^{c}p(t,x;0)\, dx=\frac{1}{2}{\sf erf}\left(\frac{c}{\sqrt{2t}}\right).
\end{equation}
Plugging \eqref{eq:premsd1}, \eqref{eq:premsd2} and \eqref{eq:premsd3} into \eqref{eq:premsd} we get \eqref{eq:msdbM}.
\qed

\section*{Acknowledgements}
G. Ascione and B. Toaldo have been partially supported by the MIUR PRIN 2017 project ``Stochastic Models for Complex Systems'', no. 2017JFFHSH.

G. Ascione is supported by the PRIN 2022XZSAFN Project “Anomalous Phenomena on Regular and Irregular Domains: Approximating Complexity for the Applied Sciences” CUP E53D23005970006 and by the GNAMPA-INdAM group.

B. Toaldo acknowledges financial support under the National Recovery and Resilience Plan (NRRP), Mission 4, Component 2, Investment 1.1, Call for tender No. 104 published on 2.2.2022 by the Italian Ministry of University and Research (MUR), funded by the European Union – NextGenerationEU– Project Title “Non–Markovian Dynamics and Non-local Equations” – 202277N5H9 - CUP: D53D23005670006 - Grant Assignment Decree No. 973 adopted on June 30, 2023, by the Italian Ministry of Ministry of University and Research (MUR).

B. Toaldo would like to thank the Isaac Newton Institute for Mathematical Sciences, Cambridge, for support and hospitality during the programme Stochastic systems for anomalous diffusion, where work on this paper was undertaken. This work was supported by EPSRC grant EP/Z000580/1.

G. Ascione and B. Toaldo have been partially supported by GNAMPA (INDAM).

\textcolor{blue}{G. Ascione and B. Toaldo would like to thank the anonymous referees, whose comment really helped improving the paper.}

\vspace{1cm}

\end{document}